\documentclass[a4paper,10pt]{article}

\usepackage[english]{babel}
\usepackage[utf8]{inputenc}
\usepackage[T1]{fontenc}
\usepackage{amsmath,mathtools}
\usepackage{amsfonts,amssymb}
\usepackage{xspace,enumitem}
\usepackage{amsthm,stmaryrd}
\usepackage[margin=1in,headheight=13.6pt]{geometry}
\RequirePackage[scaled=.92]{helvet}

\newenvironment{equivenum}%
	{\begin{enumerate}[label=(\roman{*})]}{\end{enumerate}}
\newenvironment{thmenum}%
	{\begin{enumerate}[label={\rm (\roman{*})}]}{\end{enumerate}}

\theoremstyle{plain}
  \newtheorem{thm}{Theorem}
   \newtheorem*{thm*}{Theorem}
  \newtheorem{defn}[thm]{Definition}
  \newtheorem{prop}[thm]{Proposition}
  \newtheorem{lem}[thm]{Lemma}
  \newtheorem{cor}[thm]{Corollary}
 \newtheorem{rem}[thm]{Remark}
  \newtheorem*{rem*}{Remark}
  
 \theoremstyle{definition}

  \newtheorem*{defn*}{Definition}

\numberwithin{equation}{section}
\numberwithin{thm}{section}

\newcommand{\Op}{\ensuremath{P}\xspace}
\renewcommand{\Im}{\mathrm{Im}}
\renewcommand{\Re}{\mathrm{Re}}

\newcommand{\SpecDom}{\ensuremath{{\C}^+}\xspace}

\newcommand{\Fc}{\ensuremath{\chi_1}}
\newcommand{\refl}{\ensuremath{\mathcal{S}_{rec}}}

\newcommand{\Gap}[1][n]{\ensuremath{G_{#1}}}

\newcommand{\Grad}[2]{\ensuremath{\left\llbracket #1  \right\rrbracket_{q,#2} }\xspace}

\newcommand{\Hp}[1][1]{\ensuremath{H_c^{#1}}\xspace}
\newcommand{\Hr}[1][1]{\ensuremath{H_r^{#1}}\xspace}

\newcommand{\set}[2]{\ensuremath{\{\; #1 \;\colon \; #2\;\}}}

\DeclareMathOperator{\Res}{Res}

\let\existsorig\exists
\renewcommand{\exists}{\ \existsorig\ }
\let\forallorig\forall
\renewcommand{\forall}{\ \forallorig\ }

\newcommand{\pfrac}[2]{\ensuremath{\left(\frac{#1}{#2}\right)}}

\renewcommand{\Re}{\mathfrak{R}}
\renewcommand{\Im}{\mathfrak{I}}

\DeclarePairedDelimiter{\norm}{\lVert}{\rVert}

\newcommand{\ii}{\mathrm{i}}

\DeclareMathOperator{\dist}{dist}

\DeclareMathOperator\trace{tr}%

\newcommand{\C}{\mathbb{C}}%
\newcommand{\N}{\mathbb{N}}%
\newcommand{\R}{\mathbb{R}}%
\newcommand{\Z}{\mathbb{Z}}%
\newcommand{\T}{\mathbb{T}}%

\newcommand{\gm}[1][n]{\ensuremath{\gamma_{#1}}\xspace}
\newcommand{\Gm}[1][n]{{\Gamma_{#1}}}
\newcommand{\dl}{\delta}     \newcommand{\Dl}{\Delta}
\newcommand{\lm}{\lambda}    
\newcommand{\ph}{\varphi}    

\newcommand{\upd}{\mathinner{\mathrm{d}\kern.04em\!}}

\newcommand{\dmu}{\upd \mu}

\newcommand{\dlm}{\upd{\lm}}

\newcommand{\ds}{\upd{s}}
\newcommand{\dt}{\upd{t}}

\newcommand{\dx}{\upd{x}}

\newcommand{\ddh}[1]{\upd #1 [\mathring{v}]}

         \title{On nomalized differentials on spectral curves associated to the sinh-Gordon equation}
         \author{
         Thomas Kappeler\footnote{Supported in part by the Swiss National Science Foundation.} ,
	Yannick Widmer\footnote{Supported in part by the Swiss National Science Foundation.}
         }
         \date{\today}



\begin{document}
\graphicspath{{images/}} 

        \maketitle

     \noindent
{\bf Abstract.} 
The spectral curve associated with the sinh-Gordon equation on the torus 
is defined in terms of the spectrum of the Lax operator appearing in the
Lax pair formulation of the equation. If the spectrum is simple, it is
an open Riemann surface of infinite genus. In this paper we construct
normalized differentials and derive estimates for the location of their zeroes.

\medskip


\medskip



\tableofcontents

\section{Introduction}\label{introduction}
We consider the complex sine Gordon equation,
 \begin{equation}\label{eq:sineGordonComplex}
 u_{tt}-u_{xx} = -\sinh u , \quad x\in\T =\R/\Z, \; t\in\R \, ,
 \end{equation}
 where $u$ is assumed to be complex valued. In case $u$ is real valued, \eqref{eq:sineGordonComplex}  is referred to as
 the sinh Gordon equation and in case $u$ is purely imaginary, as real sine Gordon equation.
 Equation  \eqref{eq:sineGordonComplex}  is  a nonlinear perturbation of the (complex) Klein-Gordon equation $u_{tt} -u_{xx} = mu$ (with $m=-1$).
 It has wide ranging applications in geometry and quantum mechanics and has been extensively studied, although most of the work has been done
 for its version in light cone coordinates and hence does not apply to the periodic in space setup of \eqref{eq:sineGordonComplex}.
 According to \cite{FaddeevTakhtajanZakharov1974Sine} (cf also \cite{mckean1981sineLight}), it admits a Lax pair. To describe it note that
\eqref{eq:sineGordonComplex} can be written as a system for $(u_1, u_2) := (u, u_t)$
\begin{equation} \label{eq:HamiltonianEquationpv}
\begin{pmatrix} u_{1t}\\u_{2t}\end{pmatrix} = \begin{pmatrix} u_2\\ u_{1xx} - \sinh u_1
\end{pmatrix} \, .
\end{equation}
In order to work with function spaces consisting of pairs of functions of equal regularity  we introduce
\begin{equation}\label{eq:kap1.4bis}
 (q, p) : = (u_1, - \Op^{-1} u_2) \in \Hp
 \end{equation}
where for any $s$, $\Hp[s]$ denotes the Sobolev space
$H^s(\mathbb T, \mathbb C) \times H^{s}(\mathbb T, \mathbb C)$ $(\equiv H^s(\mathbb T, \mathbb C^2) )$ and
$\Op$ the Fourier multiplier operator $\Op:=\sqrt{1-\partial_x^2}$.
When expressed in these coordinates, equation \eqref{eq:HamiltonianEquationpv} becomes
\[
\begin{pmatrix} q_t\\ p_t \end{pmatrix} = \begin{pmatrix} -\Op p \\ \Op q + \Op^{-1}(\sinh(q)-q)\end{pmatrix}.
\]
For any $v = (q,p)\in\Hp$, define the following differential operators:
\begin{equation}\label{eq:Qoperator}
Q(v) = Q_1\partial_x + Q_0(v), \quad K(v) = K_1\partial_x + K_0(v)
\end{equation}
where the coefficients $Q_1,Q_0,K_1,K_0$ are the $4\times 4$ matrices given by
\begin{align*}
\begin{split}
Q_1 =\begin{pmatrix} -J&\\&\end{pmatrix}, \quad Q_0(v)&=\begin{pmatrix} A(v)&B(v)\\B(v)&\end{pmatrix},\\
K_1 = \begin{pmatrix}-I&\\&I\end{pmatrix}, \quad K_0(v)&=\begin{pmatrix} & -2JB(v)\\-2B(v)J&                      \end{pmatrix},\end{split}
 \end{align*}
with $I, J, R, Z, A(v), B(v)$ denoting the $2\times2$ matrices
\begin{align}
I = \begin{pmatrix}1\\&1\end{pmatrix},
\quad J = \begin{pmatrix}&1\\-1\end{pmatrix},
\quad R=\begin{pmatrix}\ii& \\ &-\ii\end{pmatrix},
\quad Z=\begin{pmatrix}&1 \\1 &\end{pmatrix}
\end{align}
\begin{align}
A(v):=-\frac14 (\Op p+ q_x)Z,\quad
B(v) := \frac14\begin{pmatrix} \exp(-q/2)&\\&\exp(q/2)\end{pmatrix}. \label{eq:QoperatorAandB}
\end{align}
Here and in the sequel, we suppress matrix coefficients which vanish, so e.g. \begin{small}$\begin{pmatrix}&1\\-1\end{pmatrix}$\end{small} stands for \begin{small}$\begin{pmatrix} 0 & 1 \\-1 & 0\end{pmatrix}$\end{small}.
One verifies that $t\mapsto v(t)$ is a solution of \eqref{eq:sineGordonComplex} iff
$t\mapsto ( Q(v(t)), K(v(t)) )$ satisfies
\begin{equation} \label{eq:laxPairCondition}
Q_t = [K, Q].
\end{equation}
The pair of operators $Q, K$ is referred to as Lax pair for the sine-Gordon equation and $Q$ as the corresponding Lax operator.
Note that \eqref{eq:laxPairCondition} leads to a family of first integrals of \eqref{eq:sineGordonComplex}: expressed in a somewhat informal way
(i.e., without addressing issues of regularity) it follows from \eqref{eq:laxPairCondition} that for any solution $t\mapsto v(t)$ of \eqref{eq:sineGordonComplex},
the periodic spectrum $\mathrm{spec}_{per} Q(v(t))$ of the operator $Q(v(t))$ is independent of $t$.
Here for any $v \in \Hp,$ the periodic spectrum of $Q(v)$ is the spectrum of the operator $Q(v)$ on  $L^2(\mathbb R / (2 \mathbb Z),\C^2)$,
with domain given by the Sobolev space $H^1(\mathbb R / (2 \mathbb Z),\C^2)$.
Since for any $v \in \Hp,$
$\mathrm{spec}_{per} Q(v)$ is discrete this is saying that the periodic eigenvalues of $Q(v)$ are first integrals for the sine-Gordon equation.
We remark that the periodic spectrum of $Q(v)$ is the union of the spectrum of $Q(v)$, when considered as an operator on 
$L^2([0, 1],\C^2)$ with periodic boundary conditions ($F(1) = F(0)$) and the one of $Q(v)$,  when considered as an operator on 
$L^2([0, 1],\C^2)$ with antiperiodic  boundary conditions ($F(1)= - F(0)$).

The aim of this paper is to construct for $v$ near \Hr normalized differentials on the spectral curve $\Sigma_v$, associated with the periodic spectrum of $Q(v)$,
and study their dependence on $v$. 
Besides being of interest in their own right, our motivation to study such differentials stems from the fact that they are a key ingredient for the
construction of normal form coordinates of the sine-Gordon equation. 


To define the spectral curve and state our main results, we first need to review spectral properties of the operator $Q(v)$
-- see \cite{LOperator2018} for detailed account.
Introduce the domains $D_0:=\set{z\in\C}{|z-\frac14|< \frac1{4\pi} }$ and
 \[
 D_n := \set{\lm\in\C}{|\lm-n\pi|<\pi/3}, \quad
 D_{-n}:=\set{\lm\in\C}{\frac{1}{16\lm}\in D_{n}} \, , \qquad \forall n\geq 1\, .
 \]
 Furthermore, let $B_0:=\set{\lm\in \C}{|\lm|\leq \pi/2}$,
 \begin{equation}\label{definition Bn}
 B_n := \set{\lm\in\C}{|\lm|< n\pi+\pi/2}, \quad B_{-n} := \set{\lm\in\C}{ |\lm| \leq \frac{1}{16(n\pi+\pi/2)} },  \qquad \forall  n \ge 1,
\end{equation}
 and denote by $A_n$, $n \ge 1$, the open annulus  $A_n := B_n\setminus B_{-n}.$
\noindent  By the Counting Lemmas (cf. \cite[Lemma 3.4 and Lemma 3.11]{LOperator2018}), the following holds:
any potential in $\Hp$ admits a neighborhood $W$ in $\Hp$ and an integer $N \ge 1$ so that for any $v \in W$
and any $n > N$,
the operator $Q(v)$ has precisely two periodic eigenvalues
in each  of the domains  $D_n$, $-D_n$, $D_{-n}$, and $ - D_{-n}$
and exactly $4+8N$ periodic
eigenvalues in the annulus $A_{N}$, counted with their multiplicities.
There are no further periodic
eigenvalues.
We denote the periodic eigenvalues in $D_n$, $|n| > N$, by $\lambda_n^-, \lambda_n^+$, listed as $\lambda_n^-  \preceq \lambda_n^+ $ , 
where $\preceq$ is an order on $\C^+$ defined as follows:  for any elements $a, b \in \C^+$
one has $a  \preceq  b$ if
$$
\big[ |a| < |b| \big] \quad \text{ or } \qquad \big[  |a| = |b| \,\,\, {\text{and}} \,\,\, \Im a \le \Im b \big] \, .
$$
By symmetry (cf. \cite[Theorem 3.9]{LOperator2018}), for any $|n| > N$, the two periodic eigenvalues in $-D_n$
are given by  $- \lm^+_n$ and $- \lm^-_n$.
As a consequence, we only need to consider the periodic
eigenvalues in
$$
\SpecDom :=\set{\lm\in\C}{\Re\lm >0} \cup \set{\lm\in\C}{\Im \lm>0 \text{ and }\Re\lm =0} \, .
$$
Note that besides $\lambda_n^- , \lambda_n^+$, $| n | > N$, there are $2+4N$ periodic eigenvalues of $Q(v)$ in  $\SpecDom$.
When listed according to the order $\preceq$, the periodic eigenvalues of  $Q(v)$,
contained in $\SpecDom$,  form a bi-infinite sequence,
 \begin{equation}\label{eq:listingPeriodic}
  0 \prec \cdots  \preceq \lm_{-1}^-  \preceq \lm_{-1}^+  \preceq \lm_0^-  \preceq \lm_0^+  \preceq \lm_1^- \preceq \lm_1^+ \preceq \cdots \, .
 \end{equation}
Note that for $v\in H^s_r = H^s(\T,\R) \times H^s(\T,\R)$ $(\equiv H^s(\T,\R^2))$,  the operator $Q(v)$ is self-adjoint and hence the periodic spectrum real. The  above sequence of
eigenvalues satisfies
\begin{equation}\label{real ev}
  0 < \cdots < \lm_{-1}^- \leq \lm_{-1}^+  <  \lm_0^- \leq \lm_0^+ <  \lm_1^-\leq \lm_1^+  <  \cdots .
\end{equation}
By a slight abuse of terminology, for any $v \in \Hp$ and $n \in \Z$, we refer to the complex interval
$$
G_n := [\lambda_n^- , \lambda_n^+] = \{ (1 - t)\lambda_n^- + t \lambda^+_n  \ : \ 0 \le t \le 1 \}
$$
as the nth gap and to the difference $\gamma_n := \lambda^+_n - \lambda^-_n$ as the nth gap length.

Let $v$ be an arbitrary element in $H^1_c$ and $F$ a function in $H^1_{loc}(\R,\C^4)$.
According to  \cite[Section 2]{LOperator2018}  $ Q(v) F = 0$ if and only if $F=0$ and that  for any given $\lm\in \C^* = \C\setminus\{0\}$,  $F$ is a solution of
 \begin{equation} \label{eq:Qeigenvalues}
 QF=\lm F
 \end{equation}
 if and only if $F = (f, \lm^{-1}B f)$ and $f$ satisfies the following first order ODE
\begin{equation} \label{eq:reducedEquation}
-J\partial_x f+(A+B^{2}/\lm)f=\lm f.
 \end{equation}
 (Here and in the sequel we often write a column vector such as $F$ or $f$ as a row vector
and write $\partial_x f$ for $f_x$.)
Let $M = M(x,\lm,v)\in \C^{2\times2}$ be the fundamental solution
for equation \eqref{eq:reducedEquation}, meaning that 
 \begin{equation} \label{eq:dxM}
 \partial_x M =J\left(\lm - A - B^2/\lm \right) M,  \quad M(0,\lm,v) = I \, .
 \end{equation}
 By \cite[Lemma 2.14]{LOperator2018},  the discriminant of $Q(v)$,
 \[
   \Dl(\lm,v) : = \frac12 \trace M(1,\lm,v)\, ,
 \]
is analytic in $\lambda \in \C^*$. As discussed in Section \ref{setup}, the zeroes of $\Dl^2(\lm,v) -1$ are in one-to-one correspondance with the periodic eigenvalues of $Q(v)$
(with multiplicities). We refer to $\chi_p(\cdot, v):=  \Dl^2(\cdot, v) - 1$ as the characteristic function of $Q(v)$.
The object of study of this paper is the spectral curve $ \Sigma_v$, associated to the characteristic function $\chi_p(\cdot, v)$.
It is defined as
 \begin{equation}\label{spectral curve}
 \Sigma_v := \{(\lambda , y) \in  \C^*\times \C \ : \  y^2 =   \chi_p(\lambda, v)  \} \, .
 \end{equation}
In case all periodic eigenvalues of $Q(v)$ are simple, $\Sigma_v$ is an (open) hyperelliptic Riemann surface of infinite genus.

Our results on normalized differentials on $\Sigma_v$ are formulated for potentials in a sufficiently small neighborhood of \Hr in \Hp. 
To define the latter, we introduce the notion of isolating neighborhoods: Assume that $v_0$ is an arbitrary element in \Hr. It is shown in Section \ref{sec:realAmostReal}
that there exists a ball $W_{v_0}$ in \Hp, centered at $v_0$, and a sequence $(U_n)_{n\in \Z}$ of pairwise disjoint open neighborhoods in $\C^+$ 
so that for any $v$ in $W_{v_0}$,  the following holds:
\begin{itemize}
  \item[] (IN-1)  $\Gap = [\lambda_n^-, \lambda_n^+] \subset U_n\subset \C^+$,   $\forall n\in \Z$.
  \item[] (IN-2) For any $n\geq 0$, $U_n$ is a disc centered on the real axis so that for some constant $c\geq 1$
  \[
  c^{-1}|m-n|\leq dist(U_m,U_n)\leq c|m-n| \quad \forall m,n\geq 0, \, m\not=n.
  \]
  \item[] (IN-3) The sets  $ \set{ \frac{1}{16\lm}}{\lm\in U_{-n}}$, $n\geq 0$, satisfy (IN-2) with the same constant c.
  \item[] (IN-4) For $|n|$ sufficiently large, $U_n=D_n$.
\end{itemize}
The sequence  $(U_n)_{n\in \Z}$ is referred to as a sequence of isolating neighborhoods for $v$.
We denote by $\hat W$ the connected neighborhood of $\Hr$ in \Hp, given by
$$
\hat W := \bigcup_{v \in H^1_r} W_v \, .
$$
Denote by $\mathcal I_{rec}$ the involution
$$
\mathcal I_{rec} : \Hp \to \Hp\, , \ (q, p) \mapsto (-q, p) \, .
$$
By choosing the balls $W_v$ so that  for any $v \in \Hr$, $W_{\mathcal I_{rec}(v)} = \mathcal I_{rec}(W_v)$, it follows that
$\hat W$ is left invariant by $\mathcal I_{rec}$. 
Furthermore, to simplify the statement of our results, we introduce the following alternative notation for the periodic eigenvalues.
For $v\in\hat W$ define
 \begin{equation}\label{eq:kap6.3.1}
 \lm_{1,k}^+ := \left\{ \begin{array}{ll}
                        \lm_k^+ & k\geq 0 \\ -\lm_{-k}^- & k\leq -1
                        \end{array}\right.
 \quad \lm_{1,k}^- := \left\{ \begin{array}{ll}
                        \lm_k^- & k\geq 0 \\ -\lm_{-k}^+ & k\leq -1
                        \end{array}\right. 
 \end{equation}
 \begin{equation}\label{eq:kap6.3.2}
 \lm_{2,k}^+ := \left\{ \begin{array}{ll}
                        \frac{1}{16\lm_{-k}^-} & k\geq 0 \\ -\frac{1}{16 \lm_{k}^+} & k\leq -1
                        \end{array}\right.
 \quad \lm_{2,k}^- := \left\{ \begin{array}{ll}
                        \frac{1}{16\,\lm_{-k}^+} & k\geq 0 \\ -\frac{1}{16\,\lm_{k}^-} & k\leq -1
                        \end{array}\right.
 \end{equation}
 Correspondingly, we define
 \begin{equation}\label{G_1m}
 G_{1,m} : = [\lm_{m}^-,\lm_{m}^+]\ ,\quad \forall m \geq0, \qquad \qquad \,\, \, G_{1,-m}:= -G_{1,m} \ , \quad \forall m \geq1 \, ,
 \end{equation}
 \begin{equation}\label{G_2m}
 G_{2,m} := [-\lm_{-m}^+,-\lm_{-m}^-]\ , \,\, \forall m \geq 0 , \qquad \,\, G_{2,-m} := -G_{2,m} \ , \quad \forall m\geq1\, ,
 \end{equation}
and
 \begin{equation}\label{U_1m}
 U_{1,m} := U_m \ , \quad  \forall m\geq0\ , \qquad \quad   U_{1,-m} := -U_m \; \, \forall m\geq1, \quad
 \end{equation}
 \begin{equation}\label{U_2m}
 U_{2,m} := -U_{-m} \ , \,\, \forall m \geq0 \ , \qquad \,\, U_{2,-m} := U_{-m} \ , \quad \forall  m\geq1
 \end{equation}
 where $U_m$, $m\in\Z$, 
 are isolating neighborhoods for $v\in\hat W$. 
For any $m \in \Z$, choose within the neighborhood $ U_{m}$  a 
counter clockwise oriented contour $\Gamma_m$  around $G_{m}$ and define
\begin{equation}\label{Gamma_1m}
 \Gamma_{1,m} := \Gamma_m \ , \quad \forall m\geq0\ , \qquad \quad \,\,\,\, \Gamma_{1,-m} := (\Gamma_m)^- \; \, \forall m\geq1, 
 \end{equation}
 \begin{equation}\label{Gamma_2m}
 \Gamma_{2,m} := (\Gamma_{-m})^-  , \,\,  \forall m \geq0 \ , \qquad \,\,\, \Gamma_{2,-m} := \Gamma_{-m} \; \, \forall m\geq1 \, , \,\,\,
 \end{equation}
 where for any $m \in \Z$, $(\Gamma_m)^-$ denotes the closed curve $\{ - \lambda \ : \ \lambda \in \Gamma_m  \}$ with counterclockwise orientation.
 
\begin{figure}
  \centering
    \includegraphics[width=\textwidth]{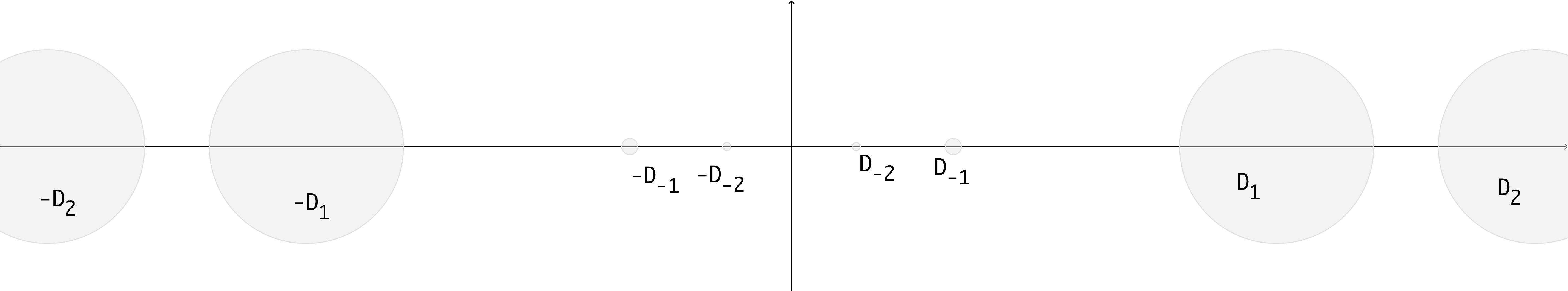}
         \caption{Illustration of the domains $D_n$, $D_{-n}$, $-D_{-n}$, $-D_n$ for $n=1,2$}
 \end{figure}
 
Finally, we need to choose an analytic branch of the square root of $\chi_p(\lambda) = \Dl^2(\lambda)-1$. 
In Section \ref{sec:roots} (cf. \eqref{eq:kap6.3.40})
we describe such a branch on $\C^* \setminus \bigcup_{m\in\Z} G_{1,m} \cup G_{2,m}$ 
and denote it by $\sqrt[c]{\chi_p(\lambda)}$.
It is locally analytic in $v\in \hat W$ away from $\bigcup_{m\in\Z} G_{1,m} \cup G_{2,m}$.
For $v$ in a neighborhood of $\Hr$, contained in $\hat W$, we construct normalized differentials on the spectral curve $\Sigma_v$,
$$
 \frac{\psi_n(\lm)}{\sqrt{\chi_p(\lm)}}
$$
where $\psi_n(\lm)$, $n\geq0$, are functions on $\C^*$ of the form
\begin{equation} \label{eq:kap8.2.9}
\psi_n(\lm) := -\frac{1}{\pi_n} C_n \psi_{n,1}(\lm) \cdot \psi_{n,2}(\lm) 
\end{equation}
with $\pi_n := n\pi$ for any $n \ne 0$ and $\pi_0 := 1$. Here
$ \psi_{n,1}(\lm)$, $ \psi_{n,2}(\lm)$ are analytic functions on $\C^*$, given by infinite products,
\begin{equation}\label{eq:kap8.2.11}
 \psi_{n,1}(\lm) = \prod_{\substack{k\in\Z \\ k\not=n}} \frac{\sigma_{1,k}^n-\lm}{\pi_k} , \qquad \psi_{n,2}(\lm) = \prod_{k\in\Z} \frac{\sigma_{2,k}^n + \frac{1}{16\lm}}{\pi_k} \, ,
\end{equation}
where $\sigma_{1,k}^n$, $\sigma_{1,k}^n$ satisfy 
\begin{equation}\label{eq:kap8.2.12}
\sigma_{1,k}^n\in U_{1,k} , \quad  k\not=n\, ,  \qquad (-16\sigma_{2,k}^n)^{-1} \in U_{2,k}, \quad \forall k \in \Z  \, ,
\end{equation}
as well as the asymptotics  $\sigma_{j, k}^n = k\pi + \ell_k^2$, 
and
\begin{equation}\label{eq:kap8.2.10}
C_n \equiv C_{\psi_{n,2}} : = 1/\psi_{n,2}(\infty)\, , \qquad   \psi_{n,2}(\infty):= \prod_{k\in\Z} \frac{\sigma_{2,k}^n}{\pi_k} \, .
\end{equation}
For what follows it is convenient to introduce the notation $\sigma_{1,n}^n := (\lm_{1,n}^+ + \lm_{1,n}^-)/2$.
The main result of this paper is the following one.
\begin{thm}\label{thm:kap8.2.12} 
  There exists a complex neighborhood $W$ of $\Hr$, contained in $\hat W$ and satisfying $\mathcal I_{rec} (W) =W$, with the following property:
   for any $v \in W$ there exist functions $\psi_n(\lm) \equiv \psi_n(\lm,v)$, $n \ge 0$,
  of the form \eqref{eq:kap8.2.9}--\eqref{eq:kap8.2.12} so that for any $m\in\Z$,
  \begin{equation}\label{eq:kap8.2.13}
  \frac{1}{2\pi} \int_{\Gm[1,m]} \frac{\psi_n(\lm)}{\sqrt[c]{\chi_p(\lm)}}\dlm = \dl_{nm},
  \end{equation}
  \begin{equation}\label{eq:kap8.2.14}
  \frac{1}{2\pi} \int_{\Gm[2,m]} \frac{\psi_n(\lm)}{\sqrt[c]{\chi_p(\lm)}}\dlm = 0.
  \end{equation}
  The possibly complex roots $\sigma_{1,k}^n$ ($k\not=n$) and $(-16\sigma_{2,k}^n)^{-1}$ ($k\in\Z$) of $\psi_n(\lm)$ depend analytically on $v$.
  Furthermore,  $\sigma_{j,k}^n$ satisfy the estimates
  \begin{equation}\label{eq:kap8.2.15}
  \sigma_{j,k}^n - \tau_{j,k} = \gm[j,k]^2\ell_k^2 \quad \forall k\in\Z,\;j=1,2
  \end{equation}
  where
  \begin{equation}\label{eq:kap8.2.16}
  \tau_{j,k} := (\lm_{j,k}^+ + \lm_{j,k}^-)/2
  \end{equation}
  and if $v$ is real valued, then $\lm_{j,k}^- \leq \sigma_{j,k}^n \leq \lm_{j,k}^+$ for any $k\in\Z$ and $j=1,2$.
  \end{thm}

  \begin{rem}\label{rem:kap8.2.12bis}
    One can show that the functions $\psi_n(\lm)$ of the form \eqref{eq:kap8.2.9}-\eqref{eq:kap8.2.12} are uniquely determined by \eqref{eq:kap8.2.13}-\eqref{eq:kap8.2.14}. For a precise statement see Proposition \ref{prop:kap8.2.20}.
    \end{rem}
Theorem \ref{thm:kap8.2.12} can be used to complement the  family of differentials $ \frac{\psi_n(\lm)}{\sqrt{\chi_p(\lm)}}$, $n \ge 0$,
by the family $ \frac{\psi_{-n}(\lm)}{\sqrt{\chi_p(\lm)}}$, $n \ge 1$, where for any $v=(q,p) \in W$ and $\lm \in \C^*$, 
 \begin{equation}\label{def psi_-n}
     \psi_{-n}(\lm,q,p) := \psi_n(\frac{1}{16\lm},-q,p)\frac{1}{16\lm^2} \, .
 \end{equation}

    \begin{cor}\label{cor:kap8.2.12}
      For any $n\geq1$ and $v=(q,p) \in W$, $\psi_{-n}(\lm,q,p)$  satisfies
    \[
     \int_{\Gm[1,m]} \frac{\psi_{-n}(\lm,v)}{\sqrt[c]{\chi_p(\lm,v)}} \dlm = 0,
    \qquad  \int_{\Gm[2,m]} \frac{\psi_{-n}(\lm,v)}{\sqrt[c]{\chi_p(\lm,v)}} \dlm = 2\pi \dl_{-n,m} \, , \qquad \forall m \in \Z  \, .
    \]
    The possibly complex roots $\sigma_{1,k}^{-n}(v)$ ($k\in\Z$) and $(-16\sigma_{2,k}^{-n}(v))^{-1}$ ($k\not=-n$) of $\psi_{-n}(\cdot, v)$ depend analytically on $v$ and satisfy
    \[
    \sigma_{j,k}^{-n}(v) - \tau_{j,k}(v) = \gm[j,k](v)^2 \ell_k^2\quad \forall (j,k)\not=(2,-n).
    \]
    Furthermore, if $v$ is real valued, then
    \[
    \lm_{j,k}^-(v) \leq \sigma_{j,k}^{-n} (v) \leq \lm_{j,k}^+(v) \quad \forall (j,k)\not= (2,-n).
    \]
    \end{cor}
    
   \bigskip 
    
\noindent
{\em Comments:} The proof of Theorem \ref{thm:kap8.2.12} relies on a perturbation argument.
To describe it, let us first consider the case where $v$ is in \Hr. In such a case, $Q(v)$ is
self-adjoint and hence its periodic spectrum consists of real eigenvalues. The ones in $\C^+$
form a bi-infinite sequence of the form 
\[
  0 < \cdots \leq \lm_{-1}^- \leq \lm_{-1}^+  <  \lm_0^- \leq \lm_0^+  <  \lm_1^-\leq \lm_1^+\leq \cdots .
\]
It is not hard to show that for any $\psi_n(\lambda)$ of the form \eqref{eq:kap8.2.9}--\eqref{eq:kap8.2.12}, 
the equations \eqref{eq:kap8.2.13} - \eqref{eq:kap8.2.14} imply that 
$$
 \lm_{1, k}^- \leq  \sigma_{1,k}^n  \leq \lm_{1, k}^+\, ,\quad
 \lm_{2, k}^- \leq  \sigma_{2,k}^n  \leq \lm_{2, k}^+\, , \qquad \forall k \in \Z\, .
$$
In fact, this observation has motivated us to make the ansatz \eqref{eq:kap8.2.9}--\eqref{eq:kap8.2.12}
for $\psi_n$ in the first place. To prove Theorem \ref{thm:kap8.2.12}, we reformulate  in a first step the statement 
about the existence of the functions $\psi_n(\lambda)$, $n \ge 0$, as a functional
equation for the sequences $( \sigma_{1,k}^n)_{n \in \Z}$, $( \sigma_{2,k}^n)_{n \in \Z}$,
which then is solved by the means of the implicit function theorem.

The method of proof described above has been initiated in \cite{kappeler2003kdv} to prove 
results, similar to the ones of Theorem \ref{thm:kap8.2.12}, for spectral curves associated to the KdV equation.
Subsequently, it has been applied in \cite{nlsbook} to obtain such results for the spectral curves
associated to the defocusing NLS equation.  The current work is an advancement of this method
to the more complicated spectral curves at hand. We remark that a key prerequisite is that for 
any real valued potential, the zeroes of $\psi_n(\lambda)$ are confined within the gaps.
Such a confinement no longer holds in the case of the spectral curves associated to the focusing NLS equation,
since in this case  the corresponding Lax operator is not self-adjoint. To handle this case,  
a quite different approach was developed in  \cite{kappeler2017HyperellipticCurves}
(cf. also \cite{kappeler2011Shubinconference}). 

In the case where $v$ is in \Hr, using different methods, the existence of differentials of the kind described in Theorem \ref{thm:kap8.2.12}
can also be derived from \cite[Sections 26-28]{mckean1981sineLight}. Note however that in order to construct 
analytic normal form coordinates for the sinh-Gordon equation, one needs that these differentials are defined
on a complex neighborhood of \Hr in \Hp and that on this neighborhhod, their zeroes are analytic and satisfy the estimates stated
in Theorem \ref{thm:kap8.2.12}. The same comments apply to work about holomorphic differentials 
on Riemann surfaces of infinite genus such as \cite{ahlfors1960} or \cite{feldman2003}
and references therein.

\medskip

\noindent
{\em Organisation:} In Section \ref{setup}, we introduce additional notation and review
results from \cite{LOperator2018}, needed throughout the paper. In particular, 
we discuss the spectrum of the operator $Q(v),$
when considered with Dirichlet boundary conditions, as well as the roots of $\partial_\lambda \Delta(\lambda)$.
Since results for Dirichlet eigenvalues of $Q(v)$, corresponding to the ones  derived in Section \ref{gradients} - Section \ref{sec:roots}
for the periodic eigenvalues,
can be obtained in a  similar way,  we included these results in the latter sections. We remark that the results
for Dirichlet eigenvalues are needed for the construction of the normal form coordinates for the sinh-Gordon equation
and hence will be useful for future work.
In Section \ref{gradients}, we derive formulas for the gradients of various quantities, needed for the
proof of Theorem \ref{thm:kap8.2.12}. 
In Section \ref{sec:realAmostReal} we discuss spectral properties of potentials near \Hr.
In particular we prove that potentials in \Hp, which are sufficiently close to \Hr, admit isolating neighborhoods. 
In Section \ref{sec:prod}, we discuss product representations of various quantities, related to
the infinite products \eqref{eq:kap8.2.12} and in Section \ref{sec:roots} we introduce specific
branches of various square roots, in particular the canonical square root of the characteristic function $\chi_p$.
Finally, in Section \ref{sec:psi} we prove Theorem  \ref{thm:kap8.2.12} and Corollary \ref{cor:kap8.2.12}.
For the convenience of the reader, we have included two appendices which collect some technical results,
needed in the proof of Theorem  \ref{thm:kap8.2.12}.
In Appendix \ref{Liouville} we formulate a version of Liouville's theorem for analytic functions on $\C^+$
and in Appendix \ref{interpolation} we discuss representations of infinite products, obtained by interpolation.


 \section{Setup}\label{setup}
 
In this section, we introduce additional notation and review some results from \cite{LOperator2018}, needed in the sequel.

Recall that by \eqref{eq:dxM}, for any $v=(q,p)\in\Hp$ and $\lambda \in \mathbb C^*$, 
$M = M(x,\lm,v)\in \C^{2\times2}$ denotes the fundamental solution
for equation \eqref{eq:reducedEquation},
$$
\partial_x M =J\left(\lm - A - B^2/\lm \right) M,  \quad M(0,\lm,v) = I \, ,
$$
or, taking the definitions in \eqref{eq:QoperatorAandB} into account,
\[
\partial_x M(x,\lm,v) = J\left(\lm + \frac14(\Op p(x)+q_x(x))\begin{pmatrix}&1\\1\end{pmatrix} - \frac1{16\,\lm}
\begin{pmatrix}e^{-q(x)}\\ &e^{q(x)}\end{pmatrix}\right) M.
\]
For any $\lm\in \C^*$, let $\grave{M}(\lm, v) := M(1, \lm, v)$ and denote the matrix coefficients 
of $M$ and $\grave{M}$ as follows
\[
M = \begin{pmatrix} m_1 & m_2 \\ m_3 & m_4\end{pmatrix},
  \quad
\grave{M} = \begin{pmatrix} \grave{m}_1 & \grave{m}_2 \\ \grave{m}_3 & \grave{m}_4\end{pmatrix} \, .
\]
Furthermore, it is convenient to introduce 
 $$     
         E_{\nu}(x) := 
          \begin{pmatrix}
         \cos(\nu x) & \sin(\nu x )\\ -\sin(\nu x) & \cos(\nu x)
           \end{pmatrix} \, , \quad \nu \in \C \, .
 $$      
We note that
  \begin{equation}\label{eq:Eomega}
 M(x,\lm,0) = E_{\omega(\lm)}(x) \, , \quad  \quad  \omega(\lambda):= \lambda - \frac{1}{16 \lambda}\, ,  \qquad \forall \lambda \in \C^* \, , x \in \R \, .
 \end{equation}            
Denote by $Q_{dir} \equiv Q_{dir}(v)$ the operator $Q = Q_1\partial_x + Q_0$ with domain
\[
 H_{dir} := \set{F = (F_1,F_2,F_3,F_4)\in H^1([0,1],\C^4)}{ F_1(0) = 0, \, F_1(1) =0} 
\]
where $H^1([0,1],\C^4)$ denotes the standard Sobolev space of functions on the interval $[0, 1]$ with values in $\C^4$.
Its spectrum, denoted by $spec_{dir}(Q(v))$, is discrete and coincides with the spectrum of the operator
$-J\partial_x +(A+B^{2}/\lm)$ in \eqref{eq:reducedEquation}
with domain consisting of functions $f = (f_1,f_2)\in H^1([0,1],\C^2)$, satisfying the Dirichlet boundary
conditions $f_1(0)=0,$ $f_1(1) = 0$. The corresponding eigenvalues are referred to as Dirichlet eigenvalues.
Clearly, $\mu\in\C^*$ is a Dirichlet eigenvalue of $-J\partial_x +(A+B^{2}/\lm)$  if there exists $a\in\C^*$ such that
\begin{equation}
	\grave{M}(\mu,v)\begin{pmatrix} 0 \\ 1 \end{pmatrix} = a\begin{pmatrix} 0 \\1 \end{pmatrix}.
\end{equation}
One infers that $spec_{dir}(Q(v))$, referred to as the Dirichlet spectrum of $Q(v)$, is given by  the zeroes  of the function 
$\chi_D(\lm, v) :=\grave{m}_2(\lm, v)$  (cf. \cite[Theorem 3.1]{LOperator2018}).
We refer to $\chi_D(\cdot, v) $ as the characteristic function of $Q_{dir}(v)$.
Furthermore, the multiplicity of any root of $\chi_D(\cdot, v)$ coincides with its algebraic multiplicity as a Dirichlet eigenvalue.
%

The spectrum $spec_{per}(Q(v))$ of the operator $Q(v) = Q_1\partial_x + Q_0$, considered with the domain 
\[
 H_{per} := \set{F\in H_{loc}^1(\R,\C^4)}{ F(x+2) =  F(x)\; \forall x\in\R}\, ,
\]
is referred to as periodic spectrum of $Q(v)$ and can be described in a similar way.
It is discrete and coincides with the periodic spectrum of the operator $-J\partial_x +(A+B^{2}/\lm)$ of \eqref{eq:reducedEquation}. Hence 
a complex number $\lm\in\C^*$ is in $spec_{per}(Q(v))$ 
iff $\grave{M}(\lm) \equiv \grave{M}(\lm,v)$ has an eigenvalue $\pm1$. 
Since in view of \eqref{eq:dxM},  $M$ satisfies the Wronskian identity, one has $\det(\grave{M}(\lambda)) = 1$ and hence
the eigenvalues $\xi_\pm$ of $\grave{M}(\lm)$ satisfy
\begin{equation}
 0=\det (\xi_\pm I-\grave{M}(\lm))=\xi_\pm^2-2\Delta(\lm) \xi_\pm+1\, , \qquad \Delta(\lm) = (\grave{m}_1(\lm)  + \grave{m}_4(\lm))/2,
\end{equation}
and thus are given by
\begin{equation}\label{eq:mEigenvalue}
  \xi_\pm = \Dl(\lm) \pm \sqrt{\Dl^2(\lm)-1}.
\end{equation}
(Note that by \eqref{eq:mEigenvalue}, $\xi_+$ and $\xi_-$ are determined up to the choice of a branch of $\sqrt{\Dl^2(\lm)-1}$.)
Altogether we conclude that $\lambda$ is a periodic eigenvalue of $Q(v)$ if and only if it is a zero 
of the function $ \chi_p(\lm,v) = \Dl^2(\lm,v)-1$ (cf. \cite[Theorem 3.9]{LOperator2018}). 
We refer to $ \chi_p(\cdot, v)$  as the characteristic function of the operator $Q(v)$ with periodic boundary conditions.
Furthermore, the multiplicity of any root of $\chi_p(\cdot, v)$ coincides with its algebraic multiplicity as a periodic eigenvalue. 

For what follows it is convenient to consider the periodic and the Dirichlet spectrum of $Q(v)$ together.
By the Counting Lemmas (cf. \cite[Lemma 3.4, Lemma 3.11]{LOperator2018}), the following holds:
any potential in $\Hp$ admits a neighborhood $W$ in $\Hp$ and an integer $N \ge 1$ so that for any $v \in W$
and any $n > N$,
the operator $Q(v)$ has precisely two periodic eigenvalues
and one Dirichlet eigenvalue
in each  of the domains  $D_n$, $-D_n$, $D_{-n}$, and $ - D_{-n}$
and exactly $4+8N$ periodic
and $2+4N$ Dirichlet
eigenvalues in the annulus $A_{N}$, counted with their multiplicities.
There are no further periodic
or Dirichlet
eigenvalues.
Recall from Section \ref{introduction} that we denote the periodic eigenvalues in $D_n$, $|n| > N$, by $\lambda_n^-, \lambda_n^+$ and list them so that $\lambda_n^-  \preceq \lambda_n^+ $.
By \cite[Lemma 2.14(i),(ii)]{LOperator2018} the periodic spectrum of $Q(v)$ is invariant under the involution $\lm\to -\lm$ and for any periodic eigenvalue $\lambda$ 
of $Q(q,p)$, $\frac{1}{16\lm}$ is a periodic eigenvalue of $Q(-q,p)$.
Furthermore, for any $|n| > N$, the two periodic eigenvalues in $-D_n$
are given by  $- \lm^+_n$ and $- \lm^-_n$ (cf. \cite[Theorem 3.9]{LOperator2018}).
The periodic eigenvalues of $Q(v)$, contained in $\C^+$, can be listed
as a bi-infinite sequence (cf. \eqref{eq:listingPeriodic}),
$$
0 \preceq \dots\preceq \lm_{-1}^- \preceq \lm_{-1}^+ \preceq \lm_{0}^-\preceq\lm_{0}^+ \preceq\lm_1^- \preceq \lm_1^+ \preceq \cdots \, .
$$
We note that for $v = 0$, one has
\begin{equation}\label{eq:kap3.170}
\Dl(\lm_k^+,0) = \Dl(\lm_k^-,0) = (-1)^k \, , \; \forall k\in\Z,
\qquad
\lm_{-k}^+ = \frac{1}{16\,\lm_k^+}\, , \,\,  \forall k\geq 0 \, .
\end{equation}
Since the line segment $\{ tv\in\Hp :  t\in [0,1] \}$, connecting $v$ to $0$ in \Hp, is compact, the integer $N$ in the Counting Lemma 
 \cite[Lemma 3.11]{LOperator2018} can be chosen uniformly with respect to $0\leq t\leq 1$. 
 Furthermore, for any $|k| > N$ $\Dl(\lm_k^+(tv),tv) =\Dl(\lm_k^-(tv),tv)$ and its sign is constant in $t$. We conclude that 
 \begin{equation}\label{sign Delta}
 \Dl(\lm_k^\pm,v)= (-1)^k \, , \quad \forall |k| > N \, .
 \end{equation} 
 Such an identity does not hold for the finitely many eigenvalues $\lm_k^\pm$, $|k| \le N$, unless $v$ satisfies further conditions 
 such as being (almost) real valued -- see Section \ref{sec:realAmostReal} for details.
 
 \begin{figure}
 \centering
   \includegraphics[width=\textwidth]{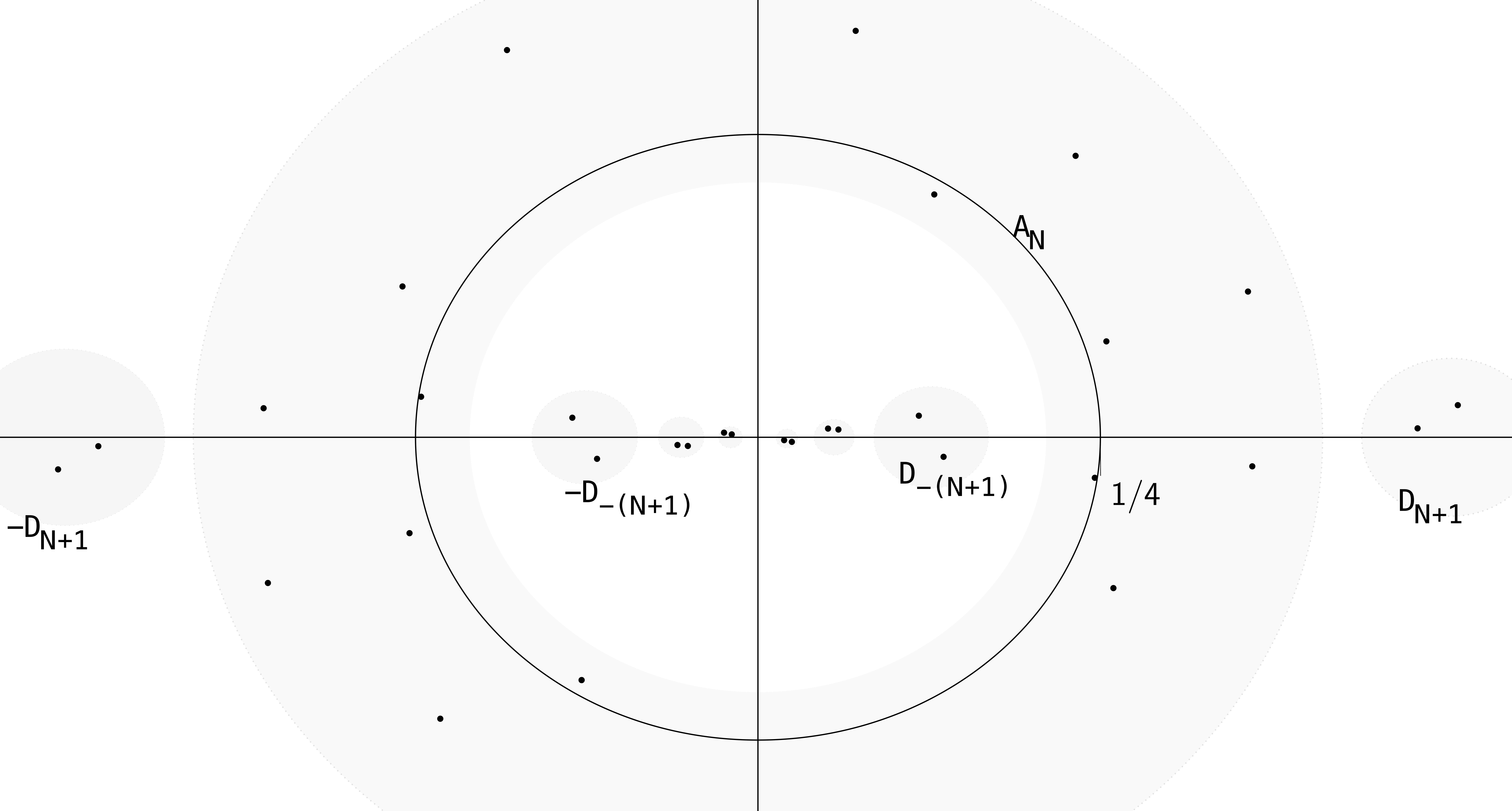}
        \caption{Distribution of periodic eigenvalues}
\end{figure}
The Dirichlet eigenvalue in $D_n$, $|n| > N$, is denoted by $\mu_n$.
By \cite[Theorem 3.1, Lemma 3.2]{LOperator2018},  the Dirichlet eigenvalue in $-D_n$ is given by $-\mu_n$.
Besides $\mu_n$, $| n | > N$, there are $1+2N$ Dirichlet eigenvalues of $Q(v)$ in $\SpecDom$.
Altogether, the eigenvalues in $\C^+$ can be listed as a bi-infinite sequence,
which is increasing with respect to the order $\preceq$,
\[
  0 \prec \cdots  \preceq \mu_{-1} \preceq \mu_0 \preceq \mu_1  \preceq  \cdots \, . 
\]


\medskip

 We finish this section with a discussion of the roots of $\dot\Dl(\lm,v) \equiv \partial_\lm \Dl(\lm,v)$.
 Since $\Dl(\lm, v)$ is even with respect to $\lm$ and hence $\dot\Dl(\lm, v)$ odd, we only need to study the roots of $\dot\Dl(\cdot, v)$ in \SpecDom.
 For $v=0$ one has $\Dl(\lm,0) = \cos(\omega(\lm))$ (cf. \eqref{eq:Eomega})
 and hence
\begin{equation}\label{eq:dotDlzero}
 \dot\Dl(\lm,0) = - (1+ \frac{1}{16\lm^2})\sin(\omega(\lm)).
\end{equation}
The roots of $\dot\Dl(\cdot, 0)$ in \SpecDom are given by the bi-infinite sequence
\[
\dot\lm_k \equiv \dot\lm_k(0) = \lm_k^+(0), \quad \forall k\in\Z
\]
together with the additional root $\dot\lm_*(0) = \frac{i}{4}$. Each of these roots has multiplicity one.
Since by \cite[Lemma 2.14(ii)]{LOperator2018},  $-\frac{1}{16\lm^2}\dot{\Dl}(\frac{1}{16\lm},q,p) =\dot{\Dl}(\lm,-q,p)$,  it follows that $\dot{\Dl}(\frac{1}{16\lm}, q, p)$ and $\dot{\Dl}(\lm, -q, p)$ 
have the same roots in $\C^*$ (counted with their multiplicities).
By \cite[Lemma 3.12]{LOperator2018}, the following Counting Lemma holds  for $\dot{\Dl}$:
For any potential in \Hp there exist a neighborhood $W \subset \Hp$ and $N \ge 1$ so that for any $v\in W$, $ \dot{\Dl}(\cdot,v)$ 
has exactly one root in each  of the domains  $D_n$, $-D_n$, $D_{-n}$, and $-D_{-n}$ with $n>N$ and $4+4N$ roots in the annulus $A_N$. There are no other roots.
(We remark that in  Lemma 3.4, Lemma 3.11, and Lemma 3.12,  in \cite{LOperator2018}, $W$ and $N$ can be chosen to be the same.)
Note that the roots of $\dot\Dl(\cdot,v)$ in $\SpecDom\setminus A_N$ can be listed as a bi-infinite sequence,
\begin{equation}\label{eq:dotLmListing}
        0 \preceq \dots\preceq \dot\lm_{-N-2} \preceq \dot\lm_{-N-1} \preceq \dot\lm_{N+1}\preceq \dot\lm_{N+2}\preceq \cdots, \quad \dot\lm_k\in D_k \; \forall |k|>N
\end{equation}
such that any remaining root $\dot\lm$ of $\dot\Dl(\cdot,v)$  in \SpecDom satisfies $\dot\lm_{-N-1}\preceq \dot\lm\preceq \dot\lm_{N+1}$.
It turns out that for arbitrary $v\in\Hp$, these remaining roots cannot be listed  in a way useful  for our purposes. 
However in the case where $v$ is in \Hr, such a listing is possible.
Indeed in this case, the algebraic multiplicity of any periodic eigenvalue 
of $Q(v)$ is at most two, whereas all  roots of  $\dot\Dl$ are simple.
The ones which are real and contained  in  \SpecDom can be listed in form of a bi-infinite sequence
\begin{equation}\label{lambda dot real}
        0 <  \cdots < \dot\lm_{-2} <\dot\lm_{-1} <\dot\lm_0 < \dot\lm_1 < \cdots 
\end{equation}
and satisfy $\lambda_k^- \le \dot\lm_k \le \lm_k^+$ for any $k \in \Z$.
There is one additional root of  $\dot\Dl$ in  \SpecDom. It is purely imaginary and denoted by  $\dot\lm_*$.
We refer to Section \ref{sec:realAmostReal} where we will also consider potentials $v \in \Hp$ near $\Hr$.



\section{Gradients}\label{gradients}

In this section we derive formulas for the $L^2-$gradients of various quantities, needed later. 
In Section \ref{subsec:gradients} we  compute the $L^2-$gradient of the Floquet matrix $\grave{M}(\lm,v) = M(1,\lm,v)$ and then use it to obtain formulas for the $L^2-$gradients 
of simple eigenvalues of $Q(v)$ such as simple periodic or Dirichlet 
eigenvalues. In Section \ref{subsec:floquet}, we compute the Floquet solutions of $Q(v)$ 
and use them to simplify the formulas for the $L^2-$gradients of periodic eigenvalues of $Q(v)$,
obtained in Section \ref{subsec:gradients}. In Section \ref{asymptotics}, we derive asymptotics
of the $L^2-$gradients of various quantities such as $\grave{M}$.
Throughout this section we use the regularity properties of the fundamental solution $M(x, \lambda, v)$ established in
\cite[Section 2]{LOperator2018} (cf.  in particular \cite[Theorem 2.2]{LOperator2018}).

We denote by $\mathrm{d}F$ the differential of a map $F:\Hp \to X$ with values in a complex Banach space $X$ and by $\mathrm{d}F[\mathring{v}]$ 
the directional derivative of $F$ in direction $\mathring{v}=(\mathring{q},\mathring{p})\in\Hp$. Furthermore if $F$ takes values in $\C$ then $\partial_q F,\partial_pF$ denote the $L^2$-gradients of $F$ with respect to $q$ and $p$ and $\partial F \equiv \partial_v F$ the one of $F$,  $\partial F= (\partial_q F,\partial_p F)$. Here $\langle \partial_q F,\mathring{q}\rangle_r = \mathrm{d}
F[\mathring{q},0]$ where $\langle\cdot,\cdot\rangle_r$
denotes the $L^2$-pairing (no complex conjugation)
\[
\langle f,g\rangle_r = \int_0^1 f(x)g(x)\dx \quad \forall  f,g\in L^2(\T,\C)
\]
and its extension to pairings between $H^n_\C$ and $H^{-n}_\C$.

\subsection{Formulas for $L^2-$gradients. Part 1.}\label{subsec:gradients}

Recall that $M(x,\lm,v)$ denotes the fundamental solution of \eqref{eq:dxM} and $I,J,Z,R,$ and $P$ are given by
\[
I=\begin{pmatrix}1&\\&1\end{pmatrix}, \quad J=\begin{pmatrix}&1\\-1\end{pmatrix},\quad Z= \begin{pmatrix}&1\\1\end{pmatrix},\quad R=\begin{pmatrix} \ii \\ & -\ii\end{pmatrix}, \quad \Op=\sqrt{1-\partial_x^2}.
\]
For a matrix valued function $A(x)$ of a real variable $x$,  $\Op(A)(x)$ is the matrix valued function obtained by applying $\Op$ to each matrix coefficient of $A(x)$. 
Furthermore, we denote by $EV_0$ the evaluation map $\Hp\to\C,(q,p)\mapsto q(0)$ and by $\text{Mat}_{2\times 2}(\C)$
the $\mathbb C-$vector space of $2\times 2$ matrices with complex coefficients.
\begin{prop} \label{prop:Mgradient}
For any fixed $\lm\in\C^*$, the $L^2-$gradient of $\grave{M}(\lm,v) $ at $v\in \Hp[1]$ is given by
\begin{align}\label{eq:prodMgradientOne}
\partial_q \grave{M} =&-\frac14 \grave{M} M^{-1}\ii RM\cdot \partial_x(\cdot)
  -\frac{1}{16\,\lm}\grave{M}M^{-1} \begin{pmatrix}&e^q \\ e^{-q}\end{pmatrix} M
\\ \partial_p \grave{M} =& -\frac\ii4 \grave{M}  M^{-1} R M \cdot \Op (\cdot). \label{eq:prodMgradientTwo}
\end{align}
Alternatively $\partial_q\grave{M}$ can be written as
\begin{equation} \label{eq:prodMgradientThree}
\partial_q\grave{M} = \frac{1}{2}\begin{pmatrix}&\grave{m}_2\\ -\grave{m}_3 &\end{pmatrix}EV_0 - \frac12 \grave{M}M^{-1}\Big(\lm Z + \frac{1}{16\,\lm}\begin{pmatrix}&e^q \\ e^{-q}\end{pmatrix}\Big) M.
\end{equation}
Here $\partial_q\grave{M}$ and $\partial_p \grave{M}$ are viewed as elements in $H^{-1}(\T, {\rm{Mat}}_{2\times 2}(\C))$, meaning that for
any $\mathring{q}, \mathring{p}\in H^1(\T,\C)$, one has
$$
\langle \partial_p \grave{M},\mathring{p}\rangle_r = \langle -\frac\ii4 \grave{M} M^{-1} RM\, ,\, \Op(\mathring{p})\rangle_r
$$ 
$$
\langle \partial_q \grave{M},\mathring{q}\rangle_r =
\langle -\frac14\grave{M} M^{-1}\ii RM,\partial_x \mathring{q}\rangle_r - \langle \frac{1}{16\,\lm}\grave{M}M^{-1} \begin{pmatrix}&e^q \\ e^{-q}\end{pmatrix} M,\mathring{q}\rangle_r.
$$
\end{prop}
\begin{proof}
By \cite[Theorem 2.2]{LOperator2018}, for any given $\lambda \in \C^*$,
the Floquet matrix $M(1,\lm,v)$ is analytic in $v$. Since all terms in the above formulas depend continuously on $v$ it suffices to verify them for sufficiently smooth $v$ for which we may interchange differentiation with respect to $x$ and $v$. For $\mathring{v}=(\mathring{q},\mathring{p})\in \Hp$ take the derivative of both sides of equation \eqref{eq:dxM} in direction $\mathring{v}$, to  obtain
\begin{equation}\label{eq:dM}
 \partial_x \mathrm{d}M[\mathring{v}] =J\left(\lm - A - B^2/\lm \right)\mathrm{d} M[\mathring{v}] -J\left(\mathrm{d}A[\mathring{v}]+\mathrm{d}(B^2)[\mathring{v}]/\lm\right)M
 \end{equation}
where by \eqref{eq:QoperatorAandB}
\[
  \mathrm{d}A [\mathring{v}]=-\frac1 4(\Op\mathring{p}+\mathring{q}_x) Z \qquad \mathrm{d}(B^2 )
  [\mathring{v}]=\frac{1}{16}\begin{pmatrix}-e^{-q}
 & \\ & e^q
 \end{pmatrix}\mathring{q}.
\]
Furthermore since $\left.M(x)\right|_{x=0}=I$ one has $\mathrm{d}M(x)[\mathring{v}]\Big|_{x=0} = 0$. Since $\mathrm{d}M(x)[\mathring{v}]$ solves the linear differential equation \eqref{eq:dM} it then follows that
\begin{equation}\label{eq:dMh}
  \mathrm{d}M(x) [\mathring{v}]= - M(x)\int_0^x M^{-1}(s) J \big(\mathrm{d}  A[\mathring{v}] +\mathrm{d} (B^2)[\mathring{v}]/\lm \big)M(s)\mathrm{d}s.
\end{equation}
For $\mathring{q} = 0$ the integrand equals
\[
M^{-1}(s)J\mathrm{d} A[0,\mathring{p}]M(s)
= -\frac14 (\Op\mathring{p}) M^{-1}(s)JZM(s)
= \frac14(\Op\mathring{p}) M^{-1}(s)\ii R M(s).
\]
Evaluating $\mathrm{d}M(x)[0,\mathring{p}]$ at $x=1$ yields  the claimed formula \eqref{eq:prodMgradientTwo} for $\partial_p \grave{M}$.\\
For $\mathring{p}= 0$, the integrand of \eqref{eq:dMh} equals
\begin{align}
\begin{split}\label{eq:gradientInp}
M^{-1}(s)J& (\mathrm{d}  A[\mathring{q},0]+\mathrm{d} B^2[\mathring{q},0]/\lm)M(s)
\\&=\frac14 \mathring{q}_x M^{-1}(s)\ii RM(s)
  +\mathring{q}\frac{1}{16\,\lm}M^{-1}(s) J\begin{pmatrix}-e^{-q}& \\ & e^q \end{pmatrix} M(s).
\end{split}
\end{align}
Evaluating $-M(x)\int_0^xM^{-1}(s)J\mathrm{d}A[\mathring{q},0]M(s)\mathrm{d}s$ at $x=1$ one obtains \eqref{eq:prodMgradientOne}. Furthermore, integrating by parts yields
\[
 -\grave{M}\int_0^1 \frac14\mathring{q}_x(s) M^{-1}(s)\ii R M(s)\ds = \frac{1}{2}\begin{pmatrix}&\grave{m}_2\\ -\grave{m}_3 &\end{pmatrix} \mathring{q}(0) 
 + \frac14 \grave{M}\int_0^1 \mathring{q}(s) \partial_s \big( M^{-1}(s)\ii R M(s) \big)\ds.
\]
Since
\begin{align*}
\partial_s(M^{-1}\ii R M) =& -M^{-1}(\partial_s M)M^{-1}\ii RM+ M^{-1}\ii R(\partial_s M)
= M^{-1}[\ii R\;,\;(\partial_s M)M^{-1}]M
\\=& M^{-1} [\ii R\;,\; J(\lm-A-B^2/\lm)] M = M^{-1}\Big(-2\lm Z + \frac{2}{16\,\lm}\begin{pmatrix}&e^q \\ e^{-q}\end{pmatrix}\Big) M
\end{align*}
(where $[A,B]$ denotes the commutator $AB-BA$ of two square matrices $A,B$) one concludes that
\begin{align*}
 \mathrm{d}\grave{M}[\mathring{q},0]=& -\grave{M}\int_0^1 M^{-1}(s)J \Big(\mathrm{d}  A[\mathring{q},0]+\mathrm{d} B^2[\mathring{q},0]/\lm\Big)M(s)\dx
 \\=&\frac{1}{2}\begin{pmatrix}&\grave{m}_2\\ -\grave{m}_3 &\end{pmatrix} \mathring{q}(0)
 \\ &-\frac12\grave{M}\int_0^1 \mathring{q}M^{-1}\Big(\lm Z - \frac{1}{16\,\lm}\begin{pmatrix}&e^q \\ e^{-q}\end{pmatrix}\Big) M  +  \mathring{q}\frac{2}{16\,\lm}M^{-1} J\begin{pmatrix}-e^{-q}& \\ & e^q \end{pmatrix} M\dx
 \end{align*}
 or
 \[
 \partial_q\grave{M} =\frac{1}{2}\begin{pmatrix}&\grave{m}_2\\ -\grave{m}_3 &\end{pmatrix} Ev_0
 -\frac12\grave{M} M^{-1}\Big(\lm Z + \frac{1}{16\,\lm}\begin{pmatrix}&e^q \\ e^{-q}\end{pmatrix}\Big) M .
\]
This proves \eqref{eq:prodMgradientThree}
\end{proof}

Proposition \ref{prop:Mgradient} can be used to compute the gradients of the discriminant $\Dl = (\grave{m}_1+\grave{m}_4)/2$ and the anti-discriminant $\dl = (\grave{m}_1-\grave{m}_4)/2$.

\begin{lem} \label{lem:gradientDldl}

\begin{thmenum}
\item For any fixed $\lm\in\C^*$, the $L^2-$gradient of $\Dl = \Dl(\lm)$ at $v\in\Hp[1]$ is given by
\begin{align*}
  \partial_q\Dl = & \frac\lm4\Big( \grave{m}_2(m_3^2-m_1^2) + \grave{m}_3(m_2^2-m_4^2) + 2\dl \cdot (m_1m_2-m_3m_4) \Big)
  \\ &\indent +\frac{1}{64 \,\lm} \Big(e^{-q}(\grave{m}_3 m_2^2 -\grave{m}_2 m_1^2  + 2\dl \cdot m_1m_2 ) + e^q ( \grave{m}_2 m_3^2-\grave{m}_3 m_4^2 - 2\dl \cdot m_3m_4)\Big)
\\ \partial_p\Dl = &\frac{1}{4}\Big(-\grave{m}_2m_1m_3+\grave{m}_3m_2m_4 + \delta \cdot  ({m}_1{m}_4+{m}_2{m}_3\Big)\Op(\cdot) ,
\end{align*}
\item For any fixed $\lm\in\C^*$, the $L^2-$gradient of the anti-discriminant is given by
\begin{align*}
\partial_q \delta =& \frac\lm4\Big(\grave{m}_2 (m_3^2-m_1^2) +  \grave{m}_3( m_4^2 - m_2^2   ) + 2\Dl\cdot(m_1m_2-m_3m_4 )\Big)
\\ &\indent -\frac{1}{64\,\lm} \Big(e^{-q}(\grave{m}_2 m_1^2 + \grave{m}_3m_2^2 - 2\Dl \cdot m_1m_2) + e^q( -\grave{m}_3m_4^2 - \grave{m}_2m_3^2 + 2\Dl\cdot  m_3m_4 )\Big)
\\	\partial_p\delta =&  \frac{1}{4}\Big(- \grave{m}_3m_2m_4 - \grave{m}_2m_1m_3  +\Dl \cdot ({m}_1{m}_4+ {m}_2{m}_3) \Big)\Op(\cdot).
\end{align*}
\item At the zero potential $v=0$, one has $\partial\Dl =0$ and $\partial \dl$ is given by
\begin{align*}
\partial_q\dl(\lm,0) =& \frac12\Big(\lm +\frac{1}{16\,\lm}\Big)\Big(\cos(\omega(\lm))\sin(2\omega(\lm)x)-\sin(\omega(\lm))\cos(2\omega(\lm)x)\Big)
\\
\partial_p\delta(\lm,0) =& \frac{\cos(\omega(\lm))}{4}\Big(\cos(2\omega(\lm)x)\Big)\Op(\cdot) +\frac{\sin(\omega(\lm))}4\Big(\sin(2\omega(\lm)x)\Big)\Op(\cdot).
\end{align*}
\end{thmenum}
\end{lem}

\begin{proof} Items (i) and (ii) follow from
Proposition \ref{prop:Mgradient}. To prove item (iii) substitute $E_{\omega(\lm)}$, defined in \eqref{eq:Eomega}, for $M$ into the formulas in item (i) to conclude that $\partial\Dl$ vanishes. For $\dl$ we obtain
\begin{align*}
\partial_q\dl(\lm,0) =& \frac14\Big(\lm +\frac{1}{16\,\lm}\Big)\Big(2\sin(\omega(\lm))\big( \sin^2(\omega(\lm)x)-\cos^2(\omega(\lm)x)  \big) + 4\cos(\omega(\lm))\sin(\omega(\lm)x)\cos(\omega(\lm)x) \Big) \
\end{align*}
and
\begin{align*}
\partial_p\delta(\lm,0) =& \frac{\cos(\omega(\lm))}4 \Big(\cos^2(\omega(\lm)x)-\sin^2(\omega(\lm)x)\Big)\Op(\cdot)
+\frac12\sin(\omega(\lm)) \Big(\sin(\omega(\lm)x)\cos(\omega(\lm)x)\Big)\Op(\cdot)
\end{align*}
which can be further simplified using that $\cos^2(x)-\sin^2(x) = \cos(2x)$ and $2\sin(x)\cos(x) = \sin(2x)$.
\end{proof}

Next we want to obtain formulas for the $L^2-$gradient of simple periodic, antiperiodic, and Dirichlet
eigenvalues of the operator $Q(v)$, when considered on $L^2([0, 1],\C^2)$.
(Recall that the periodic spectrum of $Q(v)$ is the union of the spectrum of $Q(v)$, when considered as an operator on 
$L^2([0, 1],\C^2)$ with periodic boundary conditions ($F(1) = F(0)$) and the one of $Q(v)$,  when considered as an operator on 
$L^2([0, 1],\C^2)$ with antiperiodic  boundary conditions ($F(1)= - F(0)$).
First we prove the following auxiliary result.
Recall that for any $\lm \in \C^*$, a solution $F$ of $QF=\lm F$  has the form $F = (f, \frac{1}{\lm}Bf)$ with $f(x) = M(x,\lm)a$ and $a\in \C^2$.
\begin{prop}\label{prop:spectrumGradientGeneral}
Assume that $\kappa \equiv \kappa(v)\in\C^*$ and $a(v)\in\C^2$ are  analytic functions on some open set in $\Hp[1]$ and define $f(x,v) = (f_1(x,v),f_2(x,v))=M(x,\kappa(v),v)a(v)\in H^1([0,1],\C^2)$. Then for any $\mathring{v}=(\mathring{q},\mathring{p})\in \Hp[1]$, the derivative $\ddh\kappa$ of $\kappa$ at $v$ in direction $\mathring{v}$ is given by
 \begin{align*}
 \ddh \kappa =  \frac{1}{\int_0^1 f\cdot (I+\frac{1}{\kappa^2}B^2)f\dx} \Big( & \left[\ddh f\cdot Jf - \frac12 \mathring{q}f_1f_2 \right]_0^1 + \int_0^1\Big(\frac{\kappa}{2} (f_2^2-f_1^2) +\frac{1}{32\,\kappa}(f_2^2e^q-f_1^2e^{-q}) \Big)\mathring{q}\dx
 \\&\indent- \frac12\int_0^1  f_1f_2\Op (\mathring{p})\dx\Big).
\end{align*}
\end{prop}
To prove Proposition \ref{prop:spectrumGradientGeneral} we first prove the following two lemmas.
\begin{lem}\label{lem:eigenfunctionGrad}
 Let $\kappa(v)$ and $f(v) = (f_1(v),f_2(v))$ be given as in Proposition \ref{prop:spectrumGradientGeneral}. Then for any $\mathring{v}=(\mathring{q},\mathring{p})\in \Hp[1]$,  $\ddh \kappa$ equals
\begin{equation}\label{eq:lemeigenfunctionGrad}\frac{1}{\int_0^1 f\cdot(I+\frac{1}{\kappa^2}B^2)f\dx} \left([ \ddh f\cdot Jf]_0^1 + \int_0^1\frac{1}{16\kappa}(f_2^2e^q - f_1^2e^{-q})\mathring{q} -\frac12(\Op (\mathring{p}) +\partial_x \mathring{q})f_1f_2\dx\right).
\end{equation}
\end{lem}

\begin{proof}[Proof of Lemma \ref{lem:eigenfunctionGrad}]
Let $F= ( f, \frac{1}{\kappa}Bf )$ and differentiate both sides of the identity
\[
 QF(x,v) = \kappa(v) F(x,v)
\]
in the direction $\mathring{v} =(\mathring{q},\mathring{p})\in \Hp$  to obtain
\begin{equation}\label{eq:ddhQ}
(\ddh Q) F + Q\ddh F = (\ddh \kappa )F + \kappa \ddh F.
\end{equation}
Since $Q_0(v)$ is symmetric, one obtains
\begin{align*}
 \int_0^1 (Q\ddh F-\kappa \ddh F)\cdot F \dx =& \int_0^1 (Q_1 \partial_x + Q_0)\ddh F\cdot F - \ddh F\cdot (Q_1 \partial_x + Q_0) F
 \\ =& \int_0^1 \big(Q_1 \partial_x \ddh F\big)\cdot F - \ddh F\cdot (Q_1 \partial_x  F)
 \\=& \int_0^1 -J\partial_x (\ddh f)\cdot f - \ddh f\cdot (-J\partial_x f)\dx = [\ddh f\cdot J f]_{0}^1 .
\end{align*}
Thus \eqref{eq:ddhQ} yields
\begin{equation}\label{eq:proofLemddhQ}
 \int_0^1 (\ddh Q )F \cdot F\dx + [\ddh f\cdot Jf]_{0}^1 =  \ddh\kappa \int_0^1 F\cdot F\dx
 = \ddh\kappa \int_0^1 f(I+\frac{1}{\kappa^2} B^2)f\dx.
\end{equation}
Finally
\begin{align*}
 (\ddh Q F)\cdot F =& \begin{pmatrix} \ddh A & \ddh B\\ \ddh B \end{pmatrix} \begin{pmatrix} f \\ \frac{1}{\kappa}B f\end{pmatrix} \cdot \begin{pmatrix} f \\ \frac{1}{\kappa} Bf\end{pmatrix}
 \\ =& \ddh A f\cdot f +\frac{1}{\kappa} (\ddh B) B f\cdot f + \frac{1}{\kappa}(\ddh B) f\cdot B f
 \\ =& \ddh A f\cdot f + \frac{2}{\kappa}(\ddh B) B f\cdot f
\end{align*}
where $\frac{2}{\kappa}(\ddh B)B = \frac{1}{16\,\kappa} \begin{pmatrix}-e^{-q}\\&e^q\end{pmatrix}\mathring{q}$ and
\begin{align*}
   \ddh A f\cdot f =&  -\frac14( \Op\mathring{p}+ \partial_x \mathring{q}) f\cdot Zf.
\end{align*}
Hence
\[
 \ddh Q  F\cdot F = -\frac{1}{2}( \Op(\mathring{p}) +\partial_x\mathring{q})f_1f_2+\frac{1}{16\kappa}\left(  f_2^2e^q-f_1^2e^{-q} \right)\mathring{q}.
 \]
 Combining this with \eqref{eq:proofLemddhQ} leads to the claimed identity.
\end{proof}

\begin{lem} \label{lem:dxfstarf}
For $v\in\Hp[1]$, $\lm,\mu\in\C^*$, and $a_\lm,a_\mu\in\C^2$, let $f(x) = (f_1(x),f_2(x))=M(x,\lm,v)a_\lm$ and $g(x) = (g_1(x),g_2(x))=M(x,\mu,v)a_\mu$. Then
\begin{align*}
 \partial_x(f_1g_2+f_2g_1)  &= (\lm+\mu) (f_2g_2-f_1g_1)+\frac{1}{16}\left(\frac{1}{\lm}+\frac{1}{\mu}\right)(f_1g_1e^{-q}-f_2g_2e^q).
 \end{align*}
\end{lem}
\begin{proof}[Proof of Lemma \ref{lem:dxfstarf}]
Using that $f$ and $g$ fulfill \eqref{eq:dxM} we compute
\begin{align*}
\partial_x(f_1g_1+f_2g_2)&= \partial_x(f\cdot Z g) =  (J(\lambda - A- B^2/\lambda )f\cdot Z g)
   + (f\cdot  ZJ(\mu - A- B^2/\mu )g)
  \\ &= (\lm +\mu)f\cdot ZJ g +\frac14( \Op (p)+q_x)\left(JZf\cdot Zg + f\cdot ZJZ g\right)
  \\&\indent- \frac{1}{16}\left( \frac{1}{\lm}Je^{\ii Rq}f\cdot Zg + \frac{1}{\mu}f\cdot ZJe^{\ii Rq}g\right)
  \\ &= (\lm+\mu)(f_2g_2-f_1g_1) +\frac{1}{16}\left(\frac{1}{\lm}+\frac{1}{\mu}\right)(f_1g_1e^{-q}-f_2g_2e^q).
  \end{align*}
\end{proof}
\begin{proof}[Proof of Proposition \ref{prop:spectrumGradientGeneral}]
Our starting point is formula \eqref{eq:lemeigenfunctionGrad}. To obtain a formula for the directional derivative $\ddh\kappa$ we need to integrate the term $-\frac12 \int_0^1 (\partial_x \mathring{q})f_1f_2\dx$ in \eqref{eq:lemeigenfunctionGrad} by parts. Since by Lemma \ref{lem:dxfstarf} with $g=f$ and $\lm=\kappa$, $\mu=\kappa$
\[
 \int_0^1 (\partial_x \mathring{q})f_1f_2\dx = [\mathring{q}f_1f_2]_0^1 -   \int_0^1 \kappa\mathring{q} (f_2^2-f_1^2)+ \frac{1}{16\,\kappa}\mathring{q}(f_1^2 e^{-q}-f_2^2 e^q)\dx \, ,
\]
formula \eqref{eq:lemeigenfunctionGrad} yields
 \begin{align*}
 \ddh \kappa =  \frac{1}{\int_0^1 f\cdot (I+\frac{1}{\kappa^2}B^2)f\dx} &  \Big(\big[\ddh f\cdot Jf - \frac12 \mathring{q}f_1f_2 \big]_0^1 + \int_0^1\Big(\frac{\kappa}{2} (f_2^2-f_1^2) +\frac{1}{32\,\kappa}(f_2^2e^q-f_1^2e^{-q}) \Big)\mathring{q}\dx
 \\&\indent- \frac12\int_0^1  f_1f_2\Op (\mathring{p})\dx\Big).
\end{align*}
\end{proof}

For what follows it is useful to introduce for any for $f=(f_1,f_2) \in H^1_c([0, 1], \C^2)$, $v \in H^1_c$, and $\lm \in \C^*$ the expression
\begin{equation}\label{eq:Grad}
  \Grad{f}{\lm} := \begin{pmatrix} \frac{\lm}{2} (f_2^2-f_1^2) + \frac{1}{32\lm} (f_2^2 e^q-f_1^2e^{-q})\\   -\frac12 f_1f_2\Op(\cdot)\end{pmatrix}
\end{equation}
The first component of $ \Grad{f}{\lm}$ acts on $H^1([0,1], \C)$ as a multiplication operator and the second one 
as a pseudodifferential operator.
Furthermore, we denote  by $M_1$ and $M_2$  the first, respectively second column of $M$. 
(By a slight abuse of notation, we will often write  $M_1$ and $M_2$ as row vectors $(m_1,m_3)$  and respectively, $(m_2,m_4)$.) 

We will consider the operator $Q(v)$ on $L^2([0,1],\C^2)$ with either periodic, antiperiodic, or Dirichlet boundary conditions and 
assume that on an open subset $V$ in $\Hp$, $\kappa=\kappa(v) \in \C^*$ is a simple periodic, antiperiodic, 
or Dirichlet eigenvalue. Then $\kappa(v)$ is analytic in $v$ and we can choose corresponding eigenfunctions $f\equiv f_\kappa$ so that $V \to H^1([0,1],\C^2),$ $v \mapsto f =(f_1, f_2)$ is analytic. 
Since $f$ satisfies either Dirichlet, periodic, or antiperiodic boundary conditions, one has
in all three cases that for any $v \in V$ and any $ \mathring{v}=(\mathring{q}, \mathring{p}) \in H^1([0,1],\C^2)$, satisfying the corresponding boundary conditions,
\[
\ddh f\cdot Jf\Big|_0^1 = 0 \, , \qquad   \frac12f_1f_2 \mathring{q} \Big|_0^1 = 0 \, .
\]
Then according to Proposition \ref{prop:spectrumGradientGeneral},
\begin{equation}\label{eq:ddhlmWithGrad}
 \ddh \kappa = \frac{1}{\int_0^1 f\cdot (I+\frac{1}{\kappa^2}B^2)f\dx}  \int_0^1 \Grad{f}{\kappa}\cdot \mathring{v} \dx \, .
\end{equation}
(Note that in order to simplify notation we write $\Grad{f}{\kappa}$ instead of $\Grad{f_\kappa}{\kappa}$.)
\begin{prop} \label{lem:gradientDirichlet}
Assume that for a given $v_0\in\Hp$,  $\mu(v_0) \in \C^*$ is a simple Dirichlet eigenvalue of $Q(v_0)$. Then there exists a neighborhood $V$ of $v_0$ in \Hp and $\epsilon>0$ 
so that for any $v\in V$ there is a unique Dirichlet eigenvalue $\mu(v)$ of $Q(v)$ in the disc of radius $\epsilon$ centered at $\mu(v_0)$. 
The eigenvalue $\mu$ is analytic on $V$ and
$\partial \mu   =\frac{\grave{m}_1}{\dot\chi_D}\Grad{M_2}{\mu}$.
(Note that $M_2 = M_2(\cdot,\mu(v))$ is a convenient choice for an eigenfunction, corresponding to the eigenvalue $\mu(v)$.) More explicitly, one has
for any $v \in V$ and any $ \mathring{v}=(\mathring{q}, \mathring{p}) \in H^1([0,1],\C^2)$, satisfying Dirichlet boundary conditions,
 \[
 \ddh \mu
 = - \frac{\ddh{\grave{m}_2} \big|_{\lm=\mu}}{\dot\chi_D(\mu)}
= \frac{\grave{m}_1(\mu)}{\dot\chi_D(\mu)} \int_0^1
\begin{pmatrix} \frac\mu2(m_4^2-m_2^2) + \frac{1}{32\mu}(m_4^2e^q-m_2^2e^{-q})
\\ -\frac12  m_2m_4
\end{pmatrix}\cdot
\begin{pmatrix}
\mathring{q}\\\Op(\mathring{p})
\end{pmatrix}
\dx.
 \]
\end{prop}
\begin{proof}
Existence of $\epsilon > 0$ and neighborhood $V$ with the claimed properties 
follow by the argument principle.
By \eqref{eq:Grad}-\eqref{eq:ddhlmWithGrad} we have
\begin{equation}\label{eq:pfLemkappa}
 \mathrm{d} \mu[0,\mathring{p}] = \frac{1}{\int_0^1 M_2\cdot (I+\frac{1}{\mu^2}B^2)M_2\dx}  \int_0^1 -\frac12 m_2m_4\Op (\mathring{p})\dx.
\end{equation}
On the other hand, the characteristic function $\chi_D(\lm, v) = \grave{m}_2(\lm, v)$ statisfies $\chi_D(\mu(v),v) = 0$. 
Applying the chain rule to $\chi_D(\mu(v),v) = 0$ and using that $\dot \chi_D(\mu(v), v)\not=0$ since $\mu(v)$ is simple, we get
\begin{equation}\label{eq:pfLemkappa2}
\ddh\mu =  -\frac{\ddh{\chi_D}\big|_{\lm=\mu}}{\dot\chi_D(\mu)} = -\frac{\ddh{\grave{m}_2}\big|_{\lm=\mu}}{\dot \chi_D(\mu)}
\end{equation}
where by Proposition \ref{prop:Mgradient}
\begin{align*}
\mathrm{d} \grave{m}_2(\mu)[0,\mathring{p}] =& \frac{1}{4} \int_0^1\left(2\grave{m}_1 m_2m_4 - \grave{m}_2 (m_1m_4+m_2m_3) \right)\Big|_{\lm=\mu}\Op(\mathring{p})\dx.
\end{align*}
Since at a Dirichlet eigenvalue, $\grave{m}_2=0$, one obtains by
comparing  \eqref{eq:pfLemkappa} and \eqref{eq:pfLemkappa2},
\begin{equation}\label{identity for Dirichlet}
  \frac{1}{\int_0^1 M_2\cdot (I+\frac{1}{\mu^2}B^2)M_2\dx} = \frac{\grave{m}_1(\mu)}
  {\dot\chi_D(\mu)} .
\end{equation}
This together with \eqref{eq:ddhlmWithGrad} yields the claimed formula.
\end{proof}
To obtain practical formulas for the gradients of simple periodic eigenvalues of $Q(v)$
is more involved than in the case of Dirichlet eigenvalues
and will be the topic of the subsequent section.

\subsection{Formulas for $L^2-$gradients. Part 2.}\label{subsec:floquet}
As a first step in this section, we express the $L^2-$gradient of the Floquet matrix $\grave{M}(\lm,v)$ of $Q(v)$, computed in Proposition \ref{eq:prodMgradientOne}, in terms of the Floquet solutions of $Q(v)$.
As an application, one obtains practical formulas for the $L^2-$gradient of the discriminant $\Delta$ and the anti-discriminant $\delta$.

Fix $v\in\Hp$, $\lm\in\C^*$ and consider the two eigenvalues $\xi_\pm = \Dl(\lm)\pm \sqrt{\Dl^2(\lm)-1}$ of $\grave{M}= \grave{M}(\lm,v)$. If $\lm$ is a periodic eigenvalue of $Q(v)$, then 
$\xi_+ = \xi_-$ and $\xi_+ \in\{ 1,-1\}$. 
In case $\lm$ is a double periodic eigenvalue of $Q(v)$, $\grave{M}$ equals $I$ or $-I$ and we choose $a_+ := (1, 0)$ and $a_-:= (0, 1)$ as linearly independent eigenvectors
corresponding to the eigenvalue $\xi_+$.
 If $\lm$ is not a periodic eigenvalue of $Q(v)$, then $\xi_+\not=\xi_-$ and we denote by $a_+$ and  $a_-$  corresponding eigenvectors of $\grave{M}$, $\grave{M}a_{\pm} = \xi_\pm a_\pm$.
Hence, except when $\lm$ is a simple periodic eigenvalue, there are two distinct Floquet solutions, $f_+ \equiv (f_1^+,f_2^+) := M(x,\lm,v)a_+$ and $f_-\equiv(f_1^-,f_2^-) :=M(x,\lm,v)a_-$. Using these two solutions, 
we express the gradient of $\grave{M}$ and hence the one of $\Dl$ and $\delta$ in terms of the Floquet solutions $f_+,$ $f_-$.
Since $a_+,a_-$ are linearly independent, the $2\times 2$ matrix $(a_+\; a_-)$ with columns $a_+$ and $a_-$ is regular. Note that
\begin{equation}  \label{eq:MbyFloquet}
  \left(f_+ \;f_-\right)\left(a_+ \; a_-\right)^{-1} = M \left(a_+ \;a_-\right)\left(a_+\; a_-\right)^{-1} = M.
\end{equation}
To simplify the formulas below, introduce for any column vector $\begin{pmatrix}a_1\\a_2\end{pmatrix}\in\C^2$ the row vector
$$
\begin{pmatrix}a_1\\a_2\end{pmatrix}^\bot= \left(-a_2 \  a_1\right) \,.
$$
Then for any linearly independent column vectors $a,b\in\C^2$, the determinant of the $2\times 2$ matrix \begin{small}$\begin{pmatrix}a & b\end{pmatrix}$\end{small}  
equals the $1\times1$ matrix $a^\bot  b$ and the matrix inverse \begin{small}$\begin{pmatrix}a & b\end{pmatrix}^{-1}$\end{small} can be written as
\begin{equation}\label{eq:matrixInverse}
(a\; b)^{-1} = \frac{1}{a^\bot b}\begin{pmatrix}-b^\bot\\a^\bot\end{pmatrix}.
\end{equation}
Furthermore, we introduce the star product
\[
\begin{pmatrix}a_1\\a_2\end{pmatrix}\star \begin{pmatrix}b_1\\b_2\end{pmatrix} = a_1b_1-a_2b_2.
\]
\begin{lem} \label{lem:gradMFloquet}
Let $v\in\Hp$ and assume that $\lm\in\C^*$ is not a simple periodic eigenvalue of $Q(v)$. Then,
 \begin{align*}
  \partial_q \grave{M} =&\frac{1}{2}\begin{pmatrix}&\grave{m}_2\\ -\grave{m}_3 &\end{pmatrix}EV_0
    + \frac{\lm}{2 a_+^\bot a_-} \begin{pmatrix}\xi_+a_+ &\xi_-a_-\end{pmatrix} \begin{pmatrix} f_+\star f_- & f_-\star f_-\\ -f_+ \star f_+ & -f_+ \star f_-\end{pmatrix}(a_+ \; a_-)^{-1}
    \\ &\indent +\frac{1}{32\,\lm\, a_+^\bot a_-}\begin{pmatrix}\xi_+a_+ & \xi_-a_- \end{pmatrix} \left(e^q\begin{pmatrix}-f_2^+f_2^- & -(f_2^-)^2 \\ (f_2^+)^2 & f_2^+f_2^- \end{pmatrix}+ e^{-q}\begin{pmatrix}f_1^+f_1^- & (f_1^-)^2 \\ -(f_1^+)^2 & -f_1^+f_1^-\end{pmatrix}\right) (a_+ \; a_-)^{-1}
  \\ \partial_p \grave{M} =&  \frac{1}{4 a_+^\bot a_-} \begin{pmatrix}\xi_+a_+ & \xi_-a_- \end{pmatrix} \begin{pmatrix} f_+\cdot Zf_- & f_-\cdot Zf_-\\ -f_+\cdot Zf_+ & -f_+\cdot Zf_-\end{pmatrix}\begin{pmatrix}a_+ & a_-\end{pmatrix}^{-1}\Op(\cdot).
 \end{align*}
\end{lem}
\begin{proof}
By Proposition \ref{prop:Mgradient} and \eqref{eq:MbyFloquet}
\begin{align*}
 \partial_p \grave{M} =& -\frac\ii4 \grave{M} M^{-1}R M { \Op(\cdot)}
 \\ =& -\frac\ii4 \begin{pmatrix}\xi_+a_+ & \xi_-a_-\end{pmatrix}(a_+ \; a_-)^{-1}(a_+\; a_-)(f_+\; f_-)^{-1}R(f_+\; f_-)(a_+\;a_-)^{-1} \Op(\cdot).
 \end{align*}
 By \eqref{eq:matrixInverse} it then follows
 \begin{align*}
 \partial_p \grave{M} =& -\frac\ii4 \begin{pmatrix}\xi_+a_+ & \xi_-a_-\end{pmatrix}
\frac{1}{f_+^\bot  f_-}\begin{pmatrix}-f_-^\bot\\f_+^\bot\end{pmatrix}
 R(f_+ \; f_-)(a_+\;  a_-)^{-1}\Op(\cdot)
 \\=& \frac14 \begin{pmatrix}\xi_+a_+ & \xi_-a_-\end{pmatrix} \frac{1}{f_+^{\bot} f_-}
 \begin{pmatrix}f_+\cdot Zf_- & f_-\cdot Zf_-\\ -f_+\cdot Zf_+ & -f_+\cdot Zf_-\end{pmatrix}(a_+\; a_-)^{-1}\Op(\cdot).
\end{align*}
Since in view of \eqref{eq:dxM} $M$ satisfies the Wronskian identity, $\det M =1$, it follows from
\eqref{eq:MbyFloquet} that  $ f_+^\bot f_- = \det(f_+\;f_- ) =   a_+^\bot a_-$
 and hence the claimed formula for $\partial_p \grave{M}$ is proved.
Similarly, by the formula \eqref{eq:prodMgradientThree} of Proposition \ref{prop:Mgradient},
\begin{align*}
\partial_q \grave{M} =& \frac{1}{2}\begin{pmatrix}&\grave{m}_2\\ -\grave{m}_3 &\end{pmatrix}EV_0 - \frac12 \grave{M}M^{-1}\Big(\lm Z + \frac{1}{16\,\lm}\begin{pmatrix}&e^q \\ e^{-q}\end{pmatrix}\Big) M.
\end{align*}
One computes
\begin{align*}
\grave{M}M^{-1}Z M
=&  \begin{pmatrix}\xi_+a_+ & \xi_-a_-\end{pmatrix}(f_+\; f_-)^{-1}Z(f_+\;f_-)(a_+\; a_-)^{-1}
\\ =& \begin{pmatrix}\xi_+a_+ & \xi_-a_- \end{pmatrix} \frac{1}{f_+^{\bot} f_-}
 \begin{pmatrix}- f_+\star f_- & -f_-\star f_-\\ f_+ \star f_+ & f_+ \star f_-\end{pmatrix}(a_+ \; a_-)^{-1}
\end{align*}
and
\begin{align*}
\grave{M}M^{-1} &\begin{pmatrix}&e^q \\ e^{-q}\end{pmatrix}  M
=  \begin{pmatrix}\xi_+a_+ & \xi_-a_-\end{pmatrix}(f_+\;f_-)^{-1}\begin{pmatrix}&e^q \\ e^{-q}\end{pmatrix}(f_+\; f_-)(a_+\; a_-)^{-1}
\\ =&  \begin{pmatrix}\xi_+a_+ & \xi_-a_-\end{pmatrix}\frac{1}{f_+^{\bot} f_-} \left(e^q\begin{pmatrix}f_2^+f_2^- & (f_2^-)^2 \\ -(f_2^+)^2 & -f_2^+f_2^- \end{pmatrix}+ e^{-q}\begin{pmatrix}-f_1^+f_1^- & -(f_1^-)^2 \\ (f_1^+)^2 & f_1^+f_1^-\end{pmatrix}\right) (a_+ \; a_-)^{-1}.
\end{align*}
Using again that $f_+^\bot f_-$ is constant, the claimed formula for $\partial_q \grave{M}$ follows.
\end{proof}

Lemma \ref{lem:gradMFloquet} leads to the following formula for the $L^2-$gradient of $\Dl(\lm,v)$ with respect to $v$.

\begin{lem} \label{lem:DlGradientFloquet}
Let $v = (q, p) \in\Hp$ and assume that $\lm\in\C^*$ is not a simple periodic eigenvalue of $Q(v)$. 
Then the $L^2-$gradient $\partial_v \Dl = (\partial_q \Dl, \partial_p \Dl)$ of $\Dl = \Dl(\lm)$ is given by
\begin{align*}
 \partial_q\Dl =&  \frac{\xi_+-\xi_-}{4a_+^\bot a_-}\left(
      \lm f_+\star f_-
     +\frac{1}{16\,\lm}\left(e^{-q}f_1^+f_1^-- e^qf_2^+f_2^-\right)\right)
     \\
 \partial_p\Dl =& \frac{\xi_+-\xi_-}{8a_+^\bot  a_-} f_+\cdot Zf_-\Op(\cdot)
  .
\end{align*}
\end{lem}
The following normalization of the Floquet solutions $f_+,$ $f_-$ turn out to be useful to simplify the formulas for $\partial_v \Dl$, obtained in Lemma \ref{lem:DlGradientFloquet}.

\begin{lem} \label{lem:Eigenvectors} 
Let $v \in\Hp$ and assume that $\lm\in\C^*$ is not a periodic eigenvalue of $Q(v)$. If $\grave{m}_2(\lm, v)\not=0$ (meaning that $\lm$ is not a Dirichlet eigenvalue of $Q(v)$),  then
\begin{equation}\label{eq:EigenvectorsM2}
 a_\pm :=\begin{pmatrix} \grave{m}_2 \\ \xi_\pm -\grave{m}_1 \end{pmatrix}
\end{equation}
are eigenvectors of $\grave{M}(\lm, v)$, corresponding to the eigenvalues $\xi_\pm$,
and  \eqref{eq:mEigenvalue}  yields the identity $\xi_\pm-\grave{m}_1 = -\delta \pm\sqrt{\Dl^2-1}$, implying together with \eqref{eq:EigenvectorsM2} that
\[
 f_\pm= Ma_\pm = \grave{m}_2M_1 -\delta\cdot M_2 \pm \sqrt{\Dl^2-1}M_2.
\]
If $\grave{m}_3(\lm, v)\not= 0$,
then
\begin{equation}\label{eq:EigenvectorsM3}
 a_\pm :=\begin{pmatrix} \xi_\pm -\grave{m}_4  \\ \grave{m}_3 \end{pmatrix}
\end{equation}
are eigenvectors of $\grave{M}(\lm, v)$, corresponding to the eigenvalues $\xi_\pm$,
and  \eqref{eq:mEigenvalue} yields the identity $\xi_\pm-\grave{m}_4 = \delta \pm\sqrt{\Dl^2-1}$, implying together with \eqref{eq:EigenvectorsM3} that
\[
 f_\pm = Ma_\pm = \grave{m}_3M_2 +\delta \cdot M_1 \pm \sqrt{\Dl^2-1}M_1 .
\]
\end{lem}
Writing $\Grad{f}{\lm}$, defined in \eqref{eq:Grad}, in the  form
\begin{equation}\label{eq:Graddeux}
\Grad{f}{\lm} = -\begin{pmatrix} \frac\lm2 f\star f+ \frac{1}{32\,\lm} f\cdot \begin{pmatrix}e^{-q} &\\&-e^q\end{pmatrix} f\\ \frac14 f\cdot Zf \Op(\cdot)\end{pmatrix} \, ,
\end{equation}
one deduces from Lemma \ref{lem:DlGradientFloquet} and Lemma \ref{lem:Eigenvectors} the following formulas.

\begin{prop} \label{prop:DlGradientFloquet}
Let $v\in \Hp[1]$ and assume that $\lm\in\C^*$ is so that $\grave{M}(\lm, v) \ne \pm I$ (meaning that $\lm$ is not a periodic eigenvalue of $Q(v)$ of geometric multiplicity two). 
If $\grave{m}_2(\lm)\not=0$, then
\[
 \partial \Dl = \frac1{2\,\grave{m}_2} \Grad{\grave{m}_2M_1-\delta \cdot M_2}{\lm} -\frac{\Dl^2-1}{2\,\grave{m}_2}\Grad{M_2}{\lm} \, .
 \]
Similarly, if  $\grave{m}_3(\lm)\not=0$, then
\[
 \partial \Dl = -\frac1{2\,\grave{m}_3} \Grad{\grave{m}_3 M_2 +\delta \cdot M_1}{\lm} +\frac{\Dl^2-1}{2\,\grave{m}_3}\Grad{M_1}{\lm}.
\]
\end{prop}
\begin{proof}
 Assume that $\lm\in\C^*$ is not a periodic eigenvalue. If $\grave{m}_2(\lm)\not=0$, then \eqref{eq:EigenvectorsM2} yields $a_+^\bot  a_- = (\xi_--\xi_+)\grave{m}_2$ 
 and if  $\grave{m}_3(\lm)\not=0$, then \eqref{eq:EigenvectorsM3} yields $a_+^\bot a_- = (\xi_+-\xi_-)\grave{m}_3$. If  $\lm\in\C^*$ is a periodic eigenvalue 
 of geometric multiplicity $1$, then
 $\grave{m}_2(\lm)\not=0$ or $\grave{m}_3(\lm) \not=0$ holds. Then $\grave{m}_2$ or, respectively, $\grave{m}_3$  do not vanish in a  neighborhood $U$ of $\lm$. 
 The corresponding claimed identities then hold on $U\setminus\{\lm\}$ and hence by continuity also at $\lm$.
\end{proof}

\begin{prop} \label{lem:gradientPeriodic}
 Let $v_0\in\Hp[1]$ and assume that $\lm(v_0)\in\C^*$ is  a periodic eigenvalue of $Q(v_0)$ of algebraic multiplicity $1$, meaning that $\dot\Dl(\lm(v_0))\not=0$
 and that $\grave{m}_2(\lm(v_0),v_0)\not= 0$ or $\grave{m}_3(\lm(v_0),v_0)\not = 0$.\\
 Then there exist an open neighborhood $V$ of $v_0$ and 
 an analytic function $v \mapsto \lm(v)$ on $V$, which coincides at $v_0$ with $\lm(v_0)$, so that $\lm(v)$ is a periodic eigenvalue of $Q(v)$,  $\dot\Dl(\lm(v))\not=0$ on $V$,
 and $\grave{m}_2(\lm(v),v)\not= 0$ on $V$ or $\grave{m}_3(\lm(v_0),v_0)\not = 0$ on $V$. The $L^2-$gradient of $\lm(v)$ satisfies
 \[
 \partial \lm = -\frac{1}{\dot\Dl(\lm)}\partial\Dl\big|_{\lm}.
 \]
 Furthermore, if $\grave{m}_2(\lm(v),v)\not=0$, then $\grave{m}_2M_1 -\delta \cdot M_2$ is an eigenfunction for $\lm(v)$  and
 \[
 \partial \lm = -\frac{1}{2\dot \Dl(\lm)\grave{m}_2(\lm)} \Grad{\grave{m}_2M_1 -\delta \cdot M_2}{\lm} 
  \]
and if $\grave{m}_3(\lm(v),v)\not = 0$, then $\grave{m}_3M_2 +\delta \cdot M_1$ is an eigenfunction for $\lm(v)$ and
   \[
 \partial \lm = \frac{1}{2\dot \Dl(\lm)\grave{m}_3(\lm)} \Grad{\grave{m}_3M_2 - \delta \cdot M_1}{\lm}.
  \]
\end{prop}
\begin{proof}
If $\lm(v_0)$ is a periodic eigenvalue of $Q(v_0)$, then since $\det \grave M = 1$, the two eigenvalues of $\grave{M}(\lm(v_0), v_0)$ are both 
equal to $\sigma \in \{ 1 , -1 \}$. 
If both $\grave{m}_2(\lm(v_0),v_0)$ and $\grave{m}_3(\lm(v_0), v_0)$ were to vanish then $\grave{M}(\lm(v_0))= \sigma I$
and hence $\lm(v_0)$ would be a periodic eigenvalue of geometric multiplicity two which contradicts our assumption. 
Given that $\lm(v_0)$ is a periodic  eigenvalue of $Q(v_0)$ of algebraic multiplicity 1, one infers that there exist a neighborhood $V$ of $v_0$
and an analytic function $v \mapsto \lm(v)$ on $V$, which coincides at $v_0$ with $\lm(v_0)$, so that for any $v \in V$, $\lm(v)$ is a periodic eigenvalue of $Q(v)$ of algebraic multiplicity 1 
with $\Dl(\lm(v),v)=\sigma$.
Applying the chain rule to the left hand side of the identity $\Dl(\lm(v),v)=\sigma$, one gets
\[
\partial \lm =  -\frac{1}{\dot\Dl(\lm)}\left.\partial \Dl\right|_\lm.
\]
Since $\Dl^2(\lm(v),v)= 1$, the claimed formulas for $\partial \lm$ then follow from Proposition \ref{prop:DlGradientFloquet}.
\end{proof}

Finally, based on the formulas for the gradient of the Floquet matrix of Lemma \ref{lem:gradMFloquet} and the formulas of the Floquet solutions 
of Lemma \ref{lem:Eigenvectors} we provide a formula for $\partial \grave{m}_4(\lm)$ at a Dirichlet eigenvalue which turns out to be useful in the sequel.

\begin{lem} \label{lem:GradM4}
For any $v\in \Hp$ in some open subset $V$ of \Hp, let
$\mu(v) \in \C^*$ be a  Dirichlet eigenvalue of $Q(v)$ of algebraic multiplicity 1, depending analytically on $v$. Then
\[
 \partial \grave{m}_4\big|_{\lm=\mu} =- \grave{m}_3(\mu)\Grad{M_2}{\mu}  + \frac{\grave{m}_4(\mu)}{4}\Big(\Grad{M_1 +M_2}{\mu} - \Grad{M_1- M_2}{\mu}\Big).
 \]
\end{lem}

\begin{proof}
First let us consider the case where $\grave{m}_3(\mu)\not=0$ and $\dl(\mu) = (\grave{m}_1(\mu) - \grave{m}_4(\mu))/2 \not=0$ in some open subset of $V$. 
It then follows that $\mu$ is not a periodic eigenvalue. Since $\grave{m}_2(\mu)=0$, $\grave{M}(\mu)$ is lower triangular 
and the two Floquet multipliers are $\xi_+ = \grave{m}_1(\mu)$ and $\xi_- = \grave{m}_4(\mu)$. 
(Here we have chosen the square root $\sqrt{\Dl^2(\mu)-1}$ to be given by $\dl(\mu)$.) By \eqref{eq:EigenvectorsM3},
eigenvectors corresponding to $\xi_+$ and respectively, $\xi_-$ are given by
\begin{equation}\label{eq:kap4.171}
a_+ =(2\dl(\mu),\grave{m}_3(\mu)), \quad  a_- =(0,\grave{m}_3(\mu)).
\end{equation}
Since $\xi_+-\xi_- = 2\dl(\mu) \not=0$ by assumption, $a_+$ and $a_-$ are linearly independent and with $a_+^\bot = (-a_2^+,a_1^+) = (-\grave{m}_3,2\dl)$ one gets
\begin{equation}\label{eq:kap4.172}
a_+^\bot a_- = (-\grave{m}_3 ,2\dl) \cdot (0,\grave{m}_3) = 2\dl\grave{m}_3
\end{equation}
and the inverse of \begin{small} $\begin{pmatrix}a_+ & a_-\end{pmatrix} = \begin{pmatrix} 2\dl&0 \\ \grave{m}_3 & \grave{m}_3\end{pmatrix}\in \C^{2\times 2} $\end{small} is given by
\begin{equation}\label{eq:kap4.173}
\begin{pmatrix} a_+ & a_-\end{pmatrix}^{-1} = \frac{1}{2\dl\grave{m}_3} \begin{pmatrix} \grave{m}_3 & 0 \\ -\grave{m}_3 & 2\dl\end{pmatrix}.
\end{equation}
Furthermore by Lemma \ref{lem:Eigenvectors}, the corresponding Floquet solutions $f_{\pm} = (f_1^\pm, f_2^\pm)$ of $Q(v)$ for $\lm=\mu$ are given by
\begin{equation}\label{eq:kap4.174}
 f_+(x,\mu) = 2\dl(\mu) M_1(x,\mu) + \grave{m}_3(\mu) M_2(x,\mu), \quad f_-(x,\mu) = \grave{m}_3(\mu) M_2(x,\mu).
\end{equation}
By Lemma \ref{lem:gradMFloquet} and \eqref{eq:kap4.171} - \eqref{eq:kap4.173} one has
\begin{equation}\label{eq:kap4.175}
\partial_p \grave{M} =\frac{1}{8\dl\grave{m}_3 } \left(\grave{m}_1 \begin{pmatrix} 2\dl\\ \grave{m}_3\end{pmatrix} \;\; \grave{m}_4 \begin{pmatrix} 0\\ \grave{m}_3\end{pmatrix}\right) \begin{pmatrix} f_+ \cdot Zf_- & f_- \cdot Zf_-\\ -f_+\cdot Z f_+ & -f_+ \cdot Zf_- \end{pmatrix}
\frac{1}{2\dl\grave{m}_3} \begin{pmatrix} \grave{m}_3 & 0 \\ -\grave{m}_3 & 2\dl\end{pmatrix} \Op (\cdot)
\end{equation}
Since for any $2\times 2$ matrix \begin{small} $\begin{pmatrix} b_1&b_2\\ b_3 & b_4 \end{pmatrix} \in\C^{2\times 2}$\end{small},
\[
\left(\grave{m}_1 \begin{pmatrix} 2\dl\\ \grave{m}_3\end{pmatrix} \;\; \grave{m}_4 \begin{pmatrix} 0\\ \grave{m}_3\end{pmatrix}\right)
\begin{pmatrix} b_1&b_2\\ b_3 & b_4 \end{pmatrix}
\begin{pmatrix} \grave{m}_3 & 0 \\ -\grave{m}_3 & 2\dl\end{pmatrix}
= \begin{pmatrix} * & * \\ * & 2\dl\grave{m}_3 (b_2 \grave{m}_1 + b_4 \grave{m}_4)\end{pmatrix}
\]
it follows from \eqref{eq:kap4.175} that
\[
\partial_p \grave{m}_4(\lm)\big|_{\lm=\mu} = \frac{1}{8\dl\grave{m}_3} (\grave{m}_1 f_- \cdot Zf_- - \grave{m}_4 f_+ \cdot Z f_- ) \Op ( \cdot).
\]
By \eqref{eq:kap4.174}, one has
\[
f_-\cdot Z f_- = \grave{m}_3^2 M_2 \cdot Z M_2, \quad f_+ \cdot Z f_- = \grave{m}_3 ( 2\dl M_1 + \grave{m}_3 M_2) \cdot Z M_2
\]
and hence
\[
\partial_p \grave{m}_4(\lm)\big|_{\lm=\mu} = \frac{1}{8\dl\grave{m}_3 } \big((\grave{m}_1-\grave{m}_4)\grave{m}_3^2 M_2 \cdot Z M_2 - \grave{m}_4 \grave{m}_3 2\dl M_1 \cdot Z M_2 \big) \Op (\cdot)
= \frac14 (\grave{m}_3 M_2 \cdot Z M_2 - \grave{m}_4 M_1 \cdot Z M_2) \Op(\cdot).
\]
Now let us turn to the computation of $\partial_q \grave{m}_4 \big|_{\lm=\mu}$. By Lemma \ref{lem:gradMFloquet} and \eqref{eq:kap4.171}-\eqref{eq:kap4.173}
\begin{align*}
\partial_q \grave{m}_4\big|_{\lm=\mu} =& \frac{\mu}{4\dl \grave{m}_3} \big(\grave{m}_1 f_-\star f_- - \grave{m}_4 f_+ \star f_-\big)
\\ &+ \frac{e^q}{64 \mu\dl\grave{m}_3} \big( -\grave{m}_1 (f_2^-)^2 + \grave{m}_4 f_2^+ f_2^-\big) + \frac{e^{-q}}{64\mu\dl\grave{m}_3} \big(\grave{m}_1 (f_1^-)^2 - \grave{m}_4 f_1^+f_1^-\big)
\\=& \frac{\mu}{2} \big(\grave{m}_3 M_2\star M_2 - \grave{m}_4 M_1\star M_2\big) + \frac{e^q}{32\mu} \big(\grave{m}_4 m_3m_4 - \grave{m}_3 m_4^2\big) - \frac{e^{-q}}{32\mu} \big(\grave{m}_4m_1m_2 - \grave{m}_3 m_2^2\big).
\end{align*}
 Using that by \eqref{eq:Graddeux}
\[
-\grave{m}_3\Grad{M_2}{\mu} =\grave{m}_3 \begin{pmatrix} \frac{\mu}{2} M_2\star M_2 - \frac{e^q}{32\mu} m_4^2 + \frac{e^{-q}}{32\mu} m_2^2\\ \frac14 M_2\cdot ZM_2  \Op (\cdot)\end{pmatrix}
\]
and therefore
\[
 \partial \grave{m}_4(\lm) \big|_{\lm=\mu}  =- \grave{m}_3\Grad{M_2}{\mu} + \frac{\grave{m}_4}{4} \begin{pmatrix}- 2 \mu M_1\star M_2 - \frac{1}{8\,\mu} M_1\cdot \begin{pmatrix}e^{-q} &\\&-e^q\end{pmatrix} M_2\\ - M_1\cdot Z M_2 \Op(\cdot)\end{pmatrix}.
\]
Since by a straightforward computations, one has
\[
 \begin{pmatrix}- 2 \mu M_1\star M_2 - \frac{1}{8\,\mu} M_1\cdot \begin{pmatrix}e^{-q} &\\&-e^q\end{pmatrix} M_2\\ - M_1\cdot Z M_2 \Op(\cdot)\end{pmatrix} = \Grad{M_1+M_2}{\mu} - \Grad{M_1-M_2}{\mu}
\]
the claimed identity follows in the case where $\grave{m}_3(\mu)\not=0$ and $\dl(\mu)\not=0$. By continuity, the identity then holds for any $v\in V$.
\end{proof}

\subsection{Asymptotics of the $L^2-$gradient of the Floquet matrix}\label{asymptotics}
In this section we prove for any given $v \in \Hp$ asymptotics for the gradient $\partial \grave{M}(\zeta_n, v)$ as $n \to \pm \infty$ where  $(\zeta_n)_{n\in \Z}\subset\C^*$ is any bi-infinite sequence 
with $\zeta_n\sim n\pi$ as $n\to\pm \infty$. We introduce
\begin{equation}\label{eq:langlenrangle}
\langle n\rangle := \sqrt{1+n^2\pi^2}
\end{equation}
and recall that by \eqref{eq:Eomega}
\begin{equation*}
 E_{\omega(\lm)}(x) = \begin{pmatrix}
 \cos(\omega(\lm)x) & \sin(\omega(\lm)x)\\ -\sin(\omega(\lm)x) & \cos(\omega(\lm)x)
   \end{pmatrix}, \qquad \omega(\lm) = \lm-\frac1{16\,\lm}.
\end{equation*}
Furthermore, for any sequence $(a_n)_{n \in \Z} \subset \C$ we write $a_n = \ell^2_n$ if $(a_n)_{n \in \Z}  \in \ell^2(\Z, \C)$.
In particular, $\zeta_n = n\pi + \ell^2_n$ means that $(\zeta_n - n \pi)_{n \in \Z}$ is a sequence in $\ell^2(\Z, \C)$.

\begin{lem}\label{lem:NgradientEstimate}
For any $v\in\Hp$ and any bi-infinite sequence $(\zeta_n)_{n\in\Z}$ of complex numbers in $\C^*$ with $|\zeta_n|\geq1/4$ the following
asymptotics of the $L^2-$gradient of $\grave{M}(\zeta_n) \equiv \grave{M}(\zeta_n, v)$ hold:\\
(i) If $\zeta_n =n\pi +O(1)$, then
\begin{align}\begin{split}
\partial_q \grave{M}(\zeta_n)
=& \frac{E_{\omega(\zeta_n)}(1)}{4}\begin{pmatrix}
                   \cos(2\omega(\zeta_n)x)  & \sin(2\omega(\zeta_n)x)\\
                   \sin(2\omega(\zeta_n)x) & -\cos(2\omega(\zeta_n)x)
                   \end{pmatrix}\cdot\partial_x (\cdot) + \begin{pmatrix}\ell_n^2&\ell_n^2\\\ell_n^2 & \ell_n^2\end{pmatrix}\cdot\partial_x(\cdot) + \begin{pmatrix}\ell_n^2&\ell_n^2\\\ell_n^2 & \ell_n^2\end{pmatrix} ,
\end{split}\label{eq:LemmdqM}
 \\
\partial_p \grave{M}(\zeta_n) =& \frac{E_{\omega(\zeta_n)}(1)}{4} \begin{pmatrix}
      \cos(2\omega(\zeta_n)x) &\sin(2\omega(\zeta_n)x)  \\ \sin(2\omega(\zeta_n)x)  & -\cos(2\omega(\zeta_n)x)
      \end{pmatrix} \cdot \Op (\cdot) + \begin{pmatrix}\ell_n^2&\ell_n^2\\\ell_n^2 & \ell_n^2\end{pmatrix}\cdot\Op(\cdot) .\label{eq:LemmdpM} 
\end{align}
Alternatively, the asymptotics of $\partial_q\grave{M}(\zeta_n)$ can be written as
\begin{align}\label{eq:LemmdqMTwo}\begin{split}
\partial_q\grave{M}(\zeta_n) =& \frac{1}{2}\begin{pmatrix} &\sin(\omega(\zeta_n))+\ell_n^2\\ \sin(\omega(\zeta_n)) + \ell_n^2 & \end{pmatrix}EV_0
\\&\indent - \frac{\zeta_nE_{\omega(\zeta_n)}(1)}{2}\begin{pmatrix}
 -\sin(2\omega(\zeta_n)x) & \cos(2\omega(\zeta_n)x) \\ \cos(2\omega(\zeta_n)x) & \sin(2\omega(\zeta_n)x)
 \end{pmatrix} +\langle n\rangle \begin{pmatrix}\ell_n^2&\ell_n^2\\\ell_n^2 & \ell_n^2\end{pmatrix} .
\end{split}\end{align}
These estimates hold uniformly for $x$ in the interval $[0, 1]$, for $v$ in any bounded subset of $\Hp$, and for $(\zeta_n)_{n \in \Z}$ in any subset of bi-infinite sequences $(\zeta_n)_{n\in \Z}$ with $\sup_{n\in\Z}|\zeta_n-n\pi|$ bounded.
In more detail,  e.g.  \begin{small}$\begin{pmatrix} \ell_n^2 & \ell_n^2\\ \ell_n^2 & \ell_n^2 \end{pmatrix}\cdot \partial_x(\cdot)$\end{small} 
is a sequence of the form \begin{small}$\begin{pmatrix} a_{1n}(x) & a_{2n}(x)\\ a_{3n}(x) & a_{4n}(x)\end{pmatrix}\cdot \partial_x(\cdot)$\end{small} 
with the property that
for any $\mathring{q}\in\Hp$,  \begin{small}$\begin{pmatrix} \ell_n^2 & \ell_n^2\\ \ell_n^2 & \ell_n^2 \end{pmatrix}\cdot \partial_x(\mathring{q})$\end{small}  
is given by
  \[
  \begin{pmatrix} \langle a_{1n},\partial_x\mathring{q}\rangle_r
    & \langle a_{2n},\partial_x\mathring{q}\rangle_r
    \\ \langle a_{3n},\partial_x\mathring{q}\rangle_r
    & \langle a_{4n},\partial_x\mathring{q}\rangle_r
  \end{pmatrix}.
  \]
The claimed uniformity statement for  \begin{small}$\begin{pmatrix} \ell_n^2 & \ell_n^2\\ \ell_n^2 & \ell_n^2 \end{pmatrix}\cdot \partial_x(\cdot)$\end{small}  means that
it satisfies the estimate
  \[
 \sup_{1\leq j\leq 4} \, \sup_{0\leq x\leq 1} \sum_{n\in\Z} |a_{jn}(x)|^2 \leq C
  \]
where the constant $C$ can be chosen uniformly for $v$ in any given bounded subset in \Hp 
and uniformly for bi-infinite sequences $(\zeta_n)_{n\in \Z}$ with $\sup_{n\in\Z}|\zeta_n -n\pi|$ bounded. In particular, $\sum_{n\in\Z} \norm{a_{j,n}}^2 \leq C$ 
for any $1\leq j\leq 4$. \\
(ii) If $\zeta_n = n\pi + \ell_n^2$, then
\begin{align}
\partial_q \grave{M}(\zeta_n)
=&\frac{(-1)^n}4 \begin{pmatrix}
                   \cos(2\pi nx)  & \sin(2\pi nx)\\
                   \sin(2\pi nx) & -\cos(2\pi nx)
                   \end{pmatrix}\cdot\partial_x (\cdot) + \begin{pmatrix}\ell_n^2&\ell_n^2\\\ell_n^2 & \ell_n^2\end{pmatrix}\cdot\partial_x(\cdot) + \begin{pmatrix}\ell_n^2&\ell_n^2\\\ell_n^2 & \ell_n^2\end{pmatrix}
 \label{eq:lemmadqMell2}
\\
\partial_p \grave{M}(\zeta_n) =& \frac{(-1)^n}4 \begin{pmatrix}
      \cos(2\pi nx) &\sin(2\pi nx)  \\ \sin(2\pi nx) & -\cos(2\pi n x)
      \end{pmatrix} \cdot \Op (\cdot) + \begin{pmatrix}\ell_n^2&\ell_n^2\\\ell_n^2 & \ell_n^2\end{pmatrix}\cdot\Op(\cdot).\label{eq:lemmadpMell2}
\end{align}
Alternatively, $\partial_q\grave{M}(\zeta_n)$ satisfies the asymptotics
\begin{align}\label{eq:lemmadqMell2Two}
\partial_q \grave{M}(\zeta_n) =& \begin{pmatrix} &\ell_n^2\\  \ell_n^2 & \end{pmatrix}EV_0 + (-1)^{n+1}\frac{n\pi}{2}\begin{pmatrix}
 -\sin(2\pi n x) & \cos(\pi n x) \\ \cos(2\pi nx) & \sin(2\pi nx)
 \end{pmatrix}  + \langle n\rangle\begin{pmatrix}\ell_n^2&\ell_n^2\\\ell_n^2 & \ell_n^2\end{pmatrix}  .
\end{align}
These estimates hold uniformly for $x$ in the interval $[0, 1]$, for $v$ in any bounded subset of $\Hp$, 
and for $(\zeta_n)_{n \in \Z}$ in any subset of bi-infinite sequences $(\zeta_n)_{n\in \Z}$ with $\sup_{n\in\Z}|\zeta_n-n\pi|$ bounded.
\end{lem}
\begin{proof}
(i) Let $v \in \Hp$ and $\zeta_n = n\pi +O(1)$. By
 \cite[Theorem 2.12(iii), Corollary 2.13(iii)]{LOperator2018}, and $\omega(-\zeta_n) = -\omega(\zeta_n)$ one has
\[
M(x, \zeta_n)=  {E}_{\omega(\zeta_n)}(x) + \begin{pmatrix}\ell_n^2&\ell_n^2\\\ell_n^2 & \ell_n^2\end{pmatrix},
\quad
{M}^{-1}(x, \zeta_n) = {E}_{-\omega(\zeta_n)}(x) +\begin{pmatrix}\ell_n^2&\ell_n^2\\\ell_n^2 & \ell_n^2\end{pmatrix}
\]
uniformly for $0 \le x \le 1$.
Hence  \eqref{eq:prodMgradientOne} yields
\begin{align*}
\partial_q \grave{M} =&-\frac14 \grave{M} M^{-1}\ii RM \cdot\partial_x(\cdot)
  -\frac{1}{16\,\zeta_n}\grave{M}M^{-1} \begin{pmatrix}&e^q \\ e^{-q}\end{pmatrix} M
  \\=&-\frac14 E_{\omega(\zeta_n)}(1) E_{-\omega(\zeta_n)}(x)\ii RE_{\omega(\zeta_n)}(x)\cdot \partial_x(\cdot)+ \begin{pmatrix} \ell_n^2&\ell_n^2\\\ell_n^2 & \ell_n^2\end{pmatrix}\cdot\partial_x(\cdot)
  +\begin{pmatrix} \ell_n^2&\ell_n^2\\\ell_n^2 & \ell_n^2\end{pmatrix}.
\end{align*}
Since by trigonometric identities
\begin{align*}
  E_{-\omega(\zeta_n)}(x)\ii RE_{\omega(\zeta_n)}(x)
  =&\begin{pmatrix}
                   -\cos^2(\omega(\zeta_n)x)+\sin^2(\omega(\zeta_n)x)  & -2\cos(\omega(\zeta_n)x)\sin(\omega(\zeta_n)x)\\
                   -2\cos(\omega(\zeta_n)x)\sin(\omega(\zeta_n)x) & \cos^2(\omega(\zeta_n)x)-\sin2(\omega(\zeta_n)x)
                   \end{pmatrix}
\\=&\begin{pmatrix}
                   -\cos(2\omega(\zeta_n)x)  & -\sin(2\omega(\zeta_n)x)\\
                   -\sin(2\omega(\zeta_n)x) & \cos(2\omega(\zeta_n)x)
                   \end{pmatrix}
\end{align*}
estimate \eqref{eq:LemmdqM} follows. Alternatively by \eqref{eq:prodMgradientThree}
\begin{align*}
& \partial_q \grave{M}(\zeta_n) = \frac{1}{2}\begin{pmatrix} & \grave{m}_2\\- \grave{m}_3 &\end{pmatrix}EV_0 - \frac12 \grave{M}M^{-1}\Big(\zeta_n Z + \frac{1}{16\,\zeta_n}\begin{pmatrix}&e^q \\ e^{-q}\end{pmatrix}\Big) M
\\=&  \frac{1}{2}\begin{pmatrix} &\sin(\omega(\zeta_n))+\ell_n^2\\ \sin(\omega(\zeta_n)) + \ell_n^2 & \end{pmatrix}EV_0 - \frac{\zeta_nE_{\omega(\zeta_n)}(1)}{2}E_{-\omega(\zeta_n)}(x)ZE_{\omega(\zeta_n)}(x)
+\langle n\rangle\begin{pmatrix} \ell_n^2&\ell_n^2\\\ell_n^2 & \ell_n^2\end{pmatrix}.
\end{align*}
Clearly
\begin{align*}
E_{-\omega(\zeta_n)}(x)ZE_{\omega(\zeta_n)}(x)
=&\left(\begin{array}{rr}
-2 \, \cos\left(\omega(\zeta_n)x\right) \sin\left(\omega(\zeta_n)x\right) & \cos^2\left(\omega(\zeta_n)x\right) - \sin^2\left(\omega(\zeta_n)x\right) \\
\cos^2\left(\omega(\zeta_n)x\right) - \sin^2\left(\omega(\zeta_n)x\right) & 2 \, \cos\left(\omega(\zeta_n)x\right) \sin\left(\omega(\zeta_n)x\right)
\end{array}\right)
\\ =&\begin{pmatrix}
 -\sin(2\omega(\zeta_n)x) & \cos(2\omega(\zeta_n)x)\\ \cos(2\omega(\zeta_n)x) & \sin(2\omega(\zeta_n)x)
 \end{pmatrix}
\end{align*}
yielding the estimate \eqref{eq:LemmdqMTwo}.
 For $\partial_p \grave{M}$ we proceed similarly. By Proposition \ref{prop:Mgradient}
\begin{align*}
\partial_p \grave{M} =& -\frac\ii4  \grave{M}  M^{-1} R M \cdot \Op (\cdot)
\\ =& -\frac\ii4 E_{\omega(\zeta_n)}(1) E_{-\omega(\zeta_n)}(x)R E_{\omega(\zeta_n)}(x) \cdot \Op (\cdot) + \begin{pmatrix} \ell_n^2&\ell_n^2\\\ell_n^2 & \ell_n^2\end{pmatrix}\cdot \Op(\cdot),
\end{align*}
leading to the estimate \eqref{eq:LemmdpM}. Going through the arguments of the  proof one verifies that the claimed uniformity statements hold.

\noindent (ii) Using the estimates of \cite[Theorem 2.12(iv)]{LOperator2018}
and \cite[Corollary 2.13(iv)]{LOperator2018} instead of the ones of \cite[Theorem 2.12(iii)]{LOperator2018}
and \cite[Corollary 2.13(iii)]{LOperator2018}  in the proof above, one obtains estimates \eqref{eq:lemmadqMell2}, \eqref{eq:lemmadpMell2}, and \eqref{eq:lemmadqMell2Two}. 
Going through the arguments of the  proof one sees that the claimed uniformity statements hold.
\end{proof}

The asymptotics stated in Lemma \ref{lem:NgradientEstimate} lead to  asymptotic estimates of the $L^2-$gradients of $\Dl$ and $\dl$.
\begin{cor}\label{cor:DldlgradientEstimate}
For any $v\in\Hp$ and any bi-infinite sequence $(\zeta_n)_{n\in\Z}$ of complex numbers with $|\zeta_n|\geq1/4$ the following
asymptotics for $\partial \Delta(\zeta_n)$ and $\partial \delta(\zeta_n)$ hold:\\
(i) If $\zeta_n = n\pi +O(1)$, then
\[
(\partial_q \Dl,\partial_p \Dl)\Big|_{\lm=\zeta_n}
=\big(\ell_n^2\cdot \partial_x(\cdot) +\ell_n^2 ,\;\ell_n^2\cdot\Op(\cdot)\big) \, .
\]
Alternatively, one has $\partial_q\Dl(\zeta_n) = \langle n\rangle \ell_n^2 $. Similarly,
\begin{align*}
 (\partial_q \dl,\partial_p\dl)\Big|_{\lm=\zeta_n} =& \frac14
 \Big( \cos(\omega(\zeta_n)(1-2x))\cdot\partial_x(\cdot) +\ell_n^2\cdot\partial_x(\cdot)+\ell_n^2,
 \,\, \, \,  \cos(\omega(\zeta_n)(1-2x))\cdot\Op(\cdot) + \ell_n^2\cdot \Op(\cdot)\Big).
\end{align*}
Alternatively, the asymptotics of $\partial_q\dl(\zeta_n)$ is given by
\begin{align*}
 \partial_q\dl(\zeta_n)= -\frac{\zeta_n}2 \sin(\omega(\zeta_n)(1-2x)) + \langle n\rangle\ell_n^2.
\end{align*}
These estimates hold uniformly in $0\leq x\leq 1$, on bounded subsets of  $\Hp$, and on subsets of bi-infinite sequences $(\zeta_n)_n$ with $\sup_{n\in\Z}|\zeta_n-n\pi|$ bounded.\\
(ii) If  $\zeta_n = n\pi + \ell_n^2$, then
\begin{align*}
 (\partial_q \dl,\partial_p\dl) \Big|_{\lm=\zeta_n}
 =&\frac{(-1)^n}{4} \Big( \cos(2n\pi x)\cdot \partial_x(\cdot)+\ell_n^2\cdot\partial_x(\cdot)+\ell_n^2,
 \,\,\,\,  \cos(2n\pi x)\cdot\Op(\cdot) + \ell_n^2\cdot\Op(\cdot)\Big).
\end{align*}
Alternatively, the asymptotics of $\partial_q\dl(\zeta_n)$ is given by
\begin{align*}
\partial_q\dl(\zeta_n)=& (-1)^n\frac{n\pi }2 \sin(2\pi nx) + \langle n\rangle\ell_n^2
 \end{align*}
uniformly in $0\leq x\leq 1$, on bounded subsets of  $\Hp$, and on subsets of bi-infinite sequences $(\zeta_n)_n$ with $(\zeta_n-n\pi)_n$ bounded in $\ell^2$.
\end{cor}
\begin{proof}
(i) By \eqref{eq:LemmdqM}
\
\begin{align*}
\partial_q \grave{M}(\zeta_n)
=& \frac{E_{\omega(\zeta_n)}(1)}{4}\begin{pmatrix}
                   \cos(2\omega(\zeta_n)x)  & \sin(2\omega(\zeta_n)x)\\
                   \sin(2\omega(\zeta_n)x) & -\cos(2\omega(\zeta_n)x)
                   \end{pmatrix}\cdot\partial_x (\cdot) + \begin{pmatrix}\ell_n^2&\ell_n^2\\\ell_n^2 & \ell_n^2\end{pmatrix}\cdot\partial_x(\cdot) + \begin{pmatrix}\ell_n^2&\ell_n^2\\\ell_n^2 & \ell_n^2\end{pmatrix}
\end{align*}
and by \eqref{eq:LemmdqMTwo} one has
\begin{align*}
\partial_q \grave{M}(\zeta_n)
=& \frac{1}{2}\begin{pmatrix} &\sin(\omega(\zeta_n))+\ell_n^2\\ \sin(\omega(\zeta_n)) + \ell_n^2 & \end{pmatrix}EV_0
\\&\indent - \frac{\zeta_nE_{\omega(\zeta_n)}(1)}{2} \begin{pmatrix}
 -\sin(2\omega(\zeta_n)x) & \cos(2\omega(\zeta_n)x) \\ \cos(2\omega(\zeta_n)x) & \sin(2\omega(\zeta_n)x)
 \end{pmatrix} +\langle n\rangle \begin{pmatrix}\ell_n^2&\ell_n^2\\\ell_n^2 & \ell_n^2\end{pmatrix}.
\end{align*}
Furthermore by \eqref{eq:LemmdpM}
\begin{align*}
\partial_p \grave{M}(\zeta_n) =& \frac{E_{\omega(\zeta_n)}(1)}{4} \begin{pmatrix}
      \cos(2\omega(\zeta_n)x) &\sin(2\omega(\zeta_n)x)  \\ \sin(2\omega(\zeta_n)x)  & -\cos(2\omega(\zeta_n)x)
      \end{pmatrix} \cdot \Op (\cdot) + \begin{pmatrix}\ell_n^2&\ell_n^2\\\ell_n^2 & \ell_n^2\end{pmatrix}\cdot\Op(\cdot).
\end{align*}
Since
\[
E_{\omega(\zeta_n)}(1)\begin{pmatrix}
 -\sin(2\omega(\zeta_n)x) & \cos(2\omega(\zeta_n)x) \\ \cos(2\omega(\zeta_n)x) & \sin(2\omega(\zeta_n)x)
 \end{pmatrix}
  =\begin{pmatrix}
     \sin(\omega(\zeta_n)(1-2x)) & \cos(\omega(\zeta_n)(1-2x))
     \\ \cos(\omega(\zeta_n)(1-2x)) & -\sin(\omega(\zeta_n)(1-2x) )
     \end{pmatrix}
 \]
 and
 \[
 E_{\omega(\zeta_n)}(1) \begin{pmatrix}
      \cos(2\omega(\zeta_n)x) &\sin(2\omega(\zeta_n)x)  \\ \sin(2\omega(\zeta_n)x)  & -\cos(2\omega(\zeta_n)x)
      \end{pmatrix}
   =\begin{pmatrix}
           \cos(\omega(\zeta_n)(1-2x)) & -\sin(\omega(\zeta_n)(1-2x)) \\ -\sin(\omega(\zeta_n)(1-2x)) & -\cos(\omega(\zeta_n)(1-2x))
          \end{pmatrix}
\]
the claimed estimates hold. The uniformity statements follow from Lemma \ref{lem:NgradientEstimate}. Item (ii) can be
proved by similar arguments.
\end{proof}
The last result of this section concerns asymptotics of the $L^2-$gradient of the Dirichlet eigenvalues.
\begin{lem}\label{lem:munGradientAsymptotic}
For any $v_0\in \Hp$ there exist a neighborhood $V$ of $v_0$ in \Hp  and  $N\geq 1$ such that for any $|n| > N$ and $v\in V$, $Q(v)$ has precisely one Dirichlet eigenvalue in $D_n$, denoted by $\mu_n$. For any
$|n| > N$, $\mu_n$ is analytic on $V$ and the following asymptotic estimates hold as $|n| \to\infty$
\begin{align}\label{eq:LemmaDmun}
(\partial_q \mu_n ,\partial_p \mu_n)
=& -\frac14 \Big(\sin(2n\pi x)\cdot\partial_x(\cdot) + \ell_n^2\cdot\partial_x(\cdot) + \ell_n^2 (\cdot),
 \,\,\,\, \sin(2n\pi x)\cdot \Op(\cdot) + \ell_n^2\cdot\Op(\cdot)\Big).
 \end{align}
 Integrating by parts, the asymptotics of $\partial_q\mu_n$ take the form
 \begin{align}\label{eq:LemmaDmunTwo}
\partial_q\mu_n =& \frac{n\pi }{2}\cos(2n\pi x) + \langle n\rangle \ell_n^2.
\end{align}
These estimates hold uniformly on the interval $0\leq x\leq 1$ and on bounded subsets of $V$. 
In more detail, e.g.  $\partial_q \mu_n - \frac{n\pi }{2}\cos(2n\pi x) = \ell_n^2\cdot \partial_x(\cdot)$ 
means that 
there exist functions $a_n(x)$, $|n|>N$, so that  $\ell_n^2\cdot \partial_x(\cdot) = a_n(x) \partial_x(\cdot)$ 
and
\[
\sup_{0\leq x\leq 1} \, \sum_{|n| > N} |a_n(x)|^2 \leq C
\]
where the constant $C > 0$ can be chosen uniformly on bounded subsets of $V$. 
\end{lem}

\begin{proof}
Choose $V$ and $N$ as in Counting Lemma  \cite[Lemma 3.4]{LOperator2018} for Dirichlet eigenvalues of $Q(v)$. 
Then $\mu_n$ is simple for $|n|>N$ and hence analytic (cf. Lemma \ref{lem:gradientDirichlet}).
Furthermore, by \cite[Lemma 3.16]{LOperator2018}, $\mu_n = n\pi +\ell_n^2$ as $n \to \infty$, implying that by 
\cite[Theorem 2.12(iv)]{LOperator2018} 
$$
 {M}(x,\mu_n,q,p) = E_{n\pi}(x) +\ell^2_n,
 \qquad \dot{ {M}}(x,\mu_n,q,p) = x  JE_{n\pi}(x)+\ell^2_n,
$$
as $n\to\infty$,
where by \eqref{eq:Eomega}
\begin{equation*}
 E_{n\pi}(x) = \begin{pmatrix}
 \cos(n\pi x) & \sin(n\pi x)\\ -\sin(n \pi x) & \cos(n\pi x)
   \end{pmatrix}.
   \end{equation*}
In particular, $\dot\chi_D(\mu_n) = \dot{\grave{m}}_2(\mu_n) = (-1)^n+\ell_n^2$ and $\grave{m}_1(\mu_n)= (-1)^n+\ell_n^2$. 
Hence by Proposition \ref{lem:gradientDirichlet}, for any $\mathring{v}=(\mathring{q},\mathring{p})\in \Hp$,
 \begin{align*}
\text{d} \mu_n[\mathring{v}] =&\frac{\grave{m}_1(\mu_n)}{\dot\chi_D(\mu_n)}\int_0^1\Grad{M_2}{\mu_n} \begin{pmatrix}
\mathring{q}\\\mathring{p}
\end{pmatrix}
\dx=\int_0^1\Grad{\begin{pmatrix}
\sin(n\pi x)+\ell_n^2 \\ \cos(n\pi x)+\ell_n^2\end{pmatrix}}{\mu_n}\begin{pmatrix}
\mathring{q}\\ \mathring{p}
\end{pmatrix}
\dx.
 \end{align*}
Since by definition \eqref{eq:Grad}
\begin{align*}
\Grad{\begin{pmatrix}
\sin(n\pi x)+\ell_n^2 \\ \cos(n\pi x)+\ell_n^2\end{pmatrix}}{\mu_n}
=&
\begin{pmatrix} \frac{\mu_n}2(\cos^2(n\pi x)-\sin^2(n\pi x) +\ell_n^2) + \frac{1}{32\mu_n}(\cos^2(n\pi x)e^q-\sin^2(n\pi x)e^{-q}) +\ell_n^2
\\ -\frac12  \sin(n\pi x)\cos(n\pi x)\cdot \Op(\cdot) +\ell_n^2\cdot\Op(\cdot)
\end{pmatrix}
\end{align*}
it then follows that
\begin{align*}
\ddh{\mu_n}=&\begin{pmatrix}
    \frac{n\pi}{2} \cos(2n\pi x) + \langle n\rangle\ell_n^2 \\ -\frac14\sin(2n\pi x)\cdot \Op(\cdot) + \ell_n^2\cdot\Op(\cdot)
    \end{pmatrix}
\end{align*}
yielding \eqref{eq:LemmaDmunTwo}. On the other hand, by Lemma \ref{lem:eigenfunctionGrad} and identity \eqref{identity for Dirichlet},
\begin{align*}
\ddh{\mu_n} =& \frac{\grave{m}_1(\mu_n)}{\dot\chi_D(\mu_n)}\int_0^1
\frac{1}{16\mu_n}(m_4^2e^q - m_2^2e^{-q})\mathring{q} -\frac12(\Op (\mathring{p}) +\partial_x \mathring{q})m_2m_4\dx
\\ =& \int_0^1\begin{pmatrix}\ell_n^2 -\frac12 \sin(n\pi x)\cos(n\pi x)\cdot\partial_x(\cdot) + \ell_n^2\cdot\partial_x(\cdot) + \ell_n^2
 \\ -\frac12 \sin(n\pi x)\cos(n\pi x)\cdot \Op(\cdot) + \ell_n^2\cdot\Op(\cdot)\end{pmatrix}\cdot
 \begin{pmatrix}
\mathring{q}\\\mathring{p}
\end{pmatrix}
\dx
 \\ =& -\frac14\int_0^1\begin{pmatrix}  \sin(2n\pi x)\cdot\partial_x(\cdot) + \ell_n^2\cdot\partial_x(\cdot) + \ell_n^2
 \\ \sin(2n\pi x)\cdot \Op(\cdot) + \ell_n^2\cdot\Op(\cdot)\end{pmatrix}\cdot
 \begin{pmatrix}
\mathring{q}\\\mathring{p}
\end{pmatrix}
\dx
\end{align*}
which proves \eqref{eq:LemmaDmun}.
Going through the arguments of the  proof one verifies that the claimed uniformity statements hold.
\end{proof}


 \section{Real and almost real potentials}\label{sec:realAmostReal}

In this section we consider the periodic and the Dirichlet spectrum of $Q(v)$ for potentials $v$ which are real or close to real ones. 
Recall that the subspace of real potentials in \Hp is denoted by \Hr. More generally, for any $m\geq 1$, let
 \[
 \Hr[m] := H^{m}(\T,\R)\times H^{m}(\T,\R)\subset \Hp \, .
 \]
By \eqref{sign Delta},
for any $v\in\Hp$ there exists $N \ge 1$ so that  $\Dl(\lm^\pm_{n}) = (-1)^n$ for $|n| > N$. But this identity does not necessarily hold for $|n| \le N$.
 In the case $v\in \Hr$, 
 the periodic spectrum of $Q(v)$ is real 
 and by \eqref{real ev} the listing \eqref{eq:listingPeriodic} of the periodic eigenvalues in \SpecDom is given by
 \[
 0<\cdots \leq \lm_{-1}^-\leq \lm_{-1}^+\leq \lm_0^-\leq \lm_0^+\leq \lm_1^-\leq \lm_1^+\leq \cdots .
 \]
 Actually, more can be said in this case.
 First we note that the periodic eigenvalues $\lm_n^+(v),\lm_n^-(v)$ $n\in\Z$, are continuous in $v$ on \Hr (cf. Proposition \ref{prop:spectralQuantCompact} below).
 Since for $v=0$, $\Dl(\lm_n^\pm)=(-1)^n$ one infers by deforming $v\in\Hr$ along the straight line $t\mapsto t\cdot{v}$ that for any $v$ in \Hr, 
 $\Dl(\lm_n^\pm(v), v)=(-1)^n$ for any $n\in\Z$. 
 Hence the periodic eigenvalues $\lm_n^\pm \equiv \lm_n^\pm(v)$, $n\in\Z$, of $Q(v)$ in \SpecDom satisfy
 \begin{equation}\label{real ev}
 0< \dots <\lm_{-1}^- \leq \lm_{-1}^+ < \lm_{0}^-\leq \lm_{0}^+ <\lm_1^- \leq \lm_1^+ <\cdots
 \end{equation}
 with strict inequality between $\lm_n^+$ and $\lm_{n+1}^-$  for any $n\in\Z$. Similarly, for any $v\in\Hr$, the Dirichlet spectrum of $Q(v)$ is real. It turns out that the Dirichlet eigenvalues in \SpecDom  
 are confined to the gaps:
 \begin{lem}\label{lem:realSpectrumOrder} For any $v\in\Hr$ and $n\in\Z$, the nth Dirichlet eigenvalue $\mu_n$ is real and satisfies $\lm_n^-\leq \mu_n\leq \lm_n^+$.
 Furthermore, $\mu_n$ is real analytic on \Hr.
 \end{lem}
 \begin{proof} Let $v\in\Hr$ and $n \in \Z$.  By \cite[Lemma 3.8]{LOperator2018}, $\mu_n \equiv \mu_n(v)$ satisfies $\Dl^2(\mu_n)-1 =\dl^2(\mu_n)\geq 0$ and therefore $|\Dl(\mu_n)|\geq 1$. 
 Since $\mu_n(0) =\lm^+_n(0)$, one concludes by deforming $v$ to $0$ along the straight line $t\mapsto t v$  that  $\lm_n^- \leq \mu_n \leq  \lm_n^+$. 
 In particular it follows that $\mu_n$ is simple and hence by  Proposition \ref{lem:gradientDirichlet}
 real analytic on \Hr.
 \end{proof}

Next let us consider the roots of $\dot\Dl$ in the case where $v\in\Hr$.
The case $v=0$ was already discussed in  Section \ref{setup}. In this case, all roots of $\dot\Dl(\cdot, 0)$ are simple
and the set of roots  in \SpecDom consists of a bi-infinite sequence
\begin{equation*}
         0 <  \cdots < \dot\lm_{-2}(0) <\dot\lm_{-1}(0) <\dot\lm_0(0) < \dot\lm_1(0) < \cdots
 \end{equation*}
 and the additional root $\dot\lm_*(0)=\frac{\ii}4$.  For arbitrary $v\in\Hr$ the following holds.

 \begin{lem}\label{lem:dotDlSpecOnHr}
  For any  $v\in \Hr$, the roots of $\dot\Dl$ in \SpecDom are all simple. The set of the roots in $\C^+$ consist of  a bi-infinite sequence of real numbers
  \[
   0<\dots <\dot\lm_{-2}(v)<\dot\lm_{-1}(v) < \dot\lm_0(v) < \dot\lm_1(v) < \dot\lm_2(v) <\cdots
  \]
 and one additional purely imaginary number, denoted by $\dot\lm_*(v)$, in $\ii \R_{>0}$. 
 Furthermore for any $n\in\Z$, the following holds: if $\lm_n^-=\lm_n^+$ , then $\dot\lm_n = \lm_n^+$ whereas if $\lm_n^- < \lm_n^+$,
  then $\lm_n^- < \dot\lm_n < \lm_n^+.$ 
 Furthermore, $\dot\lm_n$, $n\in\Z$, and $\dot\lm_*$ are real analytic on \Hr.
 \end{lem}

 \begin{proof}
 Let $v\in\Hr$.  By \cite[Lemma 2.14(i)]{LOperator2018}, $\Dl(-\lm,v)=\Dl(\lm,v)$ and by \cite[Lemma 2.14 (iii)]{LOperator2018},
 ${\overline{\Dl(\overline\lm, v)} }= {\Dl(\lm,v)}$. This implies that $\dot\Dl(-\lm,v) = -\dot\Dl(\lm,v)$ and $\overline{\dot\Dl(\overline\lm, v)} = {\dot\Dl(\lm,v)}$. 
 Hence in particular $\Dl(\lm,v)$  is real for $\lm\in\R$  and for any root 
 $\dot\lm$ of $\dot\Dl(\cdot, v)$ in $\C^*$, the complex numbers $-\dot\lm$, $\overline{\dot\lm}$, and $-\overline{\dot\lm}$  are also roots of $\dot\Dl(\cdot, v)$.
 Furthermore, since $\Dl^2(\lm,0) -1 < 0 $ for any $\lm_n^{+}(0) < \lm < \lm_{n+1}^{-}(0)$ and $\Dl^2(\lm,v) -1$ does not vanish in $\lm_n^{+}(v) < \lm < \lm_{n+1}^{-}(v)$
  it follows by continuity that  $\Dl^2(\lm,v) -1 < 0$ for any $\lm_n^{+}(v) < \lm < \lm_{n+1}^{-}(v)$. 
  By the same arguments one infers that $\Dl(\lambda^\pm_n) = (-1)^n$.
 If  $\lm_{n}^- = \lm_n^+$, then $\Dl^2-1$ has a root of multiplicity two at $\lm_n^+$. 
 Since $\Dl(\lm_n^+)\not=0$ it then follows that $\lm_n^+$ is also a root of $\dot\Dl$. If  $\lm_n^-\not=\lm_n^+$, both eigenvalues have algebraic multiplicity one. 
 By the considerations above, $-1<\Dl(\lm)<1$ for any $\lm_{n}^+<\lm<\lm_{n}^-$ and any
 $\lm_{n}^+<\lm<\lm_{n+1}^-$ and $\Dl(\lm_n^\pm)=(-1)^n$. Since $\lm_n^{\pm}$ are simple roots of $\Delta^2 -1$, one has $\dot\Dl(\lm_n^\pm) \ne 0$. Hence
 it follows that $(-1)^n\dot\Delta(\lambda_n^-) > 0$, $(-1)^n\dot\Delta(\lambda_n^+) < 0$, and $(-1)^n\Dl(\lm)>1$ for any $\lm_n^- < \lm < \lm_n^+$. 
 We conclude that $(-1)^n\Dl(\lm)$ has a maximum in the open interval $(\lm_n^-,\lm_n^+)$ and
 hence there exists $\dot\lm_n \in (\lm_n^-,\lm_n^+)$ with $\dot\Dl(\dot\lm_n) = 0$. 
 By the Counting Lemma \cite[Lemma 3.12]{LOperator2018},
 there exists $N\geq 1$ such that for any $n>N$, $\dot\Dl$ has exactly one simple root in each  of the domains  $D_n$, $-D_n$, $D_{-n}$, and $-D_{-n}$ 
 and  $4+4N$ roots in the annulus $A_N$. For any $n >N$ the following holds:
 since $G_n\subset D_n$ and $G_{-n}\subset D_{-n}$, the roots in $D_n$, $-D_n$, $D_{-n}$, 
 and $-D_{-n}$ are $\dot\lm_n,-\dot\lm_n,\dot\lm_{-n}$, and $-\dot\lm_{-n}$, respectively. Furthermore, the roots $\dot\lm_0,-\dot\lm_0$ and
the roots $\dot\lm_n,-\dot\lm_n,\dot\lm_{-n},-\dot\lm_{-n}$, $1\leq n\leq N$, are in $A_N$. 
 Hence according to the Counting Lemma, in addition to $\dot\lm_n,-\dot\lm_n$, $|n| \le N$, $\dot\Dl$ has two more roots in  $A_N$,
 which we denote by $\nu_1,\nu_2$. 
 By the previous discussion, $-\nu_1,-\nu_2$ as well as $\overline{\nu_1},\overline{\nu_2}$ are also roots of $\dot\Dl$. 
 The Counting Lemma then implies that $\nu_1=-\nu_2$ and $\nu_1\in (\R\cup \ii\R)\cap A_N \subset (\R\cup\ii\R)\setminus\{0\}$. 
 In particular it follows that $\nu_1\not=\nu_2$ and one of the two roots, denoted by $\dot\lm_*$, is in \SpecDom.  
 Since $\dot\lm_*(0) = \ii/4$, one then concludes by deforming $v\in\Hr$ along the straight line $t\mapsto t\cdot{v}$
 that for any $v \in \Hr$, $\dot\lm_*(v)$ is purely imaginary. 
  Taking into account that the closed intervals $G_n$, $n \in \Z$, are contained in the halfline $\R_{> 0}$ and pairwise disjoint,  
 we have shown in particular that for any $v \in \Hr$, all the roots of $\dot\Dl(\cdot,v)$
 are simple.
Since $\dot\Dl(\lm,v)$ is real analytic in $\C^*\times \Hr$, it then follows by the argument principle 
 that all the roots are real analytic.
 \end{proof}

 \begin{prop}\label{prop:spectralQuantCompact}
  The Dirichlet and the periodic eigenvalues of $Q(v)$ as well as the roots of $\dot\Dl(\cdot, v)$ are compact functions on $\Hr[1]$.
 \end{prop}
 \begin{proof}
   By Proposition 2.5 and Proposition 2.6 in \cite{LOperator2018},
   on any closed bounded subset of $ \C^* \times\Hr$, $M(1,\lm,v)$ and $\dot M(1,\lm,v)$ are compact functions and hence so are $\chi_p,\chi_D$, and $\dot\Dl$. 
  By Lemma \ref{lem:realSpectrumOrder}  and Lemma \ref{lem:dotDlSpecOnHr}, for any $v\in\Hr$, 
  the roots of $\chi_D(\cdot, v)$ and the ones of $\dot\Dl(\cdot, v)$ are simple and by \eqref{real ev}, the ones of $\chi_p(\cdot, v)$ are at most double. 
  Let us first consider the case of a potential $v_* \in \Hr$ with a double eigenvalue $ \lambda_n^-(v_*) = \lambda_n^+(v_*)$ for some $n \in \Z$. 
  Assume that $(v_k)_{k\geq 1}\subset \Hr$.
  To simplify notation introduce $f_k(\lm) := \chi_p(\lm, v_k)$, $f_*(\lm) := \chi_p(\lm,v_*)$. 
  Choose a disc $D$, centered at $ \lambda_n^+(v_*)$, of sufficiently small radius $\rho>0$ so that $\min_{\lm\in\partial D} |f_*(\lm)|\geq 2\epsilon$. 
  Since $\partial D$ is compact, $f_k,f_*$ are analytic in $\lm$, and $\lim_{k\to\infty} f_k(\lm) = f_*(\lm)$ for any $\lm\in \overline{D}$, 
  there exists $k_0\geq 1$ so that $\min_{\lm\in\partial D}|f_k(\lm)| \geq \epsilon$ for
 any $k\geq k_0$. It then follows from the argument principle that for any $k\geq k_0$,  $f_k(\lm)$ 
 has precisely two roots (counted with multiplicities) inside $D$, i.e., for $k$ sufficiently large
   \[
   \int_{\partial D} \frac{\dot f_k(\lm)}{f_k(\lm)} \dlm = \int_{\partial D} \frac{\dot f_*(\lm)}{ f_*(\lm)}\dlm = 2.
   \]
 Using the Counting Lemma and again the argument principle one sees that for $k$ sufficiently large these roots of $f_k(\lm)$ coincide with the two eigenvalues $\lm_n^+(v_k)$, $\lm_n^-(v_k)$.
 Since the radius $\rho>0$ of the disc $D$ can be chosen arbitrarily small it then follows that $\lim_{k\to\infty}\lm_n^\pm(v_k) = \lm_n^\pm(v_*)$.
 In a similar way, one shows a corresponding result in the case the eigenvalues $\lambda_n^\pm(v_*)$ of $v_*$ are simple. Altogether, this shows that the periodic
 eigenvalues of $Q(v)$ are compact functions on \Hr. The corresponding result for the Dirichlet eigenvalues of $Q(v)$ and for the roots of $\dot\Dl(\cdot, v)$
 are proved in a similar way.
 \end{proof}


Next we consider potentials in some complex neighborhood of \Hr in \Hp, ideally chosen in such a way that the Dirichlet eigenvalues of $Q(v)$ and the roots of $\dot\Dl(\cdot, v)$ 
remain simple and can be localized sufficiently accurately.

 First we need to discuss how to list the roots of $\dot\Dl$ in a convenient way.
 Since by Lemma \ref{lem:dotDlSpecOnHr} all the roots of $\dot\Dl$ are simple and real analytic  on $\Hr$ it follows from the argument principle  
 that any root of $\dot\Dl$ extends analytically to some complex neighborhood of  \Hr in \Hp as a  simple root of  $\dot\Dl$.  
 Furthermore recall that by the Counting Lemma \cite[Lemma 3.12]{LOperator2018}) (cf. \eqref{eq:dotLmListing}),  any potential $v_0$ in $\Hp$ admits
 a neighborhood $V$ in \Hp and $N\geq 1$ 
 so that for any $v \in V$, the roots of $\dot\Dl(\cdot, v)$ in $\SpecDom\setminus A_N$ can be listed as a bi-infinite sequence
 \[
  0 \preceq \dots\preceq \dot\lm_{-N-2} \preceq \dot\lm_{-N-1} \preceq \dot\lm_{N+1}\preceq \dot\lm_{N+2}\preceq \cdots
 \] with $\dot\lm_n\in D_n$ $\forall |n|>N$ and any root $\dot\lm$ of the $2N+2$ remaining ones in \SpecDom satisfies $\dot\lm_{-N-1} \preceq \dot\lm \preceq \dot\lm_{N+1}$. Hence for any element $v_0\in\Hr$ 
 there exists a complex neighborhood $V$ in $\Hp$ so that for any $v\in V$ the roots of $\dot\Dl(\lm)$ are simple and the root $\dot\lm_*(v)$ denotes the one obtained from $\dot\lm_*(v_0)$ by analytic deformation. Note that $\dot\lm_*(v)$ is not necessarily in $\SpecDom$.
 We list the roots which are different from $\dot\lm_*(v)$ and $-\dot\lm_*(v)$ and are contained in $\SpecDom$ as a bi-infinite sequence,
 \begin{equation}
   0\preceq \cdots \preceq \dot\lm_{-2}\preceq \dot\lm_{-1}\preceq \dot\lm_0 \preceq \dot\lm_1 \preceq \dot\lm_2\preceq \cdots ,
 \end{equation}
 so that
 \[
 \quad \dot\lm_n =n\pi +o(1), \quad \frac{1}{16\dot\lm_{-n}} = n\pi+o(1)\quad n\to\infty.
 \]
 By making the  neighborhood $V$ smaller if needed we can assume that the additional root $\dot\lm_*(v)$ satisfies
 \[
 \inf \set{\Im \dot\lm_*(v)}{v\in V} >0.
 \]
 
It turns out that at no cost, the notion of isolating neighborhoods, introduced in Section \ref{introduction},  can be slightly strengthened.
We say that a sequence $(U_n)_{n\in \Z}$ of pairwise disjoint open neighborhoods in \SpecDom together with an open disc $U_*\subset \set{\lm\in\C}{\Im\lm>0}$, 
centered on $\ii\R_{>0}$, form a set of isolating neighborhoods for a potential $v$ in $V$ if the following properties hold:
 \begin{itemize}
   \item[] (I-1) $\Gap \subset U_n\subset \SpecDom$ and $\mu_n,\dot\lm_n\in U_n$ for any $n\in \Z$;  $\dot\lm_*\in U_*$.
   \item[] (I-2) For any $n\geq 0$, $U_n$ is a disc centered on the real axis so that 
   \[
   c^{-1}|m-n|\leq dist(U_m,U_n)\leq c|m-n| \quad \forall m,n\geq 0, \, m\not=n
   \]
   for some constant $c\geq 1$.
   \item[] (I-3) The sets  $ \set{ \frac{1}{16\lm}}{\lm\in U_{-n}}$, $n\geq 0$, satisfy (I-2) with the same constant c.
   \item[] (I-4) For $|n|$ sufficiently large $U_n=D_n$.
   \item[] (I-5) $c^{-1}\leq \dist(U_*,U_n)$ $\forall n\in\Z$ with $c\geq 1$ given as in (I-2).
 \end{itemize}
 Note that for a potential $v\in V$, admitting a set of isolating neighborhoods, all the Dirichlet eigenvalues of $Q(v)$ and all the roots of $\dot\Dl(\cdot,v)$ are simple.
 \begin{lem} \label{lem:isolatingNeighbourhoodExtension}
   Let $(U_n)_{n\in\Z}$, $U_*$ be a set of isolating neighborhoods of $v_0$ in $\Hp$. Then there exists a neighborhood $V_{v_0}\subset V$ of $v_0$ in \Hp (with $V$ as above) 
   so that it is a set of isolating neighborhoods for any $v$ in $V_{v_0}$.
 \end{lem}
 \begin{proof}   Let $v_0$ be in $\Hp$. By the Counting Lemmas (Lemma 3.4, Lemma 3.11, and Lemma 3.12 in\cite{LOperator2018}) there exist an integer $N\geq 1$ and 
 a neighborhood ${V_{v_0}}\subset V$ of $v_0$ in \Hp so that for any $v$ in ${V_{v_0}}$,
   \[
   \Gap(v)\subset D_n, \quad \mu_n,\dot\lm_n\in D_n,\quad U_n = D_n, \qquad  |n| >  N.
   \]
   Clearly (I-2)-(I-5) are satisfied and so we only need to verify that (I-1) holds on $V_{v_0}$, possibly after shrinking it.
   It remains to control finitely many spectral quantities. Since $\chi_D$ is  analytic in $\lm$ and $v$ and since
   the $U_n$'s are isolating neighborhoods for $v_0$, $\chi(\cdot, v_0)$ does not vanish on $\partial U_n$, 
   one has, possibly after shrinking $V_{v_0}$, if needed, that for any $v\in V_{v_0}$ and $|n| \leq N$
   \begin{align}
         |\chi_D(\lm,v)-\chi_D(\lm,v_0)| &<|\chi_D(\lm,v_0)| \quad \forall \lm\in\partial U_n.
 \end{align}
 Hence by Rouch\'e's theorem, for any $|n| \leq N$, $\chi_D(\cdot,v)$ has the same number of roots inside $U_n$ as $\chi_D(\cdot,v_0)$. 
 Similarly one argues for $\chi_p$ and $\dot\Dl$, implying altogether that (I-1) holds for $v$ in $V_{v_0}$, possibly after shrinking it once more if needed.
 \end{proof}

 For any $v \in \Hr$, one has for any $n\in\Z$, $\lm_n^+<\lm_{n+1}^-$ (cf. \eqref{real ev}), $\lm_n^-\leq \dot\lm_n\leq\lm_n^+$, $\dot\lm_*\in\ii\R_{>0}$ (Lemma \ref{lem:dotDlSpecOnHr}) 
 and   $\lm_n^-\leq \mu_n\leq\lm_n^+$ (Lemma \ref{lem:realSpectrumOrder}) and hence a set of isolating neighborhoods always exists. 
 By  Lemma \ref{lem:isolatingNeighbourhoodExtension} it follows that for any $v\in\Hr$, there exists  a ball $V_{v} \subset V$ in \Hp, centered at $v$, and pairwise disjoint subsets $(U_n)_{n\in\Z}$, $U_*$ 
 so that they are isolating neighborhoods for any potential in $V_{v}$. Setting
 \begin{equation}\label{eq:DefnOfhatW}
 \hat W := \bigcup_{v\in\Hr}V_{v},
 \end{equation}
 we thus obtain an open connected neighborhood of \Hr in \Hp. Without loss of generality  we can assume that $v\in\hat W$ if and only if $(-q,p)\in\hat W$.

 \begin{lem}\label{lem:DirichletSpecReciprocity}
 On $\hat W$, the Dirichlet eigenvalues $\mu_n$, $n\in\Z$,  are real analytic functions of $q$ and $p$. For $v=(q,p)\in \hat W$ one has
 \begin{equation}\label{eq:DirichletReciprocity}
   \frac{1}{16\mu_n(q,p)} = \mu_{-n}(-q,p) \quad \forall n\in\Z.
 \end{equation}
 \end{lem}
 \begin{proof}
   The real analyticity follows by  Proposition \ref{lem:gradientDirichlet} and the fact that all Dirichlet eigenvalues are simple on $\hat W$. 
   Concerning \eqref{eq:DirichletReciprocity} recall that the Dirichlet eigenvalues are the zeros of $\grave{m}_2(\lm,q,p)$ and that by \cite[Proposition 2.1]{LOperator2018}, $\grave{m}_2(\frac{1}{16\lm},q,p) = -e^{-q(0)}\grave{m}_2(\lm,-q,p)$. Hence for any $\mu\in \C^*$
   \begin{equation} \label{eq:DirichletReciprocityProof}
     \mu \in \mathrm{spec}_{dir} Q(q,p) \quad \Leftrightarrow \quad \frac{1}{16\,\mu} \in \mathrm{spec}_{dir}Q(-q,p).
 \end{equation}
   It remains to show that the listing of the Dirichlet eigenvalues of $Q(q,p)$ and $Q(-q,p)$ lead to \eqref{eq:DirichletReciprocity}.
   Choosing $N\geq 1$ as in the Counting Lemma \cite[Lemma 3.4]{LOperator2018} and noting that by the definition of $D_n$, for any $\lm\in D_{-n}$ with $n\geq1$ one has $\frac{1}{16\lm} \in D_{n}$, 
   it follows that
   \begin{equation}\label{eq:DirichletReciprocityProof2}
     \frac{1}{16\mu_n(q,p)} = \mu_{-n}(-q,p) \quad \forall  |n| > N.
   \end{equation}
   For the finitely many Dirichlet eigenvalues $\mu_n$ with $|n|\leq N$, note that by the definition of isolating neighborhoods,  one has $\Re \mu_n(q,p)>0$ for any $n$.
   Furthermore, since by the definition of $a \preceq b$ 
   with $a,b\in\set{\lm}{\Re\lm>0}$,
   \[
      a\preceq b \quad \Leftrightarrow\quad\frac{1}{16b} \preceq \frac{1}{16 a}
   \]
   one has
   \[
   \frac{1}{16\,\mu_{N}(q,p)} \preceq \frac{1}{16\,\mu_{N-1}(q,p)} \preceq \cdots \preceq \frac{1}{16\,\mu_{-(N-1)}(q,p)} \preceq \frac{1}{16\,\mu_{-N}(q,p)}
   \]
 Since
 \[
 \mu_{-N}(-q,p)\preceq \mu_{-(N-1)}(-q,p) \preceq \cdots \preceq \mu_N(-q,p)
 \]
 it then follows from \eqref{eq:DirichletReciprocityProof2}, that also for any $|n|\leq N$, one has $\frac{1}{16\mu_n(q,p)} = \mu_{-n}(-q,p)$.
 \end{proof}

 \begin{lem}\label{lem:ReciprocityOfdotLmAndPeriodic}
   \begin{thmenum}
 \item On $\hat W$, $\dot\lm_n$, $n\in\Z$, and $\ii\dot\lm_*$ are real analytic functions.  For  $(q,p)\in\hat W$ one has
 \[
  \frac{1}{16\,\dot\lm_* (q,p)} = -\dot\lm_{*}(-q,p)
  \quad \text{and} \quad \frac{1}{16\,\dot\lm_n(q,p)} = \dot\lm_{-n}(-q,p).
 \]
 \item
 For any $(q,p)\in\hat W$,
 \[
  \frac{1}{16\,\lm_n^\pm(q,p)} = \lm_{-n}^\mp (-q,p) \quad \forall n\in\Z.
 \]
 \end{thmenum}
 \end{lem}
 \begin{proof}
 One argues as in the proof of Lemma \ref{lem:DirichletSpecReciprocity}. We only remark that for $\dot\lm_*(q,p)$ the identity  $\frac{1}{16\dot\lm_* (q,p)} = -\dot\lm_*(-q,p)$ holds since among the two roots $\dot\lm_*(-q,p)$ and $-\dot\lm_*(-q,p)$, the root $\dot\lm_*(-q,p)$ is characterized by $\Im\dot\lm_*(-q,p)>0$.
 \end{proof}

 We now analyze the following quantities in more detail,
 \begin{equation}\label{eq:taungmn}
   \tau_n = \frac{\lm_n^+ + \lm_n^-}{2},\quad \gm = \lm_n^+ - \lm_n^-, \quad n\in\Z.
 \end{equation}
 First we need to establish the following auxiliary result.

 \begin{lem}\label{lem:analyticitySpectralquantities}
   For any $k\geq 1$ and $n\in\Z$ the functions  $(\lm_n^+)^k + (\lm_n^-)^k$ are real analytic on $\hat W$.
 \end{lem}
 \begin{proof}
   Let $v$ be in $\Hr$ and $(U_n)_{n\in\Z}$, $U_*$  be a set of isolating neighborhoods for the neighborhood $V_{v}$, given as in \eqref{eq:DefnOfhatW}. 
   Then for any $k\geq 1$  and $n\in\Z$,  it follows from the argument principle  that
   \begin{equation}\label{eq:lmnpkPluslmnmk}
   (\lm_n^+)^k + (\lm_n^-)^k = \frac{1}{2\pi\ii}\int_{\partial U_n} \lm^k\frac{2\Dl(\lm)\dot \Dl(\lm)}{\Dl^2(\lm)-1} \dlm.
   \end{equation}
   Since $\Dl$ and $\dot\Dl$ are analytic on $\C^*\times V_{v}$  and $\Dl^2 - 1$ does not vanish on $\partial U_n\times V_{v}$, 
   the right-hand side of the latter identity is analytic on $V_{v}$. Finally, if $v\in\Hr$  then $\lm_n^+,\lm_n^-\in\R$ for any $n\in\Z$.
 \end{proof}

 \begin{lem} \label{lem:analyticityandGradientOfTaun} For each $n\in\Z$,
   $\tau_n = (\lm_n^++\lm_n^-)/2$ and  $\gamma_n^2 = (\lm_n^+ -\lm_n^-)^2$
   define analytic functions  on $\hat W$. Furthermore $(\gm^2)_{n\geq 0}$, $(n^4\gm[-n]^2)_{n\geq 1} \in \ell^1$,
   \[
   \tau_n = n\pi + \ell_n^2,\quad \frac{1}{16\tau_{-n}} =n\pi+\ell_n^2
   \]
 and
   \[
  \partial \tau_n =  \big(\ell_n^2\cdot \partial_x(\cdot) +\ell_n^2,\,\, \ell_n^2\cdot \Op(\cdot)\big),
  \quad \partial (\gamma_n^2) = \big( \ell_n^2\cdot \partial_x(\cdot) +\ell_n^2, \,\, \ell_n^2\cdot \Op(\cdot)\big)
   \]
 uniformly for $0\leq x\leq 1$ and  locally uniformly on $\hat W$. 
 In more detail,  e.g. the term $\ell_n^2\cdot \partial_x(\cdot)$ in the expression for $\partial_q (\gamma_n^2))$ means that there exist functions $a_n(x)$, $n \in \Z$, so that
 the value of $\ell_n^2 \cdot \partial_x$ at $\mathring{q}$ is $\langle a_n,\partial_x \mathring{q} \rangle_r$ and 
  \[
    \sup_{0\leq x\leq 1} \sum_{n \in \Z} |a_n(x)|^2 \leq C
  \]
  for some $C>0$ which can be chosen locally uniformly on $\hat W$.
 \end{lem}

 \begin{proof}
  By Lemma \ref{lem:analyticitySpectralquantities}, $\tau_n$ is analytic on $\hat W$ as is $\gm^2$ since
   \[
   \gamma_n^2 = 2(\lm_n^+)^2+2(\lm_n^-)^2 -(\lm_n^++\lm_n^-)^2.
   \]
   Furthermore, by  \cite[Lemma 3.17]{LOperator2018}, $(\gm)_{n\geq 0}\in \ell^2$ and hence $(\gm^2)_{n\geq 0}\in \ell^1$ and for any $n \ge 1$, 
   $\gm[-n](q,p) =\left( \frac{1}{16\lm_n^-} - \frac{1}{16\lm_n^+}\right)\Big|_{(-q,p)}$, implying that $(16\gm[-n](q,p))^2 = \pfrac{\gm}{\lm_n^-\lm_n^+}^2\Big|_{(-q,p)}$ and hence
   \begin{equation} \label{eq:n4gm2inell1}
     (n^4\gm[-n]^2)_{n\geq1}\in\ell^1.
 \end{equation}
 Similarly, by  \cite[Lemma 3.17]{LOperator2018}, one has $\tau_n = n\pi+\ell_n^2$ and hence for $n\geq 1$ sufficiently large
   \[
   \frac{1}{16\tau_{-n}} = \frac{1}{16(\lm_{-n}^- + \gm[-n]/2)} = \frac{1}{16\,\lm_{-n}^-} \big(1+\gm[-n]/2\lm_{-n}^- \big)^{-1} = \big(n\pi + \ell_n^2\big)\big(1+O\pfrac{\gm[-n]}{n}\big).
   \]
 By \eqref{eq:n4gm2inell1} one then concludes that $\frac{1}{16\,\tau_{-n}} = n\pi +\ell_n^2$.
   To obtain the claimed estimates for the gradients note that for any given $v\in\hat W$  one has by \eqref{eq:lmnpkPluslmnmk}
   \[
   2 \tau_n = \frac{1}{2\pi \ii}\int_{\partial U_n} \lm\frac{2 \Dl(\lm)\dot \Dl(\lm)}{\Dl^2(\lm)-1} \dlm
   \]
   where $(U_n)_{n\in\Z}$, $U_*$ is a set of isolating neighborhoods for $v$ and 
   \[
   \frac{2 \Dl(\lm)\dot \Dl(\lm)}{\Dl^2(\lm)-1} = \partial_\lm \log(\Dl^2(\lm)-1)
   \]
   for some appropriate branch of the logarithm.  Note that in view of the regularity of $\Dl(\lm,v)$ in $\lm$ and $v$, $\partial$ and $\partial_\lm$ commute,
   implying that
   \begin{equation}\label{eq:taunGradient}
   2\partial\tau_n = \frac{1}{2\pi \ii}\int_{\partial U_n} \lm \partial_\lm\left( \frac{2\Dl(\lm)\partial \Dl(\lm)}{\Dl^2(\lm)-1}\right)  \dlm.
 \end{equation}
   Integrating by parts then yields
   \[
  2\partial\tau_n = -\frac{1}{2\pi \ii}\int_{\partial U_n} \frac{2\Dl(\lm)\partial \Dl(\lm)}{\Dl^2(\lm)-1} \dlm.
   \]
   On $\partial U_n$, $|\Dl^2(\lm)-1|$ is bounded away from $0$ (cf. \cite[Lemma 2.17]{LOperator2018}), 
   $|\Dl(\lm)|$ is bounded (\cite[Lemma 3.14]{LOperator2018}) and by Corollary \ref{cor:DldlgradientEstimate} one has
   \[
   \partial \Dl(\lm) =\big(\ell_n^2\cdot \partial_x(\cdot) +\ell_n^2,\,\, \ell_n^2\cdot \Op(\cdot) \big) 
   \]
   uniformly in $0\leq x\leq 1$, $\lm\in\partial U_n$, and  locally uniformly on $\hat W$. Altogether this proves the claimed asymptotics for $\partial \tau_n$. For $\partial (\gamma_n^2)$ we proceed similarly. Since $\gm^2 = 2(\lm_n^+)^2+2(\lm_n^-)^2 - (2\tau_n)^2$ one has in view of \eqref{eq:lmnpkPluslmnmk}
   \begin{align}\begin{split}\label{eq:partialgm2}
     \partial\gamma_n^2 =& 4\frac{1}{2\pi\ii}\int_{\partial U_n} \lm^2\partial_\lm\frac{\Dl(\lm)\partial\Dl(\lm)}{\Dl^2(\lm)-1}\dlm - 8\tau_n\partial\tau_n
     \\ =&-8\frac{1}{2\pi\ii}\int_{\partial U_n} (\lm-\tau_n)\frac{\Dl(\lm)\partial\Dl(\lm)}{\Dl^2(\lm)-1} \dlm .
   \end{split}\end{align}
 Arguing as above and using that $\lm-\tau_n = O(1)$ on $\partial U_n$ the claimed asymptotics follow. Going through the arguments of the proof one verifies the uniformity statements.
 \end{proof}

 To finish this section we show that on $\Hr$, the nonvanishing of $\gm = \lm_n^+- \lm_n^-$ is generic for any $n\in\Z$.
First we need to establish the following auxiliary result. Recall that  $\grave{M} = \begin{pmatrix} \grave{m}_1&\grave{m}_2 \\ \grave{m}_3 & \grave{m}_4\end{pmatrix}$.

\begin{lem}\label{lem:collapsedGapCharacterization}
For any potential $v\in\Hr$ the following statements are equivalent:
\begin{thmenum}
\item $\gm =0$.
\item $\grave{M}(\mu_n) = (-1)^n I$.
\item $\grave{m}_3(\mu_n) = 0$ and $\grave{m}_1(\mu_n) = (-1)^n$.
\end{thmenum}
\end{lem}
\begin{proof}
  If the $\gamma_n = 0$ then $\lm^-_n = \mu_n = \lm_n^+$, and $\grave{m}_2(\lm_n^+)=0$. Since $\lm_n^+$ has geometric multiplicity two, 
  $\trace\grave{M}(\lm_n^+) = (-1)^n 2$, and $\det \grave{M}(\lm_n^+) = 1$ it follows that $\grave{M}(\mu_n) = (-1)^n I$.
Hence (i) implies  (ii). Item (ii) clearly implies (iii). Finally given (iii) we have that $\grave{m}_2(\mu_n) =0$, $\grave{m}_3(\mu_n) = 0$, and $\grave{m}_1(\mu_n) = (-1)^n$. 
Therefore the Wronskian identity reads
  \[
  1 = \det \grave{M}(\mu_n) = (-1)^n\grave{m}_4(\mu_n)
  \]
and in turn $\grave{M}(\mu_n) = (-1)^n I$. This means that $\mu_n$ is a double periodic eigenvalue and hence by Lemma \ref{lem:realSpectrumOrder} one has $\mu_n = \lm_n^+ = \lm_n^-$.
\end{proof}

For any $n\in\Z$, introduce the set
\[
Z_n := \set{v\in\hat W}{\lm_n^-(v) = \lm_n^+(v)}.
\]

\begin{prop} \label{prop:ZnAnalytic}
For any $n\in\Z$, the following holds:
\begin{equivenum} 
\item $Z_n\cap \Hr$ is contained in a codimension 1 submanifold, implying that $\Hr \setminus Z_n$  is dense in \Hr.
\item $Z_n$ is a real analytic subvariety of $\hat W$.
\end{equivenum}
\end{prop}
\begin{proof}
\noindent (i) By Lemma \ref{lem:collapsedGapCharacterization}, $Z_n \cap \Hr \subset Y_n := \{ v \in \Hr \, : \,  \grave{m}_1(\mu_n(v), v) =(-1)^n \}$.
Since $\grave m_1( \lambda, v)$ (cf. \cite[Corollary 2.3]{LOperator2018}) and $\mu_n(v)$ (cf. Lemma \ref{lem:realSpectrumOrder}) 
are real analytic on  $ (\R \setminus \{0\}) \times \Hr$, 
it follows that $\grave{m}_1(\mu_n(v), v)$ is real analytic on $\Hr$.
We claim that for an open neighborhood $U$ of $Z_n \cap \Hr$ in \Hr, $Y_n\cap U$ is a codimension 1 submanifold of \Hr. Indeed, by the chain rule one has
for any $v \in Z_n \cap \Hr$
\[
\partial_p( \grave{m}_1(\mu_n(v)), v) = \dot{m}_1\Big|_{x=1}\partial_p \mu_n  + \partial_p \grave{m}_1 \Big|_{\mu_n} \, .
\]
Since $m_2(0, \mu_n(v), v) = 0$, it then follows by \eqref{eq:pfLemkappa} that
$\Op^{-1}(\partial_p\mu_n)\Big|_{x=0} =0$. (Here we use that $P$ is self-adjoint and that any $f \in H^1([0,1], \R)$ with $f(0)=0$ and $f(1)=0$ 
is an element in $H^1(\T, \R)$.)
By \eqref{eq:prodMgradientTwo} we then conclude that for $v\in Z_n\cap \Hr$,
\[
\Op^{-1}(\partial_p \grave{m}_1(\mu_n) )\Big|_{x=0} = \Op^{-1}(\partial_p \grave{m}_1) \Big|_{x=0,\lm =\mu_n} = \frac{1}{4}\grave{m}_1\Big|_{\lm =\mu_n}  = \frac14(-1)^n.
\]
Hence there is an open neighborhood $U$ of  $Z_n \cap \Hr$  in \Hr such that $Y_n\cap U$ is a codimension 1 submanifold of \Hr.
(ii) By Lemma \ref{lem:analyticityandGradientOfTaun}, $\gm^2$ is a real analytic function on $\hat W$, which by item (i) does not vanish identically.
\end{proof}


\section{Product representations}\label{sec:prod}
 In this section we establish product representations of the characteristic functions $\chi_p(\lm)$ and $\chi_D(\lm)$ as well as
 of the function $\dot\Dl(\lm)$, needed in the sequel.
 These product formulas require to record the roots of these functions in the way described in \eqref{eq:kap6.3.1} 
 and \eqref{eq:kap6.3.2} in Section \ref{introduction}. Slightly reformulating these latter formulas, one has for any $v \in \hat W$
 \begin{equation}\label{eq:kap6.3.3}
 \lm_{j,-k}^+ = -\lm_{j,k}^- , \quad \lm_{j,-k}^- = -\lm_{j,k}^+ \, , \qquad  \forall k\geq 1\, , j=1,2
 \end{equation}
 and
 \begin{equation} \label{eq:kap6.3.4}
 \lm_{1,0}^+ = \frac{1}{16\,\lm_{2,0}^-} , \quad \lm_{1,0}^- = \frac{1}{16\,\lm_{2,0}^+}.
 \end{equation}
 According to \cite[Lemma 3.17]{LOperator2018}, one has for $j=1,2$ the asymptotics
 \begin{equation}\label{eq:kap6.3.5}
 \lm_{j,k}^+,\lm_{j,k}^- = k\pi + \ell_k^2 \, ,
 \end{equation}
 and by Lemma \ref{lem:ReciprocityOfdotLmAndPeriodic}, for any $v = (q, p) \in \hat W,$
 \begin{equation}\label{eq:kap6.3.5bis}
 \lm_{2,k}^+(q,p) = \lm_{1,k}^+(-q,p), \quad \lm_{2,k}^-(q,p) = \lm_{1,k}^-(-q,p), \qquad \forall  k\in\Z.
 \end{equation}
 Furthermore if $v$ is in \Hr
 \begin{equation}\label{eq:kap6.3.6}
 \cdots < \lm_{j,k}^- \leq \lm_{j,k}^+ < \lm_{j,k+1}^- \leq \lm_{j,k+1}^+<\cdots \, .
 \end{equation}
 Similarly we define
 \begin{equation}\label{eq:kap6.3.7}
 \mu_{1,k} := \left\{ \begin{array}{ll}
                        \mu_k & k\geq 0 \\ -\mu_{-k} & k\leq -1
                        \end{array}\right.
 \quad
 \mu_{2,k} := \left\{ \begin{array}{ll}
                        \frac{1}{16\mu_{-k}} & k\geq 0 \\ -\frac{1}{16 \mu_{k}} & k\leq -1
                        \end{array}\right.
 \end{equation}
 and
 \begin{equation}\label{eq:kap6.3.8}
 \dot{\lm}_{1,k} := \left\{ \begin{array}{ll}
                        \dot{\lm}_k & k\geq 0 \\ -\dot{\lm}_{-k} & k\leq -1
                        \end{array}\right.
 \quad
 \dot{\lm}_{2,k} := \left\{ \begin{array}{ll}
                        \frac{1}{16\dot{\lm}_{-k}} & k\geq 0 \\ -\frac{1}{16 \dot{\lm}_{k}} & k\leq -1
                        \end{array}.\right.
 \end{equation}
 Before stating the product representation of $\chi_p(\lm),$ $\chi_D(\lm)$, and $\dot \Dl(\lm)$, let us discuss the envisioned type of representations in general terms. 
 According to \cite[Lemma C.1]{nlsbook}, for any given sequences $(\sigma_{1,n})_n$, $(\sigma_{2,n})_n$ in the space
 \[
 \ell^* : = \{ \nu_n = n\pi +\ell_n^2 \, : \, \nu_n \in\C^*,\; n\in\Z \}
 \]
 the infinite products
 \begin{equation}\label{eq:kap6.3.9}
 f_1(\lm) := \prod_{n\in\Z} \frac{\sigma_{1,n}-\lm}{\pi_n} ,\quad f_2(\lm) := \prod_{n\in\Z} \frac{\sigma_{2,n} + \frac{1}{16\lm}}{\pi_n}
 \end{equation}
 define analytic functions on $\C^*$ with roots $\sigma_{1,n}$, $n\in\Z$, and respectively, $(-16\sigma_{2,n})^{-1}$, $n\in\Z$. Note that in addition, $f_1$ is analytic at $0$, $f_2$ is analytic at $\infty$ and
 \begin{equation}\label{eq:kap6.3.10}
 f_1(0) = \prod_{n\in\Z} \frac{\sigma_{1,n}}{\pi_n} , \quad f_2(\infty) = \prod_{n\in\Z} \frac{\sigma_{2,n}}{\pi_n}
 \end{equation}
 are well defined numbers in $\C^*$. Furthermore, by \cite[Lemma C.5]{nlsbook}, one sees that uniformly for $\lm$ in $\partial B_N = \{\lm\in\C \, : \, |\lm| = N\pi + \pi/2 \}$
 \begin{equation}\label{eq:kap6.3.11}
 f_1(\lm) = -(1+o(1))\sin(\lm), \quad f_2(\lm) = f_2(\infty) + O(\frac1N),\quad \text{as}\; N\to\infty
 \end{equation}
 and uniformly for $\lm$ in
 $ \partial B_{-N} = \{\lm\in\C \, : \, |16\lm|= \frac{1}{N\pi+\pi/2} \}$
 \begin{equation}\label{eq:kap6.3.12}
 f_1(\lm) = f_1(0) + O(\frac1N), \quad f_2(\lm) = (1+o(1))\sin(\frac{1}{16\lm}),\quad \text{as}\; N\to\infty.
 \end{equation}

 Let us first consider $\dot\Dl(\lm) \equiv \dot\Dl(\lm,v)$ for $v\in\hat W$. Since by \cite[Lemma 3.15]{LOperator2018}, $\dot\lm_{j,n} = n\pi + \ell_n^2$ for $j=1,2$, the infinite products
 \begin{equation}\label{eq:kap6.3.13}
 \dot\Dl_1(\lm) := \prod_{n\in\Z} \frac{\dot\lm_{1,n} - \lm}{\pi_n}, \qquad \dot\Dl_2(\lm) : = \prod_{n\in\Z} \frac{\dot\lm_{2,n} + \frac{1}{16\lm}}{\pi_n}
 \end{equation}
 are well defined analytic functions on $\C^*$ by the considerations above.

 \begin{lem}\label{lem:dotDlProd}
 On $\hat W$, $\dot\Dl(\lm)$ admits the product representation for $\lm\in\C^*$
 \begin{equation}\label{eq:kap6.3.13bis}
 \dot\Dl(\lm) = c_{\dot\Dl} (1-\frac{\dot\lm_*}{\lm})(1+\frac{\dot\lm_*}{\lm}) \dot\Dl_1(\lm)\dot\Dl_2(\lm),\quad c_{\dot\Dl}:= \frac{1}{\dot\Dl_2(\infty)}.
 \end{equation}
 Furthermore, $\dot\Dl_2(\infty) = -16\dot\lm_*^2\dot\Dl_1(0)$ or in more detail
         \begin{equation}\label{eq:6.12ter}
               -\dot\lm_*^2(16\dot\lm_0)^2\prod_{n\geq 1}(\dot\lm_n16\dot\lm_{-n})^2 = 1.
               \end{equation}
 \end{lem}
 \begin{proof}
 Let $F(\lm) := \dot\Dl_1(\lm)\dot\Dl_2(\lm)$. Then $F(\lm)$ is analytic on $\C^*$. By \eqref{eq:kap6.3.11}, uniformly for $\lm\in\partial B_N$
 \begin{equation}\label{eq:kap6.3.14}
  F(\lm) = -(1+o(1)) \dot\Dl_2(\infty) \sin(\lm) \quad \text{as} \; N\to\infty
 \end{equation}
 and by \eqref{eq:kap6.3.12}, uniformly for $\lm\in\partial B_{-N}$
 \begin{equation}\label{eq:kap6.3.15}
 F(\lm) = (1+o(1)) \dot\Dl_1(0) \sin(\frac{1}{16\lm}) \quad\text{as}\; N\to\infty.
 \end{equation}
 Hence $(1-\frac{\dot\lm_*}{\lm})(1+\frac{\dot\lm_*}{\lm}) F(\lm)$ is a holomorphic function on $\C^*$ 
 with the same roots as $\dot\Dl(\lm)$ and with the property that the quotient $G(\lm) = (1-\frac{\dot\lm_*}{\lm})(1+\frac{\dot\lm_*}{\lm}) F(\lm)/\dot\Dl(\lm)$ 
 defines a holomorphic function on $\C^*$. By the asymptotics of $\dot\Dl$ of \cite[Lemma 2.17]{LOperator2018} and of \eqref{eq:kap6.3.14}, one has uniformly in $\lm\in\partial B_N$
 \begin{equation}\label{eq:kap6.3.16}
 G(\lm) = \dot\Dl_2(\infty) (1+o(1)) \quad \text{as}\; N\to\infty.
 \end{equation}
 To obtain the asymptotics of $G$ for $\lm\in\partial B_{-N}$ as $N\to\infty$ note that by \cite[Lemma 2.14(ii)]{LOperator2018}, 
 $\dot\Dl(\lm,q,p) = \dot\Dl(-\frac1{16\lm} ,-q,p)\cdot\frac{1}{16\lm^2}$ which we rewrite with $\lm=-\frac{1}{16\mu}$ as
 \[
 \dot\Dl(-\frac{1}{16\mu},q,p) = \dot\Dl(\mu,-q,p) 16\mu^2.
 \]
 Since $\lm\in\partial B_{-N}$ iff $\mu=-\frac{1}{16\lm} \in\partial B_N$ it then follow from \eqref{eq:kap6.3.15} that uniformly for $\lm\in\partial B_{-N}$
 \[
 G(\lm) = G(-\frac{1}{16\mu}) = \frac{1-(16\mu\dot\lm_*)^2}{16\mu^2} \dot\Dl_1(0) (1+o(1))\quad\text{as}\; N\to\infty.
 \]
 Hence $G$ is bounded on $\partial B_{-N}$ as $N\to\infty$. The latter estimate together with the one of \eqref{eq:kap6.3.16} 
then allow to apply Lemma \ref{ap:lem:liouvilleCn} (Liouville's theorem) in Appendix \ref{Liouville}, yielding that $G$ is constant. 
One concludes that \eqref{eq:kap6.3.13bis} 
 holds and $\dot\Dl_2(\infty) = -16(\dot\lm_*)^2 \dot\Dl_1(0)$ or, with 
 $\dot\Dl_2(\infty) = \prod_{n\in\Z}\frac{\dot\lm_{2,n}}{\pi_n},$ $\dot\Dl_1(0) = \prod_{n\in\Z} \frac{\dot\lm_{1,n}}{\pi_n},$
 \[
 1 = -16\dot\lm_*^2 \prod_{n\in\Z} \frac{\dot\lm_{1,n}}{\dot\lm_{2,n}}.
 \]
 Taking into account identities of the type \eqref{eq:kap6.3.3}-\eqref{eq:kap6.3.4}, it then follows that
 \[
   -\dot\lm_*^2(16\dot\lm_0)^2\prod_{n\geq 1}(\dot\lm_n16\dot\lm_{-n})^2 = 1
 \]
 as claimed.
 \end{proof}
 \begin{rem}
 Note that for $v=0$, $\dot\lm_* = \ii/4$ and by Lemma \ref{lem:ReciprocityOfdotLmAndPeriodic}(i) $\frac{1}{16\dot\lm_{-n}} =\dot\lm_n$ for any $n\in\Z.$
 \end{rem}
 In the same way one can prove a corresponding product representation of $\chi_D(\lm) \equiv \chi_D(\lm,v)$ for $v\in\hat W$. 
 Since by \cite[Lemma 3.16]{LOperator2018}, $\mu_{j,n} = n\pi + \ell_n^2$ for $j=1,2$, the infinite products
 \[
 \chi_{D,1}(\lm) := \prod_{n\in\Z} \frac{\mu_{1,n}-\lm}{\pi_n},\qquad \chi_{D,2}(\lm) := \prod_{n\in\Z} \frac{\mu_{2,n} +\frac{1}{16\lm}}{\pi_n}
 \]
 are well defined analytic functions on $\C^*$.

 \begin{lem}\label{lem:ProdRepresentationChiD}
 On $\hat W$, $\chi_D$ admits the product representation for $\lm\in\C^*$
 \begin{equation}\label{eq:kap6.3.17}
 \chi_D(\lm,v)= -c_D \chi_{D,1}(\lm)\chi_{D,2}(\lm), \quad c_D= \frac{1}{\chi_{D,2}(\infty)}.
 \end{equation}
  Furthermore $\chi_{D,2}(\infty) =-e^{q(0)} \chi_{D,1}(0)$ or in more detail,
  \[
  e^{q(0)}16\mu_0^2 \prod_{n\geq1} (\mu_n16\mu_{-n})^2 =1
  \]
  and
  \[
  \chi_{D,2}(-\frac{1}{16\lm} ,q,p) = \chi_{D,1}(\lm,-q,p).
  \]
 \end{lem}
 \begin{proof}
 Let $F(\lm) =  \chi_{D,1}(\lm)\chi_{D,2}(\lm)$. Then $F$ is analytic on $\C^*$. Since $F$ and $\chi_D(\lm)$ have the same roots, $G(\lm) := F(\lm)/\chi_D(\lm)$ is well defined and analytic on $\C^*$. Arguing as in the proof of Lemma \ref{lem:dotDlProd} one sees that uniformly for $\lm\in\partial B_N$
 \begin{equation}\label{eq:kap6.3.18}
 F(\lm) = -(1+o(1)) \chi_{D,2}(\infty)\sin(\lm) \quad\text{as}\; N\to\infty
 \end{equation}
 and uniformly for $\lm\in\partial B_{-N}$
 \begin{equation}\label{eq:kap6.3.19}
 F(\lm) = (1+o(1)) \chi_{D,1}(0) \sin(\frac{1}{16\lm}).
 \end{equation}
 By \cite[Lemma 3.2(iii)]{LOperator2018}, uniformly for $\lm\in\partial B_N$,
 \[
 \chi_D(\lm,v) = (1+o(1)) \chi_D(\lm,0)\quad\text{as} \;N\to\infty.
 \]
 Since for $v=0$, $\chi_D(\lm,0)= \sin(\omega(\lm))$ (cf. \cite[Theorem 3.1]{LOperator2018}),  $\dot\Dl(\lm,0)= -(1+\frac{1}{!6\lm^2})\sin(\omega(\lm))$ (cf. \eqref{eq:dotDlzero}), 
 and $\dot\lm_* = \ii/4$, one has
 \[
 \chi_D(\lm,0) = -c_{\dot\Dl} \dot\Dl_1(\lm,0) \dot\Dl_2(\lm,0).
 \]
 Hence by \eqref{eq:kap6.3.14} uniformly for $\lm\in\partial B_N$,
 \begin{equation}\label{eq:kap6.3.20}
 \chi_D(\lm,v) = (1+o(1))\chi_D(\lm,0) =  \sin(\lm)(1+o(1)) \quad\text{as} \; N\to\infty.
 \end{equation}
 This combined with \eqref{eq:kap6.3.18} then shows that uniformly for $\lm\in\partial B_N$
 \[
 G(\lm) = -\chi_{D,2}(\infty) (1+o(1)) \quad \text{as} \; N\to\infty.
 \]
 To obtain the asymptotics of $G$ for $\lm\in\partial B_{-N}$ as $N\to\infty$, note that by \cite[Lemma 3.2]{LOperator2018},
 \[
 \chi_D(\lm,q,p) = e^{-q(0)} \chi_D(-\frac{1}{16\lm} ,-q,p).
 \]
 Arguing as in the proof of Lemma \ref{lem:dotDlProd} one then concludes from \eqref{eq:kap6.3.19} that uniformly for $\lm\in\partial B_{-N}$
 \[
 G(\lm) = -e^{q(0)} \chi_{D,1}(0) (1+o(1)).
 \]
 Hence by Lemma \ref{ap:lem:liouvilleCn} in Appendix \ref{Liouville}, $G(\lm)$ is constant, implying that
 \[
 \chi_D(\lm) = -\frac{1}{\chi_{D,2}(\infty)} \chi_{D,1}(\lm) \chi_{D,2}(\lm)
 \]
 as well as $\chi_{D,2}(\infty) = e^{q(0)} \chi_{D,1}(0)$ or in more detail,
 \[
 1 = e^{q(0)} \frac{\chi_{D,1}(0)}{\chi_{D,2}(\infty)} = e^{q(0)} \prod_{n\in\Z} \frac{\mu_{1,n}}{\mu_{2,n}}.
 \]
 Taking into account the identities of the type \eqref{eq:kap6.3.3}-\eqref{eq:kap6.3.4} one obtains the claimed identity
 \[
 1 = e^{q(0)} 16\mu_0^2 \prod_{n\geq 1} (\mu_n 16\mu_{-n})^2.
 \]
 Since by Lemma \ref{lem:DirichletSpecReciprocity}
 \[
 \mu_{2,n}(q,p) = \mu_{1,n}(-q,p)\quad \forall n\in\Z,
 \]
 one has
 $$
 \chi_{D,2}(-\frac{1}{16\lm} ,q,p) = \prod_{n\in\Z}\frac{\mu_{2,n}(q,p)-\lm}{\pi_n} = \chi_{D,1}(\lm,-q,p).
 $$
 \end{proof}

 Finally, we discuss the product representation of $\chi_p(\lm,v) = \Dl^2(\lm,v)-1$ for $v\in\hat W$. 
 Since by \cite[Lemma 3.17]{LOperator2018}, $\lm_{j,n}^\pm = n\pi + \ell_n^2$ for $j=1,2$, the infinite products
 \begin{equation}\label{definition chi_pj}
 \chi_{p,1}(\lm) : = \prod_{n\in\Z} \frac{(\lm_{1,n}^+ - \lm)(\lm_{1,n}^--\lm)}{\pi_n^2},
 \qquad \chi_{p,2} (\lm) : = \prod_{n\in\Z} \frac{(\lm_{2, n}^+ + \frac{1}{16\,\lm})(\lm_{2,n}^-+\frac{1}{16\,\lm})}{\pi_n^2},
 \end{equation}
 are well defined analytic functions on $\C^*$. Note that
 \begin{equation}\label{eq:kap6.3.22bis}
 \chi_{p,1}(0) = \prod_{n\in\Z} \frac{\lm_{1,n}^+\lm_{1,n}^-}{\pi_n^2} = \lm_0^+\lm_0^- \Big(\prod_{n \geq1} \frac{\lm_n^+\lm_n^-}{\pi_n^2}\Big)^2
 \end{equation}
 and
 \[
 \chi_{p,2}(\infty) = \prod_{n\in\Z} \frac{\lm_{2,n}^+\lm_{2,n}^-}{\pi_n^2} = (16\lm_0^+16\lm_0^-)^{-1} \Big(\prod_{n \geq1} \frac{(16\lm_{-n}^+16\lm_{-n}^-)^{-1}}{\pi_n^2}\Big)^2
 \]

 \begin{lem} \label{lem:ProdRepresentationDl2-1}
 On  $\hat W,$ $\chi_p$ admits  the product representation for $\lm\in\C^*$
 \begin{equation}\label{eq:kap6.3.23}
 \chi_p(\lm) = -c_p\chi_{p,1}(\lm)\chi_{p,2}(\lm),\quad c_p = \frac{1}{\chi_{p,1}(0)}.
 \end{equation}
 Furthermore,
 \begin{equation}\label{eq:kap6.3.24}
 \chi_{p,2}(\infty) = \chi_{p,1}(0)
 \end{equation}
 or in more detail,
 \begin{equation}\label{eq:preiodicEigenvaluesRestriction}
  (16\lm_0^+\lm_0^-)^2\prod_{n\geq 1} (\lm_n^+16\lm_{-n}^+)^2(\lm_n^-16\lm_{-n}^-)^2 = 1,
 \end{equation}
 and
   \begin{align}\begin{split} \label{eq:Freciproc}
   \chi_{p,2}\big(-\frac{1}{16\lm},q,p\big)= \chi_{p,1}(\lm,-q,p).
 \end{split}  \end{align}
 As a consequence, $\chi_{p,1}(0,-q,p) = \chi_{p,1}(0,q,p)$.
 \end{lem}

 \begin{proof}
 Let $F(\lm)  =\chi_{p,1}(\lm)\chi_{p,2}(\lm)$. Then $F(\lm)$ is analytic on $\C^*$. Since $F(\lm)$ and $\chi_p(\lm)$ have the same roots, $G(\lm):= F(\lm) /\chi_p(\lm)$ is well defined and analytic on $\C^*$. Arguing as in the proof of Lemma
 \ref{lem:dotDlProd}, one sees that uniformly for $\lm\in\partial B_N$
 \begin{equation}\label{eq:kap6.3.21}
 F(\lm) = (1+o(1)) \chi_{p,2}(\infty) \sin^2(\lm)\quad \text{as} \; N\to\infty
 \end{equation}
 and uniformly for $\lm\in\partial B_{-N}$
 \begin{equation}\label{eq:kap6.3.22}
 F(\lm) = (1+o(1)) \chi_{p,1}(0) \sin^2(\frac{1}{16\lm}) \quad\text{as} \; N\to\infty.
 \end{equation}
 Since by \cite[Lemma 2.17]{LOperator2018}, uniformly for $\lm\in\partial B_N$
 \[
 \chi_p(\lm) = -(1+o(1))\sin^2(\lm)
 \]
 it then follows from \eqref{eq:kap6.3.21} that uniformly for $\lm\in\partial B_N$
 \begin{equation}\label{eq:kap6.3.23bis}
 G(\lm) = -(1+o(1)) \chi_{p,2}(\infty) \quad\text{as}\; N\to\infty.
 \end{equation}
 To obtain the asymptotics of $G$ on $\partial B_{-N}$ note that by \cite[Lemma 2.14(i),(ii)]{LOperator2018}
 \[
 \chi_p(-\frac{1}{16\mu},q,p) = \chi_p(\mu,-q,p).
 \]
 When combined with \eqref{eq:kap6.3.22} and \eqref{eq:kap6.3.23bis} one then concludes that uniformly for $\lm\in\partial B_{-N}$
 \begin{equation*}
 G(\lm) = -(1+o(1)) \chi_{p,1}(0)\quad\text{as} \; N\to\infty.
 \end{equation*}
 Hence by Lemma \ref{ap:lem:liouvilleCn} in Appendix \ref{Liouville}, $G$ is constant and therefore
 \[
 \chi_p(\lm) = -\frac{1}{\chi_{p,2}(\infty)} \chi_{p,1}(\lm) \chi_{p,2}(\lm)
 \]
 and $\chi_{p,2}(\infty) = \chi_{p,1}(0)$ which can be  expressed as
 \[
 1 = (16\lm_0^+\lm_0^-)^2 \prod_{n\geq1} (\lm_n^+ 16\lm_{-n}^+)^2 (\lm_n^-16\lm_{-n}^-)^2.
 \]
 Furthermore, since by Lemma \ref{lem:ReciprocityOfdotLmAndPeriodic},
 $(16\lm_{-n}^\pm(q,p))^{-1} = \lm_n^\mp (-q,p)$ for any $n\in\Z$, implying that $\lm_{2,n}^\pm(q,p) = \lm_{1,n}^\pm(-q,p)$ for any $n\in\Z$, and since $\lm_{j,n}^+ = -\lm_{j,-n}^-$ for any $n\not=0$, $j=1,2$ one sees that
 \[
 \chi_{p,2}(-\frac{1}{16\mu},q,p) = \prod_{n\in\Z} \frac{(\lm_{2,n}^+(q,p) -\mu)(\lm_{2,n}^-(q,p) -\mu)}{\pi_n^2} = \chi_{p,1}(\mu,-q,p),
 \]
 proving \eqref{eq:Freciproc}. For $\mu=0$, one then gets $\chi_{p,1}(0,-q,p) = \chi_{p,2}(\infty,q,p)$ which equals $\chi_{p,1}(0,q,p)$ by \eqref{eq:kap6.3.24}.
 \end{proof}

 To finish this section we prove asymptotics for the sequences $(\tau_n-\dot\lm_n)_n$ and $(\dot\lm_n)_n$ as $n\to\infty$, 
 refining the ones of \cite[Lemma 3.15]{LOperator2018}. Recall that by \eqref{eq:taungmn}, $\tau_n = (\lm_n^++\lm_n^-)/2$. For $n\geq0$, let $\check \Dl_n(\lm)$ be defined by
 \[
 \chi_p(\lm) = \Dl^2(\lm) -1 = \check \Dl_n(\lm) (\lm_n^+-\lm) (\lm_n^--\lm).
 \]
 By Lemma \ref{lem:ProdRepresentationChiD}, $\check \Dl_n(\lm)$ admits the product representation
 \begin{equation}\label{eq:kap6.3.29bis}
\check  \Dl_n(\lm) = -c_p \chi_{p,2}(\lm) \frac{\chi_{p,1}(\lm)}{(\lm_n^+-\lm)(\lm_n^--\lm)}.
 \end{equation}
 \begin{lem}\label{lem:betterAsymptoteDotlm} On $\hat W$, one has for any $n\geq 0$,
  \begin{equation}\label{eq:tauminusdotlmn}
   2(\tau_n-\dot\lm_n) \check \Dl_n(\dot\lm_n)= \left((\tau_n-\dot\lm_n)^2-\gamma_n^2/4\right) \partial_\lambda \check \Dl_n(\dot\lm_n).
 \end{equation}
 Furthermore, locally uniformly on $\hat W$,
 \begin{equation}\label{eq:refinedDotlmAsymptotic}
         \dot\lm_n = \tau_n +\gm^2\ell_n^2 \quad\text{as}\; n\to\infty.
 \end{equation}
 \end{lem}
 \begin{proof}
 Since $\dot\Dl(\dot\lm_n)$, one has
 \begin{align*}
 0=& \left.\partial_\lm(\Dl^2(\lm)-1)\right|_{\lm= \dot\lm_n}
  =-(\lm_n^+ -2\dot\lm_n + \lm_n^-)\check \Dl_n(\dot\lm_n) + (\lm_n^+-\dot\lm_n)(\lm_n^--\dot\lm_n)\partial_\lambda \check \Dl_n(\dot\lm_n).
 \end{align*}
 The identity
 $(\lm_n^+-\dot\lm_n)(\lm_n^- - \dot\lm_n) = (\tau_n-\dot\lm_n)^2 -\gm^2/4$ then yields \eqref{eq:tauminusdotlmn}.
 To prove the asymptotics \eqref{eq:refinedDotlmAsymptotic}, note that by 
 \cite[Lemma C.4]{nlsbook},
 \[
 \frac{\chi_{p,1}(\lm)}{(\lm_n^+-\lm)(\lm_n^--\lm)} = \pfrac{\sin(\lm-n\pi)}{\lm-n\pi}^2(1+\ell_n^2)
 \]
 uniformly for $\lm\in D_n$. On the other hand,
 \[
 \chi_{p,2}(\lm) = \chi_{p,2}(\infty) + O(\frac1n) = c_p^{-1} + O(\frac1n)
 \]
 uniformly for $\lm\in D_n$. Altogether it then follows that
 \[
 \check \Dl_n(\lm) = -\pfrac{\sin(\lm-n\pi)}{\lm-n\pi}^2 + \ell_n^2
 \]
 uniformly for $\lm\in D_n$ as $n\to\infty$.
 Hence by Cauchy's estimate
 \[
 \partial_\lm\left(\check \Dl_n(\lm) + \left(\frac{\sin(\lm-n\pi)}{\lm-n\pi}\right)^2\right) =\ell_n^2
 \]
 uniformly for $\lm\in D_n$. Using that $\frac{\sin(\lm-n\pi)}{\lm-n\pi} = \int_0^1 \cos(t(\lm-n\pi))\dt$ one sees that
 \[
 \partial_\lm\left(\frac{\sin(\lm-n\pi)}{\lm-n\pi}\right) = -\int_0^1 t\sin(t(\lm-n\pi))\dt.
 \]
 Since $\dot\lm_n = n\pi + \ell_n^2$, this implies that
 \[
 \partial_\lm \left(\frac{\sin(\lm-n\pi)}{\lm-n\pi}\right)^2\Big|_{\lm=\dot\lm_n} =\ell_n^2.
 \]
 Altogether we have shown that
 \begin{equation}\label{eq:proofLemmDotDlAsymptotic6.255}
 \partial_\lambda \check \Dl_n(\dot\lm_n)=\ell_n^2, \qquad \check \Dl_n(\dot\lm_n)= -1+\ell_n^2.
 \end{equation}
 Hence  \eqref{eq:tauminusdotlmn} implies that
 \[
  (\tau_n - \dot\lm_n)(1+\ell_n^2 +(\tau_n-\dot\lm_n)\ell_n^2) = \gm^2 \ell_n^2.
 \]
 Since by \cite[Lemma 3.15, Lemma 3.17]{LOperator2018}, $\tau_n-\dot\lm_n = \ell_n^2$ there exists $N\geq1$ so that $(1+\ell_n^2+(\tau_n-\dot\lm_n)\ell_n^2)\geq\frac12$ for any $n\geq N$ and \eqref{eq:refinedDotlmAsymptotic} follows.
 Going through the arguments of the proof one verifies that \eqref{eq:refinedDotlmAsymptotic} holds locally uniformly on $\hat W$.
 \end{proof}


\section{Standard root and canoncial root}\label{sec:roots}

 In this section we introduce branches of square roots of various complex valued functions, needed in the sequel.
 For $v\in\hat W$, $n\in\Z$, and $j\in\{1,2\}$ set
 \[
 \gm[j,n] : = \lm_{j,n}^+ -\lm_{j,n}^- , \qquad \tau_{j,n} : = (\lm_{j,n}^+ + \lm_{j,n}^-)/2.
 \]
 Here $\hat W$ is the complex neighborhood of \Hr in \Hp, defined by \eqref{eq:DefnOfhatW}.
 Note that since by \eqref{eq:kap6.3.1} -- \eqref{eq:kap6.3.2}, $\lm_{j,-n}^+ = -\lm_{j,n}^-$ and $\lm_{j,-n}^- = -\lm_{j,n}^+$ for any $n\geq1$, one has
 \begin{equation}\label{gamma_j-n and tau_j-n}
 \gm[j,-n] = \gm[j,n], \qquad \tau_{j,-n} = -\tau_{j,n}, \qquad  \forall n\geq 1, \; j=1,2.
 \end{equation}
 Furthermore, recall that by \eqref{G_1m} -- \eqref{G_2m},
 \[
 G_{1,n} = [\lm_{n}^-,\lm_{n}^+],  \quad n\geq0, \quad \qquad \qquad G_{1,-n} = -G_{1,n},  \quad n\geq1,
 \]
 \[
 G_{2,n} = [-\lm_{-n}^+,-\lm_{-n}^-], \quad  n\geq0 , \quad \qquad G_{2,-n} = -G_{2,n},  \quad n\geq1.
 \]
 Denoting by $\refl$ the map
 \[
 \refl :\C^* \to \C^*, \lm\mapsto -\frac{1}{16\lm}
 \]
 one then has for any $v=(q,p)\in\hat W$ and $n\in\Z$
 \[
 \refl(G_{1,n}(q,p) ) = G_{2,n}(-q,p) , \qquad \refl(G_{2,n}(q,p)) = G_{1,n}(-q,p).
 \]
 Similarly, recall that by \eqref{U_1m} -- \eqref{U_2m}
 \[
 U_{1,n} = U_n , \quad  n\geq0, \quad \qquad U_{1,-n} = -U_n , \quad n\geq1,
 \]
 \[
 U_{2,n} = -U_{-n},  \quad n\geq0, \qquad U_{2,-n} = U_{-n} ,  \quad n\geq1,
 \]
 where $U_n$, $n\in\Z$, are isolating neighborhoods for $v\in\hat W$. Without loss of generality we assume that for any $v=(q,p)\in\hat W$, $n\in\Z$
 \[
 \refl(U_{1,n}(q,p)) = U_{2,n}(-q,p), \quad \refl(U_{2,n}(q,p)) = U_{1,n}(-q,p).
 \]
 Finally, we recall that by \eqref{Gamma_1m} -- \eqref{Gamma_2m}
 $$
  \Gamma_{1,m} = \Gamma_m \ , \quad \forall m\geq0\ , \qquad \quad \,\,\,\, \Gamma_{1,-m} = (\Gamma_m)^- \; \, \forall m\geq1, 
 $$
 $$
 \Gamma_{2,m} = (\Gamma_{-m})^-  , \,\,  \forall m \geq0 \ , \qquad \quad \Gamma_{2,-m} = \Gamma_{-m} \; \, \forall m\geq1 \, . \,\,\,
$$
   Customarily,  we denote by $\sqrt[+]{\lm}$ the principal branch of the square root defined for $\lm$ in $\C\setminus (-\infty,0]$ and determined by $\sqrt[+]{1}=1$.
 \begin{defn} \label{def:standardRoots}
 For any $n\in\Z$, $v\in\hat W$, the standard root $w_{1,n}(\lm) \equiv w_{1,n}(\lm,v)$, also referred to as s-root, is defined by
   \begin{align*}
 w_{1,n}(\lm) \equiv   w_{1,n}(\lm,v) :=& \sqrt[s]{(\lm_{1,n}^+-\lm)(\lm_{1,n}^--\lm)}, \quad \lm \not\in \Gap[1,n]\, ,
   \end{align*}
  determined by setting for $\lm\in\C$ with $\left|\gamma_{1,n}^2/4(\tau_{1,n}-\lm)^2\right|<1$,
   \begin{equation}\label{eq:wnsqrtp}
   w_{1,n}(\lm) = (\tau_{1,n}-\lm)\sqrt[+]{1-\gamma_{1,n}^2/4(\tau_{1,n}-\lm)^2}.
   \end{equation}
 \end{defn}
  Note that since by \eqref{gamma_j-n and tau_j-n}, $\gm[1,n]^2/4(\tau_{1,n}+\lm)^2 = \gm[1,-n]^2/4(\tau_{1,-n}-\lm)^2$ for any $n\geq1$, 
  \eqref{eq:wnsqrtp} implies that
   \begin{equation}\label{eq:wnsqrtpbis}
   w_{1,-n}(\lm) = -(\tau_{1,n}+\lm)\sqrt[+]{1-\gamma_{1,n}^2/4(\tau_{1,n}+\lm)^2} = -w_{1,n}(-\lm).
   \end{equation}
  for any $\lm\in\C$ with $|\gm[1,n]^2 /4(\tau_{1,n} + \lm)^2|<1.$

 \begin{lem}\label{lem:standardRootAnalyticSizeinverseintegral}
   For any $v\in \hat W$ and $n \in\Z$, the standard root $w_{1,n}(\lm)$ is analytic on $\C\setminus\Gap[1,n]$. In case $\gm[1,n] =0$, $w_{1,n}(\lm) = \tau_{1,n}-\lm$. 
   Furthermore, for any $v_0\in\hat W$,  $w_{1,n}(\lm,v)$ is analytic on $(\C\setminus U_{1,n})\times V_{v_0}$ where $V_{v_0}$ is 
   the open neighborhood of Lemma \ref{lem:isolatingNeighbourhoodExtension} and $U_n$, $n\in\Z$, are isolating neighborhoods for $V_{v_0}$. 
   Moreover, for any $n,m\in\Z$ with $m\not=n$,
   \[
   c^{-1}|m-n| \leq |w_{1,n}(\lm,v)|\leq c|m-n|, \quad (\lm,v)\in U_{1,m}\times V_{v_0}
   \]
   for some constant $c\geq 1$.
   Finally, for $n, m \in\Z$
   \[
   \frac{1}{2\pi\ii}\int_{\Gm[1,m]} \frac{\dlm}{w_{1,n}(\lm,v)} =-\delta_{mn} \, .
   \]
 \end{lem}

 \begin{proof}
 The claimed result can be proved in a straight forward way, using the asymptotics of the periodic eigenvalues of $Q(v)$ of \cite[Lemma 3.17]{LOperator2018}.
 \end{proof}

 Next we define  $\sqrt[c]{\chi_{1,p}(\lm)}$, referred to as canonical root of $\chi_{1,p}(\lm)$, in terms of the s-roots $w_{1,n}(\lm)$.
 Here $\chi_{1,p}(\lambda)$ is the infinite product, introduced in \eqref{definition chi_pj}. To this end we need the following 

 \begin{lem}\label{lem:fn}
 (i) Let $v_0\in \hat W$ be given. For any $v\in V_{v_0}$ and $n\geq 0$,
   \begin{equation}\label{eq:kap6.32ter}
   f_{1,n}(\lm,v) :=  \frac{1}{\pi_n}\prod_{m\not=n } \frac{w_{1,m}(\lm,v)}{\pi_m}
   \end{equation}
   defines a function which is analytic in $\lm$ on $\C\setminus \bigcup_{m\not=n}\Gap[1,m]$ and 
   analytic in $(\lambda, v)$ on $\C \setminus\left(\bigcup_{m\not=n}U_{1,m}\right) \times V_{v_0}$. Moreover, $f_{1,n}$ does not vanish on these domains and in case $\gm[1,m]=0$, it extends analytically in $\lm$ to $\lm=\tau_{1,m}$. \\
 (ii)
   For any $v\in \Hr$ and $n\in\Z$,
   \[
    (-1)^n f_{1,n}(\lm,v) > 0 \, , \quad \forall \lm_{1,n-1}^+ < \lm < \lm_{1,n+1}^-.
   \]
 \end{lem}
 \begin{proof}
 Item (i) follows from Lemma \ref{lem:standardRootAnalyticSizeinverseintegral} and \cite[Lemma C.1]{nlsbook}. 
 Concerning (ii) note that for $v\in\Hr$, $\lm_{1,n}^\pm \in\R\setminus\{ 0 \}$ and that by  \eqref{eq:wnsqrtp}, $w_{1,m}(\lm)$ 
 takes values in $\R_{>0}$ for  $\lm<\lm_{1,m}^-$
 and in $\R_{<0}$ for $\lm>\lm_{1,m}^+$. Hence for $\lm_{1,n-1}^+ < \lm < \lm_{1,n+1}^-$, with $n\geq0$
 \[
 \prod_{\substack{m>n} } \frac{w_{1,m}(\lm,v)w_{1,-m}(\lm,v)}{\pi_m \pi_{-m}}>0\, ,
 \qquad \prod_{0\leq m<n } \frac{w_{1,m}(\lm,v)}{\pi_{m}} >0\, , 
 \qquad  (-1)^n\prod_{\substack{-n \leq m<0} } \frac{w_{1,m}(\lm,v)}{\pi_m} >0.
 \]
 In the case where $n\leq -1$, one argues similarly.
 \end{proof}

 By the definition of $\chi_1 \equiv \chi_{p,1}$ and $f_{1,n}$, $n\in\Z$, one has on $\C$
 \[
 \chi_{1}(\lm) = w_n^2(\lm) f_{1,n}^2(\lm).
 \]
 The canonical root $\sqrt[c]{\chi_{1}(\lm)}$ of $\chi_1(\lm)$ is then defined on $\C\setminus \bigcup_{m\in\Z} G_{1,m}(v)$ by
 \begin{equation}\label{eq:kap6.3.36}
 \sqrt[c]{\Fc(\lm)} \equiv \sqrt[c]{\Fc(\lm, v)}  : = w_{1,n}(\lm, v) f_{1,n}(\lm, v).
 \end{equation}
 \begin{lem} \label{lem:rootOfF}
  For any $v \in \hat W$,  $\sqrt[c]{\Fc(\lm,v)}$ is well defined by \eqref{eq:kap6.3.36} on $\C\setminus \bigcup_{m\in\Z} G_{1,m}(v)$ 
   and for any $v_0\in\hat W$ analytic in $(\lm,v)$ on $(\C \setminus\bigcup_{m\in\Z} U_{1,m})\times V_{v_0}$.  Moreover, $\sqrt[c]{\Fc(\lm,v)}$ does not vanish on these domains,
 \begin{align}\label{eq:kap6.3.37}
  \sqrt[c]{\Fc(0,q,p)} = \sqrt[c]{\Fc(0,-q,p)},
 \end{align}
 and for $\lm\in\C\setminus \bigcup_{m\in\Z} U_{1,m}$ (and hence $-\lm\in\C\setminus \bigcup_{m\in\Z} U_{1,m}$)
 \begin{equation}\label{eq:kap6.3.37bis}
 \sqrt[c]{\Fc(-\lm,v)} = \sqrt[c]{\Fc(\lm,v)}\frac{w_{1,0}(-\lm,v)}{w_{1,0}(\lm,v)}.
 \end{equation}
 In case $\gm[1,m](v)=0$ for some $m\in\Z$, $\sqrt[c]{\Fc(\lm,v)}$ extends analytically to $\lm = \tau_{1,m}(v)$.
 \end{lem}
 \begin{proof}
 In view of Lemma \ref{lem:fn}, it remains to prove identity \eqref{eq:kap6.3.37} and \eqref{eq:kap6.3.37bis}.
 Since by Lemma \ref{lem:ProdRepresentationDl2-1}, $\Fc(0,q,p) = \Fc(0,-q,p)$,  the claimed identity \eqref{eq:kap6.3.37}  is true up to a sign. 
 Since $\Fc(0,v) \not=0$ for any $v\in \hat W$ and since \eqref{eq:kap6.3.37} clearly holds for $v=0$, the identity \eqref{eq:kap6.3.37} follows by continuity. 
 Concerning \eqref{eq:kap6.3.37bis} note that by \eqref{eq:wnsqrtpbis}, for any $n\not=0$, $w_{1,-n}(\lm)/\pi_{-n} = w_{1,n}(-\lm)/\pi_n$. 
 This then implies \eqref{eq:kap6.3.37bis}.
 \end{proof}

 \begin{lem} \label{lem:FcSign}
 On $\Hr$, the following holds for any $n\in\Z$:
 \begin{equivenum}
 \item For any $\lm_{1,n-1}^+<\lm <\lm_{1,n}^-$, $(-1)^n\sqrt[c]{\Fc(\lm)} >0 $.
 \item For any $\lm_{1,n}^-\leq \lm \leq \lm_{1,n}^+$, the limits of $\sqrt[c]{\chi_1(\lm+\ii \epsilon)}$ and $\sqrt[c]{\chi_1(\lm-\ii\epsilon)}$ as $\epsilon \to 0^+ $  exist and
 \[
 \pm\lim_{\substack{\epsilon\to0\\ \epsilon>0}} (-1)^n \Im \sqrt[c]{\chi_1(\lm\mp\ii \epsilon)} \geq0.
 \]
 Extending $\sqrt[c]{\chi_1(\lm)}$ to $G_{1,n}$ from below one has $(-1)^{n+1}\ii \sqrt[c]{\chi_1(\lm)}>0$ for any $\lm_{1,n-1}^-<\lm <\lm_{1,n}^+$.
 \item For any $\lm\in\R$ with $\lm<-\lm_0^+(q,p)$ or $\lm>\lm_{-1}^+(q,p)$
 \begin{equation}\label{eq:kap6.3.39}
  \sqrt[c]{\Fc(-\frac1{16\lm},-q,p)}>0.
 \end{equation}
 \item $\sqrt[c]{\Fc(0)} = \sqrt[+]{\lm_0^+\lm_0^-} \prod_{m\geq 1} \frac{\lm_m^+ \lm_m^-}{\pi_m^2}.$
 \end{equivenum}
 \end{lem}

 \begin{proof}
 Since by Lemma \ref{lem:fn}(ii), $(-1)^n f_n(\lm,q,p) > 0$ for any $\lm_{1,n-1}^+<\lm <\lm_{1,n+1}^-$ item (i) and (ii) follow from \eqref{eq:wnsqrtp}. Concerning item (iii) note that by item(i) with $n = 0$, 
 $\sqrt[c]{\Fc(-\frac1{16\lm},-q,p)}>0$ for $\lm_{1,-1}^+(-q,p) < -\frac1{16\lm} < \lm_{1,0}^-(-q,p)$.  Since $\lm_{1,-1}^+(-q,p) =-\lm_1^-(-q,p)=-(16\lm_{-1}^+(q,p))^{-1}$ and $\lm_{1,0}^-(-q,p) = \lm_0^-(-q,p) = (16\lm_0^+(q,p))^{-1}$ one has
 \[
 \lm_{1,-1}^+(-q,p) < -\frac1{16\lm} < \lm_{1,0}^- (-q,p) \quad \text{ iff }\quad [\lm>\lm_{-1}^+(q,p) \; \text{ or } \lm< - \lm_0^+(q,p)].
 \]
 Item(iv) follows from (i), (iii), and \eqref{eq:kap6.3.22bis}.
 \end{proof}
 For any $v=(q,p)\in\hat W$ and $-\frac1{16\lm} \in\C^*\setminus G_{1,n}(-q,p)$, $n\in\Z$,
 \[
 w_{1,n}(-\frac{1}{16\lm},-q,p) = \sqrt[s]{(\lm_{1,n}^+(-q,p) + \frac1{16\lm})(\lm_{1,n}^-(-q,p) + \frac{1}{16\lm})}.
 \]
 Since $\lm_n^\pm(-q,p) = \frac{1}{16\lm_{-n}^\mp(q,p)}$ for any $n\in\Z$ it follows from the definition of $\lm_{2,n}^\pm$ that
 \[
 w_{1,n}(-\frac1{16\lm},-q,p) = \sqrt[s]{(\lm_{2,n}^+(q,p) +\frac{1}{16\lm})(\lm_{2.n}^-(q,p) + \frac1{16\lm})}.
 \]
 We define
 \[
 w_{2,n}(\lambda) \equiv w_{2,n}(\lm,v) := \sqrt[s]{(\lm_{2,n}^+(v)+\frac{1}{16\lm})(\lm_{2,n}^-(v)+\frac{1}{16\lm})} \quad
 \]
 and the canonical root $\sqrt[c]{\chi_2(\lm)}$ of $\chi_2(\lm) \equiv \chi_{p,2}(\lm)$ by
 \[
 \sqrt[c]{\chi_2(\lm)} \equiv \sqrt[c]{\chi_2(\lm,v)} := \prod_{n\in\Z} \frac{w_{2,n}(\lm,v)}{\pi_n}.
 \]
 For any $n\in\Z$, we have
 \begin{equation}\label{eq:kap6.3.55bis}
 w_{2,n}(\lm,q,p) = w_{1,n}(-\frac{1}{16\lm},-q,p), \qquad \sqrt[c]{\chi_2(\lm,q,p)} = \sqrt[c]{\chi_{1}(-\frac{1}{16\lm},-q,p)}.
 \end{equation}
 The canonical root $\sqrt[c]{\chi_p(\lm)} \equiv \sqrt[c]{\chi_p(\lm,v)}$ is then defined on $\hat W$ by
 \begin{equation}\label{eq:kap6.3.40}
 \sqrt[c]{\chi_p(\lm)} : = \ii \frac{1}{\sqrt[c]{\chi_1(0)}} \sqrt[c]{\chi_1(\lm)} \sqrt[c]{\chi_2(\lm)}.
 \end{equation}
 where $\lm\in\C^*\setminus \bigcup_{m\in\Z} \big( G_{1,m}(v)\cup G_{2,m}(v)\big) = \C^* \setminus \bigcup_{m\in\Z}(G_m(v)\cup-G_m(v)\big)$. Note that
 \begin{equation}\label{eq:kap6.3.40bis}
 G_{2,m}(q,p)  = \set{-\frac1{16\mu}}{\mu\in G_{1,m}(-q,p)}
 \end{equation}
 and that by \eqref{eq:kap6.3.37}
 \begin{equation}\label{eq:kap6.3.401}
 \sqrt[c]{\chi_1(0)} = \sqrt[c]{\chi_2(\infty)}.
 \end{equation}

 \begin{lem}\label{lem:standardRoot}
 \begin{equivenum}
 \item   For any $v_0\in\hat W$, the canonical root $\sqrt[c]{\chi_p(\lm)}$ is an analytic function in $\lm\in\C^*\setminus\bigcup_{m\in\Z} (G_m\cup -G_m)$ and analytic in $(\lm,v)$ on $(\C^*\setminus \bigcup_{m\in\Z} (U_m\cup-U_m))\times V_{v_0}$. In case $\gm[m]=0$ for some $m\in\Z$, $\sqrt[c]{\chi_p(\lm)}$ extends analytically to $\lm = \tau_m$ and $\lm=-\tau_m$.
 \item  The canonical root at the zero potential is
 \begin{equation}\label{eq:kap6.43}
  \sqrt[c]{\chi_p(\lm,0)} = -\ii\sin(\omega(\lm)).
 \end{equation}
 \item For any $v=(q,p)\in\hat W$ and $\lm\in\C^*\setminus\big(\bigcup_{m\in\Z} G_m\cup -G_m\big)$ the following identites hold:
 \begin{equation} \label{eq:kap6.44}
 \sqrt[c]{\chi_p(-\lm,v)} = -\sqrt[c]{\chi_p(\lm,v)} \, , \qquad
 \quad \sqrt[c]{\chi_p(-(16\,\lm)^{-1},-q,p)} = \sqrt[c]{\chi_p(\lm,q,p)}.
 \end{equation}
 \end{equivenum}
 \end{lem}
 \begin{proof}
 Item (i) follows from Lemma  \ref{lem:ProdRepresentationDl2-1} and Lemma \ref{lem:rootOfF}. To prove (ii) note that
 by \cite[Lemma 2.16]{LOperator2018},  $\chi_p(\lm,0)= -\sin^2(\omega(\lm))$. Hence $\sqrt[c]{\chi_p(\lm,0)} = \pm \ii \sin(\omega(\lm))$ and it remains to determine the sign. 
 To this end note that for $v=0$, $\lm_n^+ = \lm_n^-$ and hence $w_{1,n}(\lm) = \tau_{1,n}-\lm$ for any $n\in\Z$. 
 One then concludes from 
 the definition \eqref{eq:kap6.3.36} of the $c$-root of $\chi_1(\lm, 0)$ that 
 $\sqrt[c]{\chi_1(\lm, 0)} = \prod_{n\in\Z}\frac{\tau_{1,n}-\lm}{\pi_n}$, $\lm\in\C$. It implies that $\sqrt[c]{\chi_1(0)} = \prod_{n\in\Z}\frac{\tau_{1,n}}{\pi_n}$. 
 As $\sqrt[c]{\chi_1(-\frac{1}{16\lm}, 0)} = \sqrt[c]{\chi_2(\lm, 0)}$ one has $\sqrt[c]{\chi_2(\infty)} = \sqrt[c]{\chi_1(0)}$ and since $\tau_{1,n} =n\pi +\ell_n^2$ 
 and $\omega(\lm) = \lm-\frac{1}{16\lm}$ it follows from the definition \eqref{eq:kap6.3.40} of $\sqrt[c]{\chi_p(\lm)}$ and \cite[Lemma C.5]{nlsbook} that uniformly for $\lm\in\partial B_N$
 \[
 \sqrt[c]{\chi_p(\lm,0) } = -\ii (1+o(1))\sin(\omega( \lm)) \quad\text{as}\; N\to\infty
 \]
 and uniformly for $\lm\in\partial B_{-N}$
 \[
 \sqrt[c]{\chi_p(\lm,0) } = -\ii(1+o(1))\sin(\omega( \lm)) \quad\text{as}\; N\to\infty.
 \]
 Since $\sqrt[c]{\chi_p(\lm,0)}$ and $\sin(\omega(\lm))$ have the same roots it then follows from Lemma \ref{ap:lem:liouvilleCn} in Appendix \ref{Liouville} that
 \[
 \sqrt[c]{\chi_p(\lm,0)} = -\ii \sin(\omega(\lm)).
 \]

 \noindent (iii) By the identities of $\Dl$ stated in \cite[Lemma2.14]{LOperator2018}, the claimed symmetries hold up to a sign. Furthermore, since $\omega(-\lm) = -\omega(\lm)$ and $\omega(-\frac{1}{16\,\lm}) = \omega(\lm)$, they hold for $v=0$.
 Hence by continuity they hold on $\hat W$.
 \end{proof}

 On \Hr, the sign table for $\sqrt[c]{\chi_p(\lm)}$ can be computed by using Lemma \ref{lem:FcSign}.
 \begin{lem}\label{lem:DlSign}
 On $\Hr$, for any  $n\in\Z$, the following holds:\\
 (i) For any $\lm_{1,n-1}^+< \lm <\lm_{1,n}^-$,
 \begin{equation}\label{eq:kap6.3.42bis}
  (-1)^{n}\Im\sqrt[c]{\chi_p(\lm)} >0.
 \end{equation}
 Similarly, for any $\lm_{2,n-1}^+<-\frac{1}{16\lm} <\lm_{2,n}^-$,
 \begin{equation}\label{eq:kap6.3.42ter}
  (-1)^{n}\Im\sqrt[c]{\chi_p(\lm)} >0.
 \end{equation}
 (ii)
 For any $\lm_{1,n}^-\leq \lm  \leq \lm_{1,n}^+$, the limits of $\sqrt[c]{\chi_p(\lm+\ii \epsilon)}$ and $\sqrt[c]{\chi_p(\lm-\ii \epsilon)}$ as $\epsilon \to 0^+$ exist and
 \[
 \pm \lim_{\substack{\epsilon\to0 \\\epsilon>0}} (-1)^n \Re \sqrt[c]{\chi_p(\lm\pm \ii \epsilon)} \geq 0.
 \]
 Extending $\sqrt[c]{\chi_p(\lm)}$ continuously to $G_{1,n}$ from below one has
 \begin{equation}\label{eq:kap6.3.34}
 (-1)^{n+1} \sqrt[c]{\chi_p(\lm)}>0,\quad  \forall \;\lm_{1,n}^- < \lm <\lm_{1,n}^+.
 \end{equation}
 Similarly, for any $\lm_{2,n}^-<-\frac{1}{16\lm} <\lm_{2,n}^+$, the limit $\lim_{\substack{\epsilon\to0\\\epsilon>0}} \sqrt[c]{\chi_p(\lm-\ii\epsilon})$ 
 exists and if one extends $\sqrt[c]{\chi_p(\lm)}$ to $G_{2,n}$ from below, then
 \begin{equation}\label{eq:kap6.34bis}
 (-1)^{n+1} \sqrt[c]{\chi_p(\lm)} >0, \quad \forall \lm_{2,n}^- < -\frac{1}{16\lm} < \lm_{2,n}^+.
 \end{equation}
 \end{lem}
 \begin{figure} \label{fig:signOfSqrtFc}
  \centering
    \includegraphics[width=\textwidth]{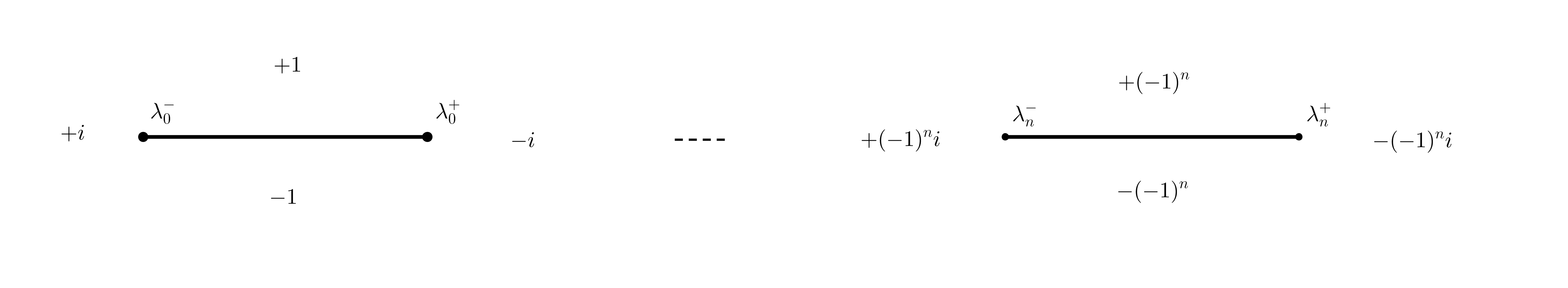}
         \caption{Illustration of the sign of $\sqrt[c]{\Dl^2-1}$}
 \end{figure}

 \begin{proof}
Let $v = (q, p) \in \Hr$ and $n \in \Z$.
(i)  For any $\lm_{1,n-1}^+<\lm <\lm_{1,n}^-$ one has according to Lemma \ref{lem:FcSign},
 \[
  (-1)^n \sqrt[c]{\chi_1(\lm, v)} >0 \, , \qquad \sqrt[c]{\chi_1(0, v)}>0 \, , \qquad \sqrt[c]{\chi_1(-(16\lm)^{-1} ,-q,p)}>0.
 \]
 By the definition of $\sqrt[c]{\chi_p(\lm)}$ it then follows that $(-1)^n\Im \sqrt[c]{\chi_p(\lm)}>0$. Next consider $\lm_{2,n-1}^+<-\frac{1}{16\lm} <\lm_{2,n}^-$. 
 Since $\lm_{2,n-1}^+(q,p) = \lm_{1,n-1}^+(-q,p)$ and $\lm_{2,n}^-(q,p) = \lm_{1,n}^-(-q,p)$ one has $\lm_{1,n-1}^+(-q,p) < -\frac{1}{16\lm} < \lm_{1,n}^-(-q,p)$ 
 and thus by \eqref{eq:kap6.3.42bis}, $(-1)^n \Im \sqrt[c]{\chi_p(-(16\lm)^{-1},-q,p)}>0$. 
 The claimed inequality then follows 
 since by Lemma \ref{lem:standardRoot}(iii), $\sqrt[c]{\chi_p(-(16\lm)^{-1},-q,p)} = \sqrt[c]{\chi_p(\lm,q,p)}$,

 \noindent (ii) The claimed results follow from Lemma \ref{lem:FcSign}. 
 Indeed, for any $\lm_{1,n}^-<\lm <\lm_{1,n}^+$, one has by Lemma~\ref{lem:FcSign} that
 $$
(-1)^{n+1}\ii \sqrt[c]{\chi_1(\lm, v)} >0  \ ,  \qquad \sqrt[c]{\chi_1(-(16\lm)^{-1},-q,p)}>0 \ , \qquad  \sqrt[c]{\chi_1(0, v)}>0 \, . 
 $$
 By the definition of $\sqrt[c]{\chi_p(\lm)}$ it then follows that when $\sqrt[c]{\chi_p(\lm)}$ is extended continuously to $G_{1,n}$ from below 
 then $(-1)^{n+1}\sqrt[c]{\chi_p(\lm)}>0$ for any $\lm_{1,n}^-<\lm <\lm_{1,n}^+$. \\
 It remains to consider the claimed results for $\lm_{2,n}^- <-\frac{1}{16\lm} <\lm_{2,n}^+$, i.e., for
 $-\lm_{-n}^+ < \lm<-\lm_{-n}^-$  in the case $n\geq0$ and for $\lm_{-n}^- <\lm<\lm_{-n}^+$ in the case $n \leq -1$.
 Use again that $\lm_{2,n}^\pm(q,p) = \lm_{1,n}^\pm(-q,p)$ to conclude by the above that 
 $(-1)^{n+1}\sqrt[c]{\chi_p(-(16\lm)^{-1},-q,p)}>0$. 
 Since $\sqrt[c]{\chi_p(-(16\lm)^{-1},-q,p)} = \sqrt[c]{\chi_p(\lm,q,p)}$ it then follows 
 that $(-1)^n\sqrt[c]{\chi_p(\lm,q,p)}>0$ for any $\lm_{2,n}^- < -\frac{1}{16\lm} < \lm_{2,n}^+.$
 \end{proof}
 We finish this section with various asymptotic estimates.
For any $a=(a_n)_n$ define 
$$
\norm{a}_{\ell^\infty} := \sup_{n\in\Z}|a_n| \, .
$$

  \begin{lem} \label{lem:sindurchsqrtcEstimate}
  \begin{equivenum}
  \item
   Locally uniformly in $v\in \hat W$ and uniformly in $\lm\in U_{1,m}$, 
  \begin{equation}\label{eq:kap6.3.44}
    \frac{w_{1,m}(\lm)}{\sqrt[c]{\chi_1(\lm)}} \frac{- \sin(\omega(\lm))}{\pi_m - \omega(\lm)} ,\,
   \frac{w_{1,m}(\lm)}{\sqrt[c]{\chi_p(\lm)}} \frac{-\ii \sin(\omega(\lm))}{\pi_m - \omega(\lm)} \, = 1+\ell_m^2 \quad \text{ as }\; |m|\to\infty \, .
  \end{equation}
  \item
  Locally uniformly in $v\in \hat W$ and uniformly in  $\lm \in \partial B_N$,
  \begin{equation}\label{eq:kap6.3.45}
  \frac{- \sin(\omega(\lm))}{\sqrt[c]{\chi_1(\lm)}} , \,
  \frac{-\ii \sin(\omega(\lm))}{\sqrt[c]{\chi_p(\lm)}} \, = 1+ o(1) \quad \text{ as }\; N\to\infty \, .
  \end{equation}
 \item Locally uniformly in $v\in \hat W$ and uniformly in $\lm\in U_{2,m}$,
 \begin{equation}\label{eq:kap6.3.46}
 \frac{w_{2,m}(\lm)}{\sqrt[c]{\chi_2(\lm)}} \frac{-\sin(\omega(\lm))}{\pi_m - \omega(\lm)} ,\,
 \frac{w_{2,m}(\lm)}{\sqrt[c]{\chi_p(\lm)}} \frac{-\ii \sin(\omega(\lm))}{\pi_m - \omega(\lm)} \, = 1+\ell_m^2 \quad \text{ as } \; |m|\to\infty \, .
 \end{equation}
 \item Locally uniformly in $v\in \hat W$ and uniformly in $\lm\in \partial B_{-N}$,
 \begin{equation}\label{eq:kap6.3.47}
 \frac{- \sin(\omega(\lm))}{\sqrt[c]{\chi_2(\lm)}} ,\,
 \frac{-\ii \sin(\omega(\lm))}{\sqrt[c]{\chi_p(\lm)}} \, = 1+ o(1) \quad \text{ as } \; N\to\infty \, .
 \end{equation}
 \end{equivenum}
  \end{lem}
  \begin{proof}
  (i) We begin by verifying the claimed asymptotics for 
  $\frac{w_{1,m}(\lm)}{\sqrt[c]{\chi_p(\lm)}} \frac{-\ii\sin(\omega(\lm))}{\pi_m-\omega(\lm)}$. For $v = (q,p) \in\hat W$ and $\lm\in U_{1,m}$, 
  it follows from the definitions \eqref{eq:kap6.3.36} and \eqref{eq:kap6.3.40} of the canonical  roots  
  ${\sqrt[c]{\chi_1(\lm)}}$ and respectively, ${\sqrt[c]{\chi_p(\lm)}}$ as well as \eqref{eq:kap6.3.55bis} that
  \[
  \frac{\sqrt[c]{\chi_p(\lm)}}{w_{1,m}(\lm)} = \ii \frac{\sqrt[c]{\chi_1(-(16\lm)^{-1},-q,p)}}{\sqrt[c]{\chi_1(0,v)}}f_{1,m}(\lm,v)
  \]
 where by \eqref{eq:kap6.32ter}, $f_{1,m}(\lm)\equiv f_{1,m}(\lm,v)$ is given by
 \[
 f_{1,m}(\lm) = \frac{1}{\pi_m} \prod_{k\not=m} \frac{w_{1,k}(\lm)}{\pi_k}.
 \]
 Let us first consider the quotient $\sqrt[c]{\chi_1(-(16\lm)^{-1},-q,p)} / \sqrt[c]{\chi_1(0,q,p)}$. Since $\sqrt[c]{\chi_1(\lm,-q,p)}$ does not vanish at $\lm=0$ 
 and is differentiable (cf. definition \eqref{eq:DefnOfhatW} of $\hat W$) and since $\sqrt[c]{\chi_1(0,q,p)} = \sqrt[c]{\chi_1(0,-q,p)}$ (cf. \eqref{eq:kap6.3.36}) 
 one sees by expanding $\sqrt[c]{\chi_1(\mu,-q,p)}$ at $\mu=0$ that for $\lm\in U_{1,m}$ and the definition of $U_{1,m}$ 
 (cf. (I-2) in Section \ref{sec:realAmostReal}) that
 \[
 \frac{\sqrt[c]{\chi_1(-(16\lm)^{-1},-q,p)}}{\sqrt[c]{\chi_1(0,q,p)}} = 1+\ell_m^2\quad\text{as}\quad |m|\to\infty
 \]
 and in turn
 \[
 \frac{\sqrt[c]{\chi_p(\lm)}}{w_{1,m}(\lm)} = \ii f_{1,m}(\lm) (1+\ell_m^2)\quad \text{as }\; |m|\to\infty.
 \]
 In particular, for $v=0$ one obtains, taking into account Lemma \ref{lem:standardRoot}(ii),
 \[
 \frac{-\ii \sin(\omega(\lm))}{\pi_m - \omega(\lm)} = \frac{\sqrt[c]{\chi_p(\lm,0)}}{\pi_m-\omega(\lm)} = \ii f_{1,m} (\lm,0)(1+\ell_m^2)
 \]
 implying that
 \[
 \frac{w_{1,m}(\lm)}{\sqrt[c]{\chi_p(\lm)}} \frac{-\ii \sin(\omega(\lm))}{\pi_m-\omega(\lm)} = \frac{f_{1,m}(\lm,0)}{f_{1,m}(\lm)} (1+\ell_m^2)\quad\text{as}\quad |m|\to\infty.
 \]
 It remains to analyze the asymptotics of $f_{1,m}(\lm)$, $\lm\in U_{1,m}$, as $|m|\to\infty$. By property (I-3), defined in in Section \ref{sec:realAmostReal}, 
 there exists $M\geq1$ so that for any $|m|\geq M$ and $\lm\in U_{1,m}$,
 \[
 w_{1,k}(\lm) = (\tau_{1,k}-\lm) \sqrt[+]{1-\gm[1,k]^2/4(\tau_{1,k}-\lm)^2}, \quad k\not=m.
 \]
  Hence
 \[
 \frac{f_{1,m}(\lm,0)}{f_{1,m}(\lm, v)} = \prod_{k\not=m} \frac{\pi_k-\lm}{\tau_{1,k}-\lm} \frac{1}{\sqrt[+]{1-\gm[1,k]^2/4(\tau_{1,k}-\lm)^2}}.
 \]
 By \cite[Lemma C.3]{nlsbook}, it then follows that uniformly for $\lm\in U_{1,m}$
 \[
 \frac{f_{1,m}(\lm,0)}{f_{1,m}(\lm)} = 1+\ell_m^2\quad\text{as}\quad |m|\to\infty
 \]
 and hence \eqref{eq:kap6.3.44} is proved. Going trough the arguments of the proof one sees that \eqref{eq:kap6.3.44} holds locally uniformly on $\hat W$.
 Since $\sqrt[c]{\chi_p(\lm)} = \ii \sqrt[c]{\chi_1(\lm)}\sqrt[c]{\chi_2(\lm)} /\sqrt[c]{\chi_2(\infty)}$,
 the claimed asymtptotics for  $\frac{w_{1,m}(\lm)}{\sqrt[c]{\chi_1(\lm)}} \frac{-\ii\sin(\omega(\lm))}{\pi_m-\omega(\lm)}$
 have also been established by the above proof.

 \noindent (ii) Arguing in a similar way as in the proof of item(i) one shows that the claimed asymptotics follow from 
 \cite[Lemma C.5]{nlsbook}. The claimed asymptotics of item (iii) and (iv) follow by reciprocity (cf \eqref{eq:kap6.44}),
 \[
 \sqrt[c]{\chi_p(-(16\lm)^{-1},-q,p)} = \sqrt[c]{\chi_p(\lm,q,p)}.
 \]
 \end{proof}
 
 Lemma \ref{lem:sindurchsqrtcEstimate} has the following application. Recall that
 $\ell^* := \set{\sigma = (\sigma_n)_n \subset \C^*}{ (\sigma_m - m\pi )_m \in\ell^2}$.
 \begin{cor} \label{cor:kap6.4.1}
 \begin{equivenum}
 \item Locally uniformly in $v\in\hat W$ and in $\sigma\in\ell^*$,
 \[
 \sup_{\lm\in U_{1,n}}\Big| \prod_{m\not=n} \frac{\sigma_m - \lm}{w_{1,m}(\lm) } -1 \Big| = \ell_n^2.
 \]
 \item Locally uniformly in $v\in\hat W$ and in $\sigma\in\ell^*$,
 \[
 \sup_{\lm\in \partial B_N}\Big| \prod_{m\in\Z} \frac{\sigma_m - \lm}{w_{1,m}(\lm) } -1 \Big| = o(1)\quad \text{ as } \; N\to\infty.
 \]
 \end{equivenum}
 \end{cor}
 \begin{proof} (i) The claimed asymptotics follow from the asymptotics of infinite products of \cite[Lemma C.4]{nlsbook}
 and the asymptotics of $ \frac{- \sin(\omega(\lm))}{\sqrt[c]{\chi_1(\lm)}}$, stated in Lemma \ref{lem:sindurchsqrtcEstimate}.\\
 (ii) The claimed asymptotics follow from the asymptotics and their proof in \cite[Lemma C.5]{nlsbook}
 and the asymptotics of $ \frac{- \sin(\omega(\lm))}{\sqrt[c]{\chi_1(\lm)}}$, stated in Lemma \ref{lem:sindurchsqrtcEstimate}.
 \end{proof}


\section{Proofs of Theorem \ref{thm:kap8.2.12} and Corollary \ref{cor:kap8.2.12}}\label{sec:psi}

In this section we prove Theorem  \ref{thm:kap8.2.12} and, at the end, Corollary \ref{cor:kap8.2.12}.

Let us begin by summarizing the setup for Theorem  \ref{thm:kap8.2.12}, described in Section \ref{introduction} and Section \ref{sec:realAmostReal} --
Section \ref{sec:roots}.
For any given $v_0\in\hat W$ (cf. \eqref{eq:DefnOfhatW}), let $V_{v_0}\subset\hat W$ be a neighborhood of $v_0$ in \Hp and $(U_m)_{m\in\Z}$, $U_*$ 
be a family of isolating neighborhoods satisfying properties (I-1) - (I-5)
(cf. Section \ref{sec:realAmostReal}), which work for any $v\in V_{v_0}$ (cf. Lemma \ref{lem:isolatingNeighbourhoodExtension}). 
These neighborhoods are conveniently listed in form of two sequences $(U_{j, m})_{m \in \Z}$, $j = 1, 2$ (cf. \eqref{U_1m} - \eqref{U_2m}).
Correspondingly, the gaps $G_{j, m}$ and the contours  $\Gamma_{j, m} \subset U_{j, m}$ around $G_{j,m}$ 
are listed in this way (cf. \eqref{G_1m} - \eqref{G_2m} and \eqref{Gamma_1m} - \eqref{Gamma_2m}).
Theorem \ref{thm:kap8.2.12} states that for any $v$ in a complex neighborhood $W \subset \hat W$ of \Hr,   there exist normalized differentials 
on the spectral curve $\Sigma_v$ (cf. \eqref{spectral curve}), given by
$$
 \Sigma_v = \{(\lambda , y) \in  \C^*\times \C  :  y^2 =   \chi_p(\lambda, v)  \}\, , \qquad \chi_p(\lambda, v) = \Delta^2(\lambda, v) -1 \, .
 $$
These differentials are required to be of the form  $\psi_n(\lambda)/ \sqrt{\chi_p(\lambda)}$, $n \ge 0,$
where $\psi_n(\lm)\equiv \psi_n(\lm, v)$ are analytic functions on $\C^*$, given by 
$\psi_n(\lm) = - \frac{1}{\pi_n} C_n \psi_{n,1}(\lm) \psi_{n,2}(\lm)$ (cf. \eqref{eq:kap8.2.9})
with $\psi_{n,1}$, $\psi_{n,2}$ being infinite products (cf. \eqref{eq:kap8.2.11}),
$$
 \psi_{n,1}(\lm) = \prod_{\substack{k\in\Z \\ k\not=n}} \frac{\sigma_{1,k}^n-\lm}{\pi_k} , \qquad 
 \psi_{n,2}(\lm) = \prod_{k\in\Z} \frac{\sigma_{2,k}^n + \frac{1}{16\lm}}{\pi_k} \, 
 $$
and $\sigma_{j,k}^n$  satisfying $\sigma_{j,k}^n = k\pi + \ell_k^2$, $j = 1,2$, (cf. \eqref{eq:kap8.2.15}) so that (cf. \eqref{eq:kap8.2.12})
$$
\sigma_{1,k}^n\in U_{1,k} , \quad  k\not=n\, ,  \qquad (-16\sigma_{2,k}^n)^{-1} \in U_{2,k}, \quad \forall k \in \Z  \, .
$$
By \cite[Lemma C.1]{nlsbook},  $\psi_{n,1}(\lm)$, $ \psi_{n,2}(\lm) $ 
define analytic functions on $\C^*$ with roots $\sigma^n_{1,k}$ ($k \ne n$), and respectively, $(-16\sigma^n_{2,k})^{-1}$ ($k\in\Z$). 
The normalization of the differentials $\psi_n(\lambda)/ \sqrt{\chi_p(\lambda)}$, $n \ge 0,$ is defined by
\[
\int_{\Gm[1,m]} \frac{\psi_n(\lm)}{\sqrt[c]{\chi_p(\lm)}} \dlm = 2 \pi \delta_{m, n}\ , \qquad 
\int_{\Gm[2,m]} \frac{\psi_n(\lm)}{\sqrt[c]{\chi_p(\lm)}} \dlm = 0\ , \qquad \forall m \in \Z \  ,
\]
where ${\sqrt[c]{\chi_p(\lm)}}$ denotes the canoncial root of $\chi_p(\lambda)$ (cf. \eqref{eq:kap6.3.40}).
Note that under the assumptions made, these contour integrals are well defined and 
do not depend on the choice of the contours $\Gm[j,m]$ as long as $\Gm[j,m]\subset U_{j,m}$ and $G_{j,m}$ is inside the contour $\Gm[j,m]$. 
Our approach for proving Theorem \ref{thm:kap8.2.12} is to reformulate it as a system of  functional equations for the (unknown) complex numbers
$\sigma^n_{j, k}$, 
which we then solve by the means of the implicit function theorem. 

To this end we first introduce some more notation. 
For any given $v_0 \in \hat W$, let $V_{v_0}\subset\hat W$ be a neighborhood of $v_0$ in \Hp with $U_{j,m}$, $G_{j, m},$
and $\Gamma_{j, m}$ ($j = 1,2$, $m \in \Z$) as at the beginning of this section.
For any $j = 1,2$ and $m \in \Z$, we choose annular neighborhoods $U_{j,m}'$ of $\Gm[j,m]$ such that for any potential in $V_{v_0}$
(maybe after shrinking $V_{v_0}$) one has $\overline{U_{j,m}'}\subset U_{j,m}\setminus (G_{j,m}\cup\{\mu_{j,m}\})$ and
\begin{equation}\label{eq:yan1}
\inf \{ |\chi_p(\lm, v)| \, ,  |\grave{m}_2(\lm, v)| \,  : \,v \in V_{v_0}, \lm \in \bigcup_{\substack{m\in\Z, j=1,2}}U_{j,m}' \}  >0 \, .
\end{equation}
Here $(\mu_{j,m})_{m \in \Z, j= 1,2}$ is a listing of the Dirichlet eigenvalues of $Q(v)$, defined by \eqref{eq:kap6.3.7}.
Furthermore, introduce
\[
\ell_{\hat n}^2 := \ell^2(\Z\setminus\{n\}, \;\C)
\]
and denote by $\Omega$ the open subset of $\hat W \times \ell^2\times\ell^2$ of elements $(v, s_1, s_2)$ so that
\begin{equation} \label{eq:kap8.2.30}
\sigma_{1,k}:= k\pi + s_{1,k}\in U_{1,k} \, , \quad \forall k\in\Z\ ,
\end{equation}
\begin{equation}\label{eq:kap8.2.31}
\sigma_{2,k}:= k\pi + s_{2,k}\in\C^* \quad\text{with}\quad (-16\sigma_{2,k})^{-1}\in U_{2,k} \ , \ \forall k\in\Z.
\end{equation}
Here  $U_{j,k}$ are a set of isolating neighborhoods which work for a complex neighborhood of $v$ in $\hat W$ (cf. \eqref{U_1m} - \eqref{U_2m}).
Since by Lemma \ref{lem:kap8.2.18.5} below, the functions $\psi_n$, $n \ge 0$, we are going to construct for $v\in\Hr$ 
must have roots in the gap intervals $G_{1,m}$ ($m\not=n$) and $G_{2,m}$ ($m\in\Z$), it suffices to consider $\Omega$. In addition, define for any $n \ge 0$
\[
\Omega_n := \set{(v,s_1,s_2)\in\Omega}{s_{1,n} =\tau_{1,n} -n\pi}.
\]
By a slight abuse of notation, for any $(v,s_1,s_2)\in\Omega_n$ we also denote $(s_{1,k})_{k\not=n}$ by $s_1$.
Given $v\in\hat W$ and $n \ge 0$, we are looking for a solution $s_1 =  (s_{1,k})_{k\not=n}\in\ell_{\hat n}^2$ and $s_2 = (s_{2,k})_{k\in\Z}\in\ell^2$ 
with $(v,s_1,s_2)\in\Omega_n$ to the equation
\begin{equation}\label{eq:kap8.2.17}
 F^n(v,s_1,s_2) = (F_1^n(v,s_1, s_2), F_2^n(v,s_1,s_2)) = 0
\end{equation}
where $F_1^n = (F_{1,m}^n)_{m\not=n}$ and $F_2^n = (F_{2,m}^n)_{m\in\Z}$ are defined as follows
\begin{equation}\label{eq:kap8.2.20}
 F_{1,m}^n := A_{1,m}^n (v) f_n(s_1,s_2) := (n-m)\int_{\Gm[1,m]} \frac{f_n(s_1,s_2,\lm)}{\sqrt[c]{\chi_p(\lm)}} \dlm
\end{equation}
\begin{equation} \label{eq:kap8.2.21}
F_{2,m}^n : = A_{2,m}^n (v) f_n(s_1,s_2) : = 16\pi_m^2\pi_n \int_{\Gm[2,m]} \frac{f_n(s_1,s_2,\lm)}{\sqrt[c]{\chi_p(\lm)}} \dlm.
\end{equation}
Here
\[
f_n(s_1,s_2,\lm) : = -\frac{1}{\pi_n} \frac{1}{f_{n,2}(s_2,\infty)} f_{n,1}(s_1,\lm) f_{n,2}(s_2,\lm), \qquad \qquad
\]
\begin{equation}\label{eq:kap8.2.22}
f_{n,1}(s_1,\lm) : = \prod_{k\not=n} \frac{\sigma_{1,k}-\lm}{\pi_k} , \qquad f_{n,2}(s_2,\lm) : = \prod_{k\in\Z} \frac{\sigma_{2,k} + \frac{1}{16\lm}}{\pi_k} \ ,
\end{equation}
and
\[
\sigma_{1,k}: = k\pi + s_{1,k} \; (k\not=n) , \qquad \sigma_{2,k} := k\pi + s_{2,k}\; (k\in\Z).
\]
Note that since $(v,s_1,s_2)\in \Omega_n,$ 
$$
f_{n,2}(s_2,\infty) = \prod_{k\in\Z} \frac{\sigma_{2,k}}{\pi_k} \in \C^*
$$ 
and hence $f_n(s_1,s_2,\lm)$ is well defined.
By a slight abuse of notation, for any $n\geq 0$ and $m\in\Z$, we view $F_{1,m}^n$ ($m\not=n$) and $F_{2,m}^n$ either as a function on $\Omega_n$ or $\Omega$.
The reason we make the ansatz \eqref{eq:kap8.2.9} - \eqref{eq:kap8.2.12} for $\psi_n(\lm)$ stems from the following observation for real valued potentials.

\begin{lem}\label{lem:kap8.2.18.5}
Let $j\in\{1,2\}$, $m\in\Z$ (with $m\not=n$ in the case $j=1$), and $v\in\Hr$. Furthermore assume that $f:U_{j,m} \to\C$ is real analytic. 
If $A_{j,m}(v) f = 0$, then $f$ has a root in $G_{j,m}$.
\end{lem}
\begin{proof}
Let us first treat the case $j=1$. If $G_{1,m}$ is not a single point, then we may shrink the contour $\Gm[1,m]$ to the interval $G_{1,m}$. 
By \eqref{eq:kap6.3.34} of Lemma \ref{lem:DlSign}, one has $(-1)^{m+1} \sqrt[c]{\chi_p(\lm)} >0$ for $\lm_{1,m}^- <\lm< \lm_{1,m}^+$. Hence
\[
0 = \int_{\Gm[1,m]} \frac{f(\lm)}{\sqrt[c]{\chi_p(\lm)} } \dlm = 2(-1)^{m+1} \int_{\lm_{1,m}^- }^{\lm_{1,m}^+} \frac{f(\lm)}{\sqrt[+]{\chi_p(\lm)}}\dlm .
\]
Since by assumption $f(\lm)\in\R$ for $\lm\in G_{1,m}$, the latter integral can only vanish if $f(\lm) \equiv 0$ on $G_{1,m}$ 
(and hence on $U_{1,m}$) or $f$ changes sign on the open interval $(\lm_{1,m}^- , \lm_{1,m}^+)$ at least once. 
If on the other hand $G_{1,m}$ is a single point, i.e., $G_{1,m} = \{\tau_{1,m}\}$, then we may extract the factor $\tau_{1,m}-\lm$ from $\sqrt[c]{\chi_p(\lm)}$ 
to obtain a Cauchy integral around $\tau_{1,m}$ which gives $f(\tau_{1,m}) = 0$. 
The case $j=2$ is treated in the same fashion, by using now \eqref{eq:kap6.34bis} of Lemma \ref{lem:DlSign}.
\end{proof}
In a first step, we study the maps $F^n$ in more detail.
\begin{lem} \label{lem:kap8.2.14}
For any $n\geq0$, the formulas \eqref{eq:kap8.2.17} - \eqref{eq:kap8.2.22} define a real analytic map
\[
F^n : \Omega_n \to \ell_{\hat n}^2 \times \ell^2.
\]
The maps $F^n$ are locally uniformly bounded. More precisely,
\[
\sum_{m\not=n} |F_{1,m}^n (v,s_1,s_2)|^2 + \sum_{m\in\Z} |F_{2,m}^n (v,s_1,s_2)|^2 \leq C
\]
where $C>0$ can be chosen uniformly in $n\geq 0$ and locally uniformly on $\Omega_n$.
\end{lem}

Before proving Lemma \ref{lem:kap8.2.14}, we make some preliminary considerations.

\begin{lem}\label{lem:kap8.2.15}
Let $m\in\Z$ and $v\in\hat W$.
\begin{equivenum}
\item If $f:U_{1,m} \to\C$ is analytic, then
\[
\frac{1}{2\pi} \big|\int_{\Gm[1,m]} \frac{f(\lm)}{w_{1,m}(\lm)} \dlm \big| \leq \max_{\lm\in G_{1,m}} |f(\lm)|.
\]
Moreover, if $v\in\Hr$ and $f$ is real analytic then there exists $\nu\in G_{1,m}$ so that
\[
\frac{1}{2\pi\ii} \int_{\Gm[1,m]} \frac{f(\lm)}{w_{1,m}(\lm)} \dlm  = - f(\nu).
\]
\item If $f:U_{2,m}\to\C$ is analytic, then
\[
\frac{1}{2\pi} \big|\int_{\Gm[2,m]} \frac{f(\lm)}{w_{2,m}(\lm)} \dlm \big| \leq \max_{\lm\in G_{2,m}} 16|\lm|^2|f(\lm)|.
\]
Moreover, if $v\in\Hr$ and $f$ is real analytic then there exists $\nu\in G_{2,m}$ so that
\[
\frac{1}{2\pi\ii} \int_{\Gm[2,m]} \frac{f(\lm)}{w_{2,m}(\lm)} \dlm  = - 16\nu^2 f(\nu).
\]
\end{equivenum}
\end{lem}
\begin{proof}
Item (i) is proved in \cite[Lemma 14.3]{nlsbook}. Considering (ii), set $\tilde\Gamma_{2,m} := \set{-\frac{1}{16\lm}}{\lm\in \Gm[2,m]}$ and $\tilde {G}_{2,m} :=\set{-\frac{1}{16\lm}}{\lm\in {G}_{2,m}}$. With the change of variable $\lm := -\frac{1}{16\mu}$ one obtains
\[
\frac{1}{2\pi\ii} \int_{\Gm[2,m]} \frac{f(\lm)}{w_{2,m}(\lm)} d\lm = \frac{1}{2\pi\ii} \int_{\tilde{\Gamma}_{2,m}} \frac{f(-\frac{1}{16\mu})}{\sqrt[s]{(\lm_{2,m}^+ - \mu)(\lm_{2,m}^- - \mu)}} \frac{\dmu}{16\mu^2}.
\]
By replacing $f$ by $f(-\frac{1}{16\mu})/16\mu^2$ in (i) one gets
\[
\frac{1}{2\pi\ii} \big|\int_{\tilde{\Gamma}_{2,m}} \frac{f(-\frac{1}{16\mu})}{\sqrt[s]{(\lm_{2,m}^+ - \mu)(\lm_{2,m}^- - \mu)}} \frac{\dmu}{16\mu^2} \big|\leq \max_{\mu\in \tilde{G}_{2,m}} \frac{|f(-\frac{1}{16\mu})|}{16|\mu|^2}
= \max_{\lm\in G_{2,m}} 16|\lm|^2|f(\lm)|
\]
and, in case $v\in\Hr$,
\[
\frac{1}{2\pi\ii} \int_{\tilde{\Gamma}_{2,m}} \frac{f(-\frac{1}{16\mu})}{\sqrt[s]{(\lm_{2,m}^+ - \mu)(\lm_{2,m}^- - \mu)}} \frac{\dmu}{16\mu^2} = - \frac{f(-\frac{1}{16\rho})} {16\rho^2}.
\]
for some $\rho \in \tilde{G}_{2,m}$. Hence $\nu = -\frac{1}{16\rho}\in G_{2,m}$ and
\[
- \frac{f(-\frac{1}{16\rho})} {16\rho^2} = -16\nu^2 f(\nu).
\]
\end{proof}
Next let us introduce some more notation. Recall that for any $n\geq0$, $f_n(\lm) \equiv f_n(s_1,s_2,\lm)$ is assumed to be of the form
\[
f_n(\lm) = -\frac{1}{\pi_n} \frac{1}{f_{n,2}(s_2,\infty)} f_{n,1}(s_1,\lm) f_{n,2}(s_2,\lm).
\]
For $j=1$ and $m\in\Z\setminus\{n\}$ define
\begin{equation}\label{eq:kap8.2.24bis}
\zeta_{1,m}^n (\lm) = \frac{\ii } {w_{1,n}(\lm)} \Big(\prod_{k\not=n,m} \frac{\sigma_{1,k}-\lm}{w_{1,k}(\lm)}\Big)
\frac{f_{n,2}(\lm) /f_{n,2}(\infty)}{\sqrt[c]{\chi_2(\lm)} / \sqrt[c]{\chi_2(\infty)}} .
\end{equation}
Then one has
\begin{equation}\label{eq:kap8.2.24}
\frac{f_n(\lm)}{\sqrt[c]{\chi_p(\lm)}} = \frac{\sigma_{1,m}-\lm}{w_{1,m}(\lm)} \zeta_{1,m}^n(\lm)\, .
\end{equation}
Note that locally uniformly on $\Omega_n$
\[
\frac{f_{n,2}(\lm)}{f_{n,2}(\infty)} = 1 + O(\frac{1}{\lm}), \qquad 
\frac{\sqrt[c]{\chi_2(\lm)}}{\sqrt[c]{\chi_2(\infty)}} = 1 + O(\frac{1}{\lm}) \qquad \text{ as } \; |\lm|\to\infty.
\]
Furthermore, setting $\sigma_{1,n} = \tau_n \in U_{1,n}$ one has $\sigma_{1,n}-\lm\not=0$ $\forall \lm\in U_{1,m}$, $m\not=n$, and
\[
\frac{n-m}{\sigma_{1,n}-\lm}\Big|_{U_{1,m}} = \frac{1}{\pi} + \ell_m^2 \quad \text{as } \; |m|\to\infty.
\]
By Corolary \ref{cor:kap6.4.1}(i) one then concludes that for $|m|\to\infty$
\[
\ii \frac{n-m}{w_{1,n}(\lm)} \prod_{k\not=n,m} \frac{\sigma_{1,k}-\lm}{w_{1,k}(\lm)}\Big|_{U_{1,m}} = \ii \frac{n-m}{\sigma_{1,n}-\lm} \prod_{k\not=m} \frac{\sigma_{1,k}-\lm}{ w_{1,k}(\lm)} \Big|_{U_{1,m}} = \frac{\ii}{\pi} + \ell_m^2
\]
locally uniformly on $\Omega_n$. Altogether, we thus have proved that
\begin{equation}\label{eq:kap8.2.25}
(n-m)\zeta_{1,m}^n(\lm) \Big|_{U_{1,m}} = \frac{\ii}{\pi} + \ell_m^2 \quad \text{as } \; |m|\to\infty
\end{equation}
locally uniformly on $\Omega_n$. In the case $j=2$, one obtains a similar estimate. For $m\in\Z$ we write
\begin{equation}\label{eq:kap8.2.26}
\pi_n \frac{f_n(\lm)}{\sqrt[c]{\chi_p(\lm)}} = \frac{\sigma_{2,m} + \frac{1}{16\lm}}{w_{2,m}(\lm)} \zeta_{2,m}^n(\lm)
\end{equation}
where
\begin{equation}\label{eq:kap8.2.26bis}
\zeta_{2,m}^n (\lm) = \ii \frac{f_{n,1}(\lm)}{\sqrt[c]{\chi_1(\lm)} /\sqrt[c]{\chi_1(0)}} \frac{1}{f_{n,2}(\infty)} \prod_{k\not=m} \frac{\sigma_{2,k} + \frac{1}{16\lm}}{w_{2,k}(\lm)}
\end{equation}
Here we used that by \eqref{eq:kap6.3.401}, $\sqrt[c]{\chi_1(0)} = \sqrt[c]{\chi_2(\infty)}$. Note that for $\lm$ near $0$,
\[
f_{n,1}(\lm) = f_{n,1}(0) + O(\lm), \quad \sqrt[c]{\chi_1(\lm)}/ \sqrt[c]{\chi_1(0)} = 1+ O(\lm).
\]
One then again concludes from Corollary \ref{cor:kap6.4.1}(i) that
\begin{equation}\label{eq:kap8.2.27}
\zeta_{2,m}^n \Big|_{U_{2,m}} = \ii \frac{f_{n,1}(0)}{f_{n,2}(\infty)} + \ell_m^2 \quad \text{as}\; m\to\infty
\end{equation}
locally uniformly on $\Omega_n$.

\begin{proof}[Proof of Lemma \ref{lem:kap8.2.14}] Let $n\geq0$. We first consider the case $F_1^n$. For $(v,s_1,s_2)\in \Omega_n$ we have by \eqref{eq:kap8.2.20}, \eqref{eq:kap8.2.24} for $m\not=n$
\begin{equation}\label{eq:kap8.2.28}
F_{1,m}^n = (n-m) \int_{\Gm[1,m]} \frac{\sigma_{1,m}-\lm}{w_{1,m}(\lm)} \zeta_{1,m}^n (\lm)\dlm.
\end{equation}
By Lemma \ref{lem:standardRootAnalyticSizeinverseintegral}, $w_{1,m}(\lm)$ is analytic on $U_{1,m}'\times V_{v_0}$ with $U_{1,m}',$ $V_{v_0}$ 
given as at the beginning of this section (maybe after shrinking $V_{v_0}$). 
By Lemma \ref{lem:fn} and \cite[Lemma C.1]{nlsbook}, $\zeta_{1,m}^n$ is analytic on $U_{1,m}'\times (\Omega_n\cap (V_{v_0}\times \ell_{\hat n}^2 \times \ell^2))$. 
Hence the integrand in \eqref{eq:kap8.2.28} is analytic on $U_{1,m}'\times (\Omega_n\cap (V_{v_0}\times \ell_{\hat n}^2 \times \ell^2))$ for any $m\not=n$ and $F_{1,m}^n$ is analytic on $\Omega_n$. 
Moreover, by \eqref{eq:kap8.2.25} and Lemma \ref{lem:kap8.2.15},  the right hand side of \eqref{eq:kap8.2.28} is of the order of  $\max_{\lm\in G_{1,m}} |\sigma_{1,m}-\lm|$. 
Since $\sigma_{1,m} = m\pi + \ell_m^2$, $\max_{\lm\in G_{1,m}} |\tau_{1,m} - \lm| = |\gm[1,m]|/2 = \ell_m^2$ and $\tau_{1,m} = m\pi + \ell_m^2$ (cf. \cite[Lemma 3.17]{LOperator2018}) 
one has locally uniformly $\max_{\lm\in G_{1,m}} |\sigma_{1,m} - \lm| = \ell_m^2$. 
Altogether we thus have proved that $F_1^n : \Omega_n \to \ell_{\hat n}^2$ is locally bounded and, by \cite[Theorem A.5]{nlsbook}, analytic on  $\Omega_n$. 
Now let us turn to $F_2^n$. Making the change of variable $\lm = -\frac{1}{16\mu}$, it follows by \eqref{eq:kap8.2.21},\eqref{eq:kap8.2.26} that on $\Omega_n$ one has for any $m\in\Z$, 
\begin{equation}\label{eq:kap8.2.29}
 F_{2,m}^n = 16\pi_m^2 \int_{\Gm[2,m]} \frac{\sigma_{2,m}  + \frac{1}{16\lm}}{w_{2,m}(\lm)} \zeta_{2,m}^n (\lm) \dlm = 
 16\pi_m^2 \int_{\tilde{\Gamma}_{2,m}} \frac{\sigma_{2,m}^n - \mu}{w_{2,m}(-\frac{1}{16\mu})} \zeta_{2,m}^n (-\frac{1}{16\mu}) \frac{\dmu}{16\mu^2}.
\end{equation}
Arguing as in the case of $F_{1,m}^n$ one sees that $F_{2,m}^n$ is analytic on $\Omega_n$. 
Moreover by \eqref{eq:kap8.2.27} and Lemma \ref{lem:kap8.2.15} one concludes 
that the right hand side of \eqref{eq:kap8.2.29} is of the order $\max_{\lm\in G_{2,m}} \big( \pi_m^2 |\lm|^2 |\sigma_{2,m} + \frac{1}{16\lm}|\big)$. 
Since $\sigma_{2,m} =m\pi + \ell_m^2$, $\tau_{2,m} = (\lm_{2,m}^+ + \lm_{2,m}^-)/2 = m\pi + \ell_m^2$, $\lm_{2,m}^+ - \lm_{2,m}^- = \ell_m^2$ 
one sees that $-\frac{1}{16\lm} = m\pi + \ell_m^2 $ for $\lm\in G_{2,m}$ and $\max_{\lm\in G_{2,m}} \pi_m^2 |\lm|^2 |\sigma_{2,m} + \frac{1}{16\lm}| = \ell_m^2$. 
Altogether we thus have proved that $F_2^n:\Omega_n\to \ell^2$ is locally bounded and, by \cite[Theorem A.5]{nlsbook}, analytic on $\Omega_n$.  
Next we prove that $F_{1,m}^n,$ $F_{2,m}^n$ are real valued on $\Hr\times \ell_{\R,\hat n}^2\times \ell_{\R}^2$ where
\[
\ell_{\R,\hat n}^2 : = \ell^2(\Z\setminus\{n\}, \;\R).
\]
Recall that for any $v\in\Hr$, $k\in\Z$, and $j=1,2$, $\lm_{j,k}^\pm \in\R$
and $w_{j,k}(\lm)\in\R$ for any $\lm\in \R\setminus G_{j,k}$, whereas for $\lm\in G_{j,k}$ one has $\ii w_{j,k}(\lm)\in \R.$
It then follows that $F_{j,m}^n$ are real valued on $\Hr\times \ell_{\R,\hat n}^2 \times \ell_{\R}^2$.
Indeed, in the case $\lm_{j,m}^+ = \lm_{j,m}^-$, this follows by Cauchy's integral, 
whereas in the case $\lm_{j,m}^- < \lm_{j,m}^+$, this can be seen by deforming the contour $\Gm[j,m]$ to the interval $G_{j,m} \subset \R$. 
To establish the statement on the uniform bounds for $F^n$ with respect to $n$, it  remains to consider $n$ large. 
Given $v\in\hat W$ and $(s_1,s_2)\in \ell^2\times \ell^2$, let $k_0\geq 1$ be such that
\[
U_{1,k} = D_k,\qquad U_{2,k} = -D_{-k},  \qquad \forall |k|\geq k_0\, ,
\]
and
\[
 |s_{j,k} | < \pi/4 \, , \qquad \forall |k|\geq k_0 , \,\, j=1,2.
\]
Then $\sigma_{j,k} = k\pi + s_{j,k}$, $j=1,2$, satisfy for any $|k|\geq k_0$,
\[
\sigma_{1,k} \in U_{1,k}, \quad -\frac{1}{16\sigma_{2,k}} \in U_{2,k}.
\]
According to (I-2) and (I-3) (cf. Section \ref{sec:realAmostReal}), for any $|k|\geq k_0$
\begin{equation} \label{eq:kap8.2.29.1}
|\sigma_{1,k}-\lm| \geq |k-m|/c \quad \forall \lm\in U_{1,m}, \; \forall m\not=k
\end{equation}
\begin{equation} \label{eq:kap8.2.29.2}
|\sigma_{2,k}+\frac{1}{16\lm}| \geq |k-m|/c \quad \forall \lm\in U_{2,m}, \; \forall m\not=k.
\end{equation}
To estimate $F_{1,m}^n$, write the integrand in \eqref{eq:kap8.2.28} in the form
\begin{equation}\label{eq:kap8.2.29.3}
(n-m) \frac{\sigma_{1,m}-\lm}{w_{1,m}(\lm)} \zeta_{1,m}^n (\lm) = \frac{\sigma_{1,m}-\lm}{w_{1,m}(\lm)} \frac{n-m}{\sigma_{1,n} - \lm} \ii \zeta_{1,m}(\lm)
\end{equation}
where
\begin{equation}\label{eq:kap8.2.29.3bis}
\zeta_{1,m}(\lm) : = \big(\prod_{k\not=m} \frac{\sigma_{1,k} - \lm}{w_{1,k}(\lm)} \big) \frac{f_{n,2}(\lm) / f_{n,2}(\infty)}{\sqrt[c]{\chi_2(\lm)} / \sqrt[c]{\chi_2(\infty)}}.
\end{equation}
Since we only need to consider the case where $n$ is  large we can assume that $n\geq k_0$.
In this case $\big|\frac{n-m}{\sigma_{1,n}-\lm}\big|\leq C$ for any $\lm\in U_{1,m}$ with $m\not=n$ and, by Corollary \ref{cor:kap6.4.1},
\begin{equation}\label{eq:kap8.2.29.4}
\sup_{\lm\in U_{1,m}} |\zeta_{1,m}(\lm) -1 | = \ell_m^2
\end{equation}
locally uniformly on $\Omega\cap (V_{v_0} \times \ell^2 \times \ell^2)$. Arguing as in the first part of the proof, 
one then concludes that
\[
\sup_{n\in\Z\setminus\{m\}} |F_{1,m}^n| = \ell_m^2
\]
locally uniformly on $\Omega \cap (V_{v_0} \times \ell^2\times \ell^2)$. The corresponding estimate for $F_{2,m}^n$ follows from \eqref{eq:kap8.2.27} and Corollary \ref{cor:kap6.4.1}. In this way one obtains that $\sup_{n} |F_{2,m}^n| = \ell_m^2$ locally uniformly on $\Omega \cap(V_{v_0}\times \ell^2 \times \ell^2).$
\end{proof}

With the application of the implicit function theorem in mind we now investigate the differential of the maps $F^n$ at a point $(v,s_1,s_2)$. For our purposes it suffices to restrict ourselves to the open subset $\Omega_r: = \Omega\cap (\Hr \times \ell_{\R}^2\times \ell_{\R}^2)$. We now compute the differential of $F^n$ with respect to $s_1,$ $s_2$ at any point of $\Omega_r$. By the analyticity of $F^n$ this is a bounded linear operator
\[
Q^n : \ell_{\hat n}^2 \times \ell^2 \to \ell_{\hat n}^2 \times \ell^2
\]
which is represented by an infinite matrix $\begin{pmatrix} Q_{11}^n & Q_{12}^n \\ Q_{21}^n & Q_{22}^n\end{pmatrix}$ where
\begin{align*}
Q_{11,mr}^n : =& \frac{\partial F_{1,m}^n}{\partial s_{1,r}} \quad (m,r\not=n),
& Q_{12,mr}^n : =& \frac{\partial F_{1,m}^n}{\partial s_{2,r}} \quad (m\not=n),
\\ Q_{21,mr}^n :=& \frac{\partial F_{2,m}^n }{\partial s_{1,r}} \quad (r\not=n),
& Q_{22,mr}^n :=& \frac{\partial F_{2,m}^n}{\partial s_{2,r}} \quad (m,r\in\Z).
\end{align*}

\begin{lem} \label{lem:kap8.2.16}
Let $n \ge 0$. On $\Omega_r$, the matrix elements of $Q^n$ are real and satisfy:
\[
(i)\quad
0\not= Q_{11,mm}^n = 2 + \ell_m^2 \; (m\not=n),
\quad Q_{11,mr}^n =\frac{\ell^2_m}{|m-r|} \; (m \not= r, \  m, r\not=n),
\quad Q_{12,mr}^n= \frac{\ell_m^2}{\langle r\rangle^2}\; (m\not=n).
\]
\[
(ii)\quad
0\not= Q_{22,mm}^n = 2\pi \frac{f_{n,1}(0)}{f_{n,2}(\infty)} + \ell_m^2,
\quad Q_{22,mr}^n = \frac{\ell^2_m}{|m-r|}+\frac{\ell_m^2}{\langle r\rangle}\;  (m \not= r),
\quad  Q_{21,mr}^n= \frac{\ell_m^2}{\langle r\rangle} \; (m,r\in\Z).
\]
The estimates of $Q_{11,mr}^n$  ($m \ne r, \ m, r \ne n$),  $Q_{22,mr}^n$  ($m \ne r$),   $Q_{12,mr}^n$ ($m \ne n$),  and 
$Q_{21,mr}^n$  are uniform with respect to $r \in \Z$.
\end{lem}
\begin{proof}
By Lemma \ref{lem:kap8.2.14}, all coefficients of $Q^n$ are real valued on $\Omega_r$. Towards (i) let us first consider $(Q_{11}^n)_{mr}$ with $m,r\not=n$. By \eqref{eq:kap8.2.24bis} the term $\prod_{k\not=n,m} \frac{\sigma_{1,k}-\lm}{w_{1,k}(\lm)}$ of $\zeta_{1,m}^n$ is the only one which depends on $s_1$. Since for $r\not=n,m$
\[
\partial_{s_{1,r}} (\frac{\sigma_{1,r}-\lm}{w_{1,r}(\lm)}) =  \frac{\sigma_{1,r}-\lm}{w_{1,r}(\lm)} \frac{1}{\sigma_{1,r}-\lm}
\]
one has
\[
\partial_{s_{1,r}} \zeta_{1,m}^n (\lm) = \frac{1}{\sigma_{1,r}-\lm} \zeta_{1,m}^n(\lm).
\]
By \eqref{eq:kap8.2.28}, one then concludes that for $r\not=m,n$
\begin{equation}\label{eq:kap8.2.301}
Q_{11,mr}^n = (n-m) \int_{\Gm[1,m]} \frac{\sigma_{1,m}-\lm}{w_{1,m}(\lm)} \frac{1}{\sigma_{1,r}-\lm} \zeta_{1,m}^n(\lm) \dlm.
\end{equation}
By Lemma \ref{lem:kap8.2.15} and \eqref{eq:kap8.2.29.1}, \eqref{eq:kap8.2.29.3}, \eqref{eq:kap8.2.29.4},
$\sup_{n\in\Z\setminus \{m,r\}} |Q_{11,mr}^n|$ is of the order
\[
\max_{\lm\in G_{1,m}} | (\sigma_{1,m}-\lm) \frac{1}{\sigma_{1,r}-\lm}| = \frac{\ell_m^2}{|m-r|}
\]
locally uniformly on $\Omega_r$. For $m=r$ with $m\not=n$ the same arguments lead to
\[
Q_{11,mm}^n = (n-m) \int_{\Gm[1,m]} \frac{\zeta_{1,m}^n(\lm)}{w_{1,m}(\lm)} \dlm.
\]
Again using Lemma \ref{lem:kap8.2.15} one sees that there exists $\nu\in G_{1,m}$ so that 
$$
Q_{11,mm}^n = -2\pi\ii (n-m) \zeta_{1,m}^n(\nu) .
$$ 
By the estimate \eqref{eq:kap8.2.25} it then follows that $Q_{11,mm}^n = 2 + \ell_m^2$. 
By inspection of \eqref{eq:kap8.2.24bis} one concludes that $Q_{11,mm}^n\not=0$ for any $m\not=n$. 
Next let us estimate $Q_{12,mr}^n$. By \eqref{eq:kap8.2.24bis}, $f_{n,2}(\lm) = \prod_{k\in\Z} \frac{\sigma_{2,k} + \frac{1}{16\lm} }{\pi_k}$ and $f_{n,2}(\infty)^{-1}$ 
are the only terms in $\zeta_{1,m}^n(\lm)$ which depend on $s_2$. Since $f_{n,2}(\infty) = \prod_{k\not=n} \frac{\sigma_{2,k}}{w_{2,k}(0)}$, one has for $r\not=n$
\[
\partial_{s_{2,r}} (f_{n,2}(\infty))^{-1} = - (f_{n,2}(\infty))^{-1}\frac{1}{\sigma_{2,r}}, \qquad 
\partial_{s_{2,r}} f_{n,2}(\lm) = f_{n,2}(\lm) \frac{1}{\sigma_{2,r} + \frac{1}{16\lm}} ,
\]
implying that
\begin{equation}\label{eq:kap8.2.302}
Q_{12,mr}^n = (n-m) \int_{\Gm[1,m]} \frac{\sigma_{1,m}-\lm}{w_{1,m}(\lm)} \big(\frac{1}{\sigma_{2,r}+\frac{1}{16\lm}} - \frac{1}{\sigma_{2,r}}\big) \zeta_{1,m}^n(\lm) \dlm.
\end{equation}
For $|m|$ sufficiently large, $U_m = D_m$ implying that $1/16\lm \in U_{-|m|}$ for any $\lm\in U_{|m|}$. 
Hence for such $m$,  $|\sigma_{2,r}+ \frac{1}{16\lm}| \geq \frac{1}{C}\frac{1}{\langle r\rangle}$ for any $\lm\in U_{1,m}$ and $r\in\Z$. 
A similar estimate holds for $|r|$ sufficiently large and any $m\in\Z$. Hence for any $\lm\in G_{1,m}$,
\[
\big| \frac{1}{\sigma_{2,r} + \frac{1}{16\lm}} - \frac{1}{\sigma_{2,r}} \big| \leq C\frac{1}{\langle r\rangle^2}.
\]
Again using Lemma \ref{lem:kap8.2.15} and estimate \eqref{eq:kap8.2.25} one sees that $\sup_{n\in\Z\setminus\{m\}} |Q_{12,mr}^n| = \frac{1}{\langle r\rangle^2} \ell_m^2$.

\noindent (ii) First let us consider $Q_{22,mr}^n$ ($m,r\in\Z$). Arguing as in item (i) one sees that for $m=r$ one has by \eqref{eq:kap8.2.26bis} and \eqref{eq:kap8.2.29}
\begin{equation}\label{eq:kap8.2.303}
Q_{22,mm}^n = 16\pi_m^2 \int_{\Gm[2,m]} (1-\frac{\sigma_{2,m}+ \frac{1}{16\lm} }{\sigma_{2,m}} ) \frac{\zeta_{2,m}^n(\lm)}{w_{2,m}(\lm)} \dlm = 16\pi_m^2 \int_{\tilde{\Gamma}_{2,m}} \frac{\mu}{\sigma_{2,m}} \frac{\zeta_{2,m}^n (-\frac{1}{16\mu})}{ w_{2,m}(-\frac{1}{16\mu})} \frac{\dmu}{16\mu^2}.
\end{equation}
Again using \eqref{eq:kap8.2.29} and Lemma \ref{lem:kap8.2.15} one concludes that there exists $\nu\in G_{2,m}$ so that
\[
Q_{22,mm}^n = -2\pi \ii \frac{-\frac{1}{16\nu}}{\sigma_{2,m}} \zeta_{2,m}^n(\nu) 16\nu^2 \cdot 16\pi_m^2
\]
which by \eqref{eq:kap8.2.26bis} does not vanish. Note that since $\nu\in G_{2,m}$
\[
\frac{-\frac{1}{16\nu}}{\sigma_{2,m}} = 1+ \ell_m^2 , \qquad 16\nu^2 \cdot 16\pi_m^2 = 1+ \frac{1}{m} \ell_m^2.
\]
Hence by \eqref{eq:kap8.2.27} one has $Q_{22,mm}^n = 2\pi \frac{f_{n,1}(0)}{f_{n,2}(\infty)} + \ell_m^2$ as claimed. Next let us estimate $Q_{22,mr}^n$ for $m\not=r$. Arguing as in the proof of item (i) one sees that
\[
Q_{22,mr}^n = 16\pi_m^2 \int_{\Gm[2,m]} \frac{\sigma_{2,m} + \frac{1}{16\lm}}{w_{2,m}(\lm)} \big(\frac{1}{\sigma_{2,r} + \frac{1}{16\lm}} - \frac{1}{\sigma_{2,r}}\big) \zeta_{2,m}^n(\lm)\dlm.
\]
Since for $\lm\in G_{2,m}$, $-\frac{1}{16\lm} = m\pi + \ell_m^2,$ one has
\[
|\frac{1}{\sigma_{2,r} + \frac{1}{16\lm}}| \leq C\frac{1}{|r-m|}, \qquad \quad
|\frac{1}{\sigma_{2,r}}| \leq \frac{1}{\langle r\rangle}.
\]
By the change of variable $\lm= - \frac{1}{16\mu}$ one deduces from \eqref{eq:kap8.2.27} and Lemma \ref{lem:kap8.2.15} that
\[
\sup_n |Q_{22,mr}^n| = \frac{1}{|m-r|} \ell_m^2 + \frac{1}{\langle r\rangle} \ell_m^2
\]
as claimed. Finally we estimate $Q_{21,mr}^n$. In this case one has
\[
Q_{21,mr}^n = 16\pi_m^2 \int_{\Gm[2,m]} \frac{\sigma_{2,m} + \frac{1}{16\lm}}{w_{2,m}(\lm)} \frac{1}{\sigma_{1,r} + \lm} \zeta_{2,m}^n (\lm) \dlm.
\]
Since $|\sigma_{1,r} + \lm | \geq \frac1C\langle r\rangle$ for $\lm\in G_{2,m}$ one gets
\[
|\frac{1}{\sigma_{1,r} + \lm} | \leq C\frac{1}{\langle r\rangle}.
\]
Again using  \eqref{eq:kap8.2.27}  and Lemma \ref{lem:kap8.2.15} one sees that $\sup_n | Q_{21,mr}^n| = \frac{1}{\langle r\rangle }\ell_m^2$ as claimed.
\end{proof}

\begin{lem}\label{lem:kap8.2.17}
Let $n \ge 0$.
For any $(v,s_1,s_2)\in \Omega_r$, the differential $Q^n$ of $F^n$ with respect to $(s_1,s_2)$ is of the form
\[
Q^n = D^n + K^n: \ell_{\hat n}\times \ell^2 \to \ell_{\hat n}^2 \times \ell^2
\]
where $D^n$ is a linear isomorphism in diagonal form and $K^n$ a compact operator.
\end{lem}
\begin{proof} Let $D^n$ be the diagonal of $Q^n$,  
$$
D^n = \text{diag}((Q_{11,mm}^n)_{m\not=n}, (Q_{22,mm}^n)_{m\in\Z}).
$$ 
By the preceding lemma, any diagonal entry is nonzero and 
$$
\lim_{|m|\to\infty} Q_{11,mm}^n =2 \ ,  \qquad
\lim_{|m|\to\infty} Q_{22,mm}^n =2\pi  \frac{f_{n,1}(0)}{f_{n,2}(\infty)}.
$$ 
Since $\sigma_{j,k}^n \in \C^*$ and $\sigma_{j,k}^n = \pi k + \ell_k^2$ one has
\[
f_{n,2} (\infty) = \prod_{k\in\Z} \frac{\sigma_{2,k}^n}{\pi_k} \in \C^*, \qquad 
f_{n,1}(0) = \prod_{k\not=n} \frac{\sigma_{1,k}^n}{\pi_k} \in\C^*.
\]
Hence $D^n : \ell_{\hat n}^2 \times \ell^2 \to \ell_{\hat n}^2 \times \ell^2 $ is a linear isomorphism. 
Moreover, $K^n:= Q^n - D^n$ is a bounded linear operator on $\ell_{\hat n}^2 \times \ell^2 $ with vanishing diagonal elements. 
By the estimates of Lemma \ref{lem:kap8.2.16}, $K^n$ is Hilbert-Schmidt and hence compact.
\end{proof}
 \begin{lem} \label{lem:kap8.2.18}
 Let $n \ge 0$.
 At any point $(v,s_1,s_2)\in \Omega_r$, $Q^n$ is one-to-one on $\ell_{\hat n}^2 \times \ell^2 $.
 \end{lem}
\begin{proof}
Suppose that $Q^nh =0$ for some $h=(h_1,h_2)\in \ell_{\hat n}^2 \times \ell^2 $. Recall that $f_n(s_1, s_2,\lm)$ is analytic on $\ell_{\hat n}^2 \times \ell^2 \times\C^*$. 
In particular, $\phi_n(\lm) : = \partial_\epsilon |_{\epsilon=0} f_n((s_1,s_2) + \epsilon (h_1,h_2),\lm)$ is analytic on $\C^*$. By assumption we have for any $m\not=n$
\[
0 = \sum_{r\not=n} Q_{11,mr}^n h_{1,r} + \sum_{r\in\Z} Q_{12,mr}^n h_{2,r} = (n-m) \int_{\Gm[1,m]} \frac{\phi_n(\lm)}{\sqrt[c]{\chi_p(\lm)}} \dlm
\]
and for any $m\in\Z$
\[
0 = \sum_{r\in\Z} Q_{21,mr}^n h_{1,r} + \sum_{r\in\Z} Q_{22,mr}^n h_{2,r} = 16\pi_m^2\pi_n  \int_{\Gm[2,m]} \frac{\phi_n(\lm)}{\sqrt[c]{\chi_p(\lm)}} \dlm.
\]
We will use Lemma \ref{ap:lem:interpolation} (interpolation lemma) to show that $\phi_n\equiv 0$ which will imply that $h=0$. 
Since $\phi_n(\lm)$ is real for $\lm\in\R^*$, it follows from Lemma \ref{lem:kap8.2.18.5} that $\phi_n(\lm)$ has roots $\rho_{1,m}\in G_{1,m}$ $(m\not=n)$, 
and $-\frac{1}{16\rho_{2,m}}\in G_{2,m}$ ($m\in\Z$). On the other hand, for $\lm\in\partial B_N$ or $\lm\in\partial B_{-N}$ with $N$ sufficiently large, 
$\sigma_{1,r} - \lm \not=0$, $\sigma_{2,r} + \frac{1}{16\lm} \not=0$ for any $r\not=0$. Therefore $\phi_n(\lm)$ for $\lm\in\partial B_N$ or $\lm\in\partial B_{-N}$ can be written as
\begin{align*}
\phi_n(\lm) =& \sum_{r\not=n} \partial_{s_{1,r}} (f_n(\lm)) h_{1,r} + \sum_{r\in\Z} \partial_{s_{2,r}}(f_n(\lm))h_{2,r}
\\=& f_n(\lm) \sum_{r\not=n} \frac{1}{\sigma_{1,r}-\lm} h_{1,r} + f_n(\lm) \sum_{r\in\Z} \big(\frac{1}{\sigma_{2,r} + \frac{1}{16\lm} } - \frac{1}{\sigma_{2,r}}\big) h_{2,r}.
\end{align*}
We want to apply Lemma \ref{ap:lem:interpolation} to
\[
\phi(\lm) : = \phi_n(\lm) (\sigma_{1,n}-\lm).
\]
We now verify that its assumptions are satisfied. First we show that $\sup_{\lm\in\partial B_N} |\frac{\phi(\lm)}{\sin(\lm)}| \to0$ as $N\to\infty$. By the product representation of $\sin(\lm)$,
$\sin(\lm) = -\prod_{m\in\Z} \frac{m\pi-\lm}{\pi_m},$ one gets for $\lm\in\partial B_N$
\[
f_n(\lm) \cdot \frac{\sigma_{1,n}-\lm}{\sin(\lm)} = -\frac{f_{n,2}(\lm)}{f_{n,2}(\infty)} \frac{f_{n,1}(\lm)}{\sin(\lm)} \frac{\sigma_{1,n}-\lm}{\pi_n} = \frac{ f_{n,2}(\lm)}{f_{n,2}(\infty)} \prod_{k\in\Z} \frac{\sigma_{1,k}-\lm}{k\pi -\lm}.
\]
It implies that the quotient of $\phi(\lm) = \phi_n(\lm)(\sigma_{1,n}-\lm)$ with $\sin(\lm)$ can be written as
\[
\frac{\phi(\lm)}{\sin(\lm)} = \phi_n(\lm) \frac{\sigma_{1,n}-\lm}{\sin(\lm)} = \frac{ f_{n,2}(\lm)}{f_{n,2}(\infty)} \big(\prod_{k\in\Z} \frac{\sigma_{1,k} - \lm}{k\pi - \lm}\big) \cdot \sum_{r\not=n} \big(\frac{1}{\sigma_{1,r}-\lm} h_{1,r} + \sum_{r\in\Z} \big(\frac{1}{\sigma_{2,r} + \frac{1}{16\lm}} -\frac{1}{\sigma_{2,r}}\big) h_{2,r}.
\]
By \cite[Lemma C.5]{nlsbook}
\[
\sup_{\lm\in\partial B_N} \big|\prod_{k\in\Z} \frac{\sigma_{1,k}-\lm}{k\pi - \lm} \big| = O(1)  \quad \text{as} \; N\to\infty.
\]
Since $f_{n,2}$ is analytic near $\lm=\infty$ and $f_{n,2}(\infty)\not=0$ one has
\[  
\sup_{\lm \in \partial B_N} | \frac{ f_{n,2}(\lm)}{f_{n,2}(\infty)} | = O(1) \qquad  \text{as} \quad N \to \infty\, .
\]
Furthermore,
\[
\sup_{\lm\in\partial B_N} \big| \sum_{r\not=n} \frac{1}{\sigma_{1,r}-\lm} h_{1,r}\big| \leq 
C\big(\sum_{|r|\leq N/2} \frac{1}{(r-N)^2}\big)^{1/2} + C\big( \sum_{|r|>N/2} |h_{1,r}|^2\big)^{1/2} = o(1) \quad \text{as }\; N\to\infty
\]
and since 
$$
\frac{1}{\sigma_{2,r}+\frac{1}{16\lm}} - \frac{1}{\sigma_{2,r}} = \frac{\frac{1}{16\lm}}{(\sigma_{2,r} + \frac{1}{16\lm})\sigma_{2,r}}
$$ 
one has
\[
\sup_{\lm\in\partial B_N} \big| \sum_{r\in\Z} \big(\frac{1}{\sigma_{2,r} + \frac{1}{16\lm}} -\frac{1}{\sigma_{2,r}}\big) h_{2,r}\big| \leq 
C\sup_{\lm\in\partial B_N} \frac{1}{16|\lm|} \big(\sum_{r\in\Z} \frac{1}{\langle r\rangle^2} |h_{2,r}| \big) \to 0\quad \text{as} \;  N\to \infty.
\]
Altogether it thus follows that
\[
\sup_{\lm\in\partial B_N}  \big|\frac{\phi(\lm)}{\sin(\lm)}\big| = o(1) \quad \text{as}\; N\to\infty.
\]
Similarly we estimate
$$
 \sup_{\lm\in\partial B_{-N}} \Big|\frac{\phi(\lm)}{\sin(-\frac{1}{16\lm})}\Big| =
 \sup_{\mu\in\partial B_N} \Big| \frac{\phi(-\frac{1}{16\mu})}{\sin(\mu)} \Big| \, .
 $$
 First note that the latter expression equals
$$
\sup_{\mu\in\partial B_N} \Big( \Big| \frac{\sigma_{1,n} + \frac{1}{16\mu}}{f_{n,2}(\infty)} \Big|  \Big| \frac{f_{n,2}(-\frac{1}{16\mu})}{\sin(\mu)} \Big|  \Big| f_{n,1} (-\frac{1}{16\mu}) \Big|\Big) 
 \cdot \Big( \big| \sum_{r\not=n} \frac{1}{\sigma_{1,r} + \frac{1}{16\mu}} h_{1,r}\big| + \big| \sum_{r\in\Z} \big(\frac{1}{\sigma_{2,r}-\mu} - \frac{1}{\sigma_{2,r}}\big) h_{2,r}\big| \Big).
$$
Since $|\frac{1}{\sigma_{1,r} + \frac{1}{16\mu}}| \leq C\frac{1}{\langle r\rangle}$ we get
\[
\sum_{r\not=n}\big|  \frac{1}{\sigma_{1,r} + \frac{1}{16\mu}} h_{1,r}\big|  \leq \Big(\sum_{r\not=n} \frac{1}{|\sigma_{1,r} + \frac{1}{16\mu}|^2} \Big)^{1/2} \Big(\sum_{r\in\Z} |h_{1,r}|^2\Big)^{1/2} = O(1)
\]
and since
\[
|\frac{1}{\sigma_{2,r}-\mu} - \frac{1}{\sigma_{2,r}} | \leq \frac{1}{|\sigma_{2,r}-\mu|} + \frac{1}{|\sigma_{2,r}|} \leq C\Big( \frac{1}{|r-(N+1/2)|} + \frac{1}{|r+(N+1/2)|} + \frac{1}{\langle r\rangle}\Big)
\]
it follows that
\[
 \sum_{r\in\Z} \big|\frac{1}{\sigma_{2,r}-\mu} - \frac{1}{\sigma_{2,r}}\big| |h_{2,r}|
 \leq C \big(\sum_{r} |h_{2,r}|^2\big)^{1/2}.
\]
Altogether this shows
\[
\sup_{\lm\in\partial B_{-N}} |\frac{\phi(\lm)}{\sin(\lm)}| = O(1) \text{ (and hence } o(N) )\quad \text{as} \; N\to\infty.
\]
We thus can apply Lemma \ref{ap:lem:interpolation} to $\phi(\lm)$. Since $\phi(\rho_{1,m}) = 0$ ($m\not=n$), $\phi(-(16\rho_{2,m})^{-1}) = 0$ ($m\in\Z$), 
and $\rho_{j,m} = m\pi + \ell_r^2$ it then follows that $\phi\equiv 0$ and hence $\phi_n \equiv 0$, or for any $\lm\in\C^*$
\begin{equation}\label{eq:kap8.2.40}
 0= \sum_{r\not=n} \frac{f_n(\lm) }{\sigma_{1,r}-\lm} h_{1,r} + \sum_{r\in\Z} \frac{-\frac{1}{16\lm}}{(\sigma_{2,r} + \frac{1}{16\lm}) \sigma_{2,r}} f_n(\lm) h_{2,r}.
\end{equation}
Evaluating the right hand side of \eqref{eq:kap8.2.40} at $\lm=\sigma_{1,m}$ for $m\not=n$ one obtains $0 = \dot f_n(\sigma_{1,m}) h_{1,m}$. 
Since $\sigma_{1,m}$ is a simple root of $f_n(\lm)$, $\dot f_n(\sigma_{1,m})\not=0$ and hence $h_{1,m} = 0$. 
Similarly, evaluating the right hand side of \eqref{eq:kap8.2.40} at $\kappa_{2,m} = -(16\sigma_{2,m})^{-1}$ for any given $m\in\Z$ 
one obtains $0 = \dot f_n(\kappa_{2,m}) h_{2,m},$ yielding $h_{2,m}=0$. Altogether we have shown that $h=0$ and hence $Q^n$ is one-to-one.
\end{proof}

The latter two lemmas together with the Fredholm alternative yield
\begin{cor}\label{cor:kap8.2.19} Let $n \ge 0$.
At any point in $\Omega_r$, $Q^n$ is a linear isomorphism of $\ell_{\hat n}^2 \times \ell^2.$
\end{cor}

Lemma \ref{lem:kap8.2.14} and Corollary \ref{cor:kap8.2.19} allow to apply the implicit function theorem to any particular  solution 
$s^n= (s_1^n,s_2^n)$ of $F^n(v,s_1,s_2)=0$ in $\Omega_{r,n} = \Omega_r\cap \Omega_n$.

\begin{prop}\label{prop:kap8.2.20}
For any $n\geq0$, there exists a real analytic map
\[
s^n= (s_1^n,s_2^n) : \Hr \to \ell_{\hat n}^2 \times \ell^2
\]
with graph in $\Omega_{r,n}$ so that $F^n(v,s^n(v))=0$ for any $v\in\Hr$. 
Actually, for any $v\in\Hr$, $\sigma_{1,m}^n \in G_{1,m}(v)$, for any $m\not=n$ and $\sigma_{2,m}^n\in G_{2,m}(v)$ for any $m\in\Z$. 
The map $s^n$ is unique within the class of all such real analytic maps with graph in $\Omega_{r,n}.$ For $v=0$,
\begin{equation}\label{eq:kap8.2.42}
\sigma_{1,k}^n = k\pi + s_{1,k}^n = \tau_{1,k} \;(k\not=n),
\qquad \sigma_{2,k}^n = k\pi + s_{2,k}^n = \tau_{2,k} \;(k\in\Z).
\end{equation}
\end{prop}
\begin{proof}
First we note that any solution of $F^n(v,s) =0$ in $\Omega_{r,n}$ satisfies
\begin{equation}\label{eq:kap8.2.41}
\sigma_{1,m} = m\pi + s_{1,m} \in G_{1,m}(v),  \,\, \forall m\not=n , \qquad 
\sigma_{2,m} = m\pi + s_{2,m} \in G_{2,m} (v), \,\,  \forall m\in\Z ,
\end{equation}
where $s=(s_1,s_2)\in \ell_{\hat n}^2 \times \ell^2$. Indeed, by Lemma \ref{lem:kap8.2.18.5}, for any $m\not=n$, $f_{1,n}(s_1,\lm)$ 
has a root $\rho_{1,m} \in G_{1,m}$ and for any $m\in\Z$, $f_{2,n}(s_2,\lm)$ has a root $-(16\rho_{2,m})^{-1}\in G_{2,m}$.
By assumption, $(v,s)\in \Omega_{r,n}$ and hence $\sigma_{1,m}\in U_{1,m}$ $(m\not=n)$, $-(16\sigma_{2,m})^{-1} \in U_{2,m}$ $(m\in\Z)$. 
Since $\sigma_{1,m}$ $(m\not=n)$, $-(16\sigma_{2,m})^{-1}$ ($m\in\Z$) are the only roots of $f_n(v,s)$ and $U_{j,m}$ are pairwise disjoint 
it follows that $\rho_{1,m} = \sigma_{1,m}$ ($m\not=n$) and $\rho_{2,m} = \sigma_{2,m}$ ($m\in\Z$). By Lemma \ref{lem:kap8.2.14}, Corollary \ref{cor:kap8.2.19}, 
and the implicit function theorem, any given solution $(v^0,s^0)\in \Omega_{r,n}$ of $F^n(v,s)=0$ can be uniquely extended locally 
so that $s$ is given as a real analytic function of $v$. We claim that by the continuation method, this local solution can be extended 
along any path from $v^0$ to any given point in $\Hr$ since by Corollary \ref{cor:kap8.2.19}
\[
\partial_s F : \ell_{\hat n}^2 \times \ell^2 \to \ell_{\hat n}^2 \times \ell^2
\]
is a linear isomorphism at each point in $\Omega_{r,n}$. Indeed, let $(v^k,s(v^k))_{k\geq1}$ be any sequence in $\Omega_{r,n}$ with $F^n(v^k,s(v^k))=0$ for any $k\geq 1$ and $v =\lim_{k\to\infty} v^k$ in $\Hr$. 
By Proposition \ref{prop:spectralQuantCompact}, the endpoints of $G_{j,m}(v^k)$ converge to the endpoints of $G_{j,m}(v)$. 
As for any $j=1,2$, $m\in\Z$, and $k\geq 1$, $G_{j,m}(v^k)$ and $G_{j,m}(v)$ are compact intervals there exists a subsequence of $(v^k)_{k\geq1}$, 
which we again denote by $(v^k)_{k\geq 1}$ such that $(\sigma_{j,m}(v^k))_{k\geq 1}$ converges. Its limit, denoted by $\sigma_{j,m}(v)$,
 then satisfies $\sigma_{j,m}(v)\in G_{j,m}(v)$ and $F^n_{j,m}(v,s(v))=0$ where $s(v) = (s_1(v), s_2(v))$ and 
 $s_{1,k}(v) = \sigma_{1,k}(v) - k\pi$ ($k\not=n$), $s_{2,k}(v) = \sigma_{2,k}(v) - k\pi$ $(k\in\Z)$. 
 Hence $(v,s(v))\in \Omega_{r,n}$ and we can apply the implicit function theorem to $F^n$ at $(v,s(v))$. 
 This shows that the continuity method applies. Since $\Hr$ is simply connected, any particular solution $(v^0,s^0)\in \Omega_{r,n}$ of $F^n(v,s) =0$  thus extends uniquely 
 and globally to a real analytic map $s^n:\Hr\to \ell_{\hat n}^2\times \ell^2$ with graph in $\Omega_{r,n}$, satisfying $F^n(v,s^n(v))= 0$ everywhere. 
 At $v=0$, one verifies in a straightforward way with Cauchy's formula that a solution of $F^n(0,s^n)=0$ is given by
\[
 s_{1,k}^n : = \tau_{1,k} - k\pi \;(k\not=n), \qquad  s_{2,k}^n : = \tau_{2,k} - k\pi \; (k\in\Z).
\]
Clearly $(0,s^n)\in\Omega_{r,n}$. Note that this solution is also unique since $G_{j,m} = \{\tau_{j,m}(0)\}$ for any $j=1,2$ and $m\in\Z$. 
We thus have established \eqref{eq:kap8.2.42} and shown that there is exactly one such real analytic map.
\end{proof}

We now turn to the question of analytically extending the maps $s^n$ to a common complex neighborhood of $\Hr$ in \Hp.
\begin{lem}\label{lem:kap8.2.21}
The real analytic maps $s^n:\Hr \to \ell_{\hat n}^2\times \ell^2$, $n \ge 0$, of Proposition \eqref{eq:kap8.2.20} extend to a complex neighborhood $W\subset \hat W$ of $\Hr$
which is independent of $n$ so that for any potential $v \in\Hr$ and any $n \ge 0$, the restriction of the solution $s^n$ to $W\cap V_{v}$ satisfies $\sigma_{j,m}^n\in U_{j,m}$ (with $\sigma_{1,n}^n = \tau_{1,n}$) 
for any $m\in\Z$, $j=1,2$ where $(U_{j,m})_{m\in\Z, j=1,2}$ are isolating neighborhoods for $V_{v}$ (cf. \eqref{U_1m} - \eqref{U_2m}).
\end{lem}
\begin{proof}
Recall that by Proposition \ref{prop:kap8.2.20}, there exists for each $n\geq 0$ and $v\in\Hr$ a solution $s^n(v)$ so that $F^n(v,s^n(v))= 0$. 
Furthermore, by Corollary \ref{cor:kap8.2.19}, $Q^n = \partial_sF^n:\ell_{\hat n}^2\times\ell^2 \to \ell_{\hat n}^2 \times \ell^2$ at $(v,s^n(v))$ is invertible 
and hence by the implicit function theorem there exists a (simply connected) complex neighborhood $V_{v,n}\subset \Omega$ of $v$ 
so that the solution $s^n$ extends as a real analytic map to $V_{v,n}$.  We now prove estimates of the operator norm $\norm{Q^n(v,s)^{-1}}$ of the inverse $Q^n(v,s)^{-1}$ 
which will allow for any $v\in\Hr$ to chose $V_{v,n}$ independently of $n$. As a first step we prove in Lemma \ref{lem:kap8.2.21bis} below that on
\[
\Omega_{r,*} : = \set{(v,s)\in\Omega_r}{\sigma_{1,k}\in G_{1,k}(v) , \; (-16\sigma_{2,k})^{-1}\in G_{2,k}\; \forall k\in\Z}
\]
the operator norm $\norm{Q^n(v,s)^{-1}}$ of $Q^n(v,s)^{-1}$ is bounded uniformly in $n$, locally uniformly in $v\in \Hr$, 
and for any given $v\in\Hr$ uniformly in $s$ with $(v, s) \in \Omega_{r,*}$. As above, write $\sigma_{j,k} = k\pi + s_{j,k}$ for any $j= 1,2$ and $k\in\Z$. 
We now extend Lemma \ref{lem:kap8.2.21bis} to a complex neighborhood. By Lemma \ref{lem:kap8.2.14}, the maps
\[
F^n = (F_1^n, F_2^n) : \Omega \to \ell_{\hat n}^2 \times \ell^2 , (v,s) \mapsto F^n(v,s),\; s=(s_1,s_2)
\]
are real analytic, uniformly bounded with respect to $n$, and locally uniformly bounded on the (simply connected) neighborhood 
$\Omega$ (contained in $\hat W \times \ell^2\times\ell^2$). Then the maps
\[
Q^n = \partial_s F^n : \Omega \to \mathcal{L} (\ell_{\hat n}^2\times \ell^2)
\]
are real analytic as well. By Cauchy's estimate, $Q^n$ is uniformly bounded with respect to $n$ and  locally uniformly bounded on $\Omega$. Hence $\mathcal{L}(\ell_{\hat n}^2\times \ell^2)$ denotes the space of bounded linear operators on $\ell_{\hat n}^2\times \ell^2$. It then follows again by Cauchy's estimate that the variation $\dl  Q^n$ of $Q^n$ with respect to $v$ and $s$ can be kept as small as needed by restricting oneself to a sufficiently small neighborhood of any given point $(v,s)\in\Omega$. Representing $(Q+\dl Q)^{-1}$ by its Neumann series one obtains the standard estimate
$\norm{(Q + \dl Q)^{-1}} \leq 2 \norm{Q^{-1}}$ for any $\dl Q$ with $\norm{\dl Q}\leq \frac{1}{2\norm{ Q^{-1}}}$ for $Q=Q^n$ with $n\geq 0$. Hence if for a given $(v,s)\in\Omega$, $\norm{(Q^n)^{-1}}$ can be bounded uniformly in $n$, the same holds for elements in a sufficiently small neighborhood of $(v,s)$ in $\Omega$. By Lemma \ref{lem:kap8.2.21bis} below it then follows that $\sup_{n\geq0} \norm{(Q^n)^{-1}}$ can be bounded locally uniformly on a simply connected complex neighborhood of $\Omega_{r,*}$, contained in
\[
 \bigcup_{v_0\in\Hr} V_{v_0} \times \set{(s_1,s_2)\in \ell^2\times \ell^2}{\sigma_{1,k}\in U_{1,k}, \; (-16 \sigma_{2,k})^{-1} \in U_{2,k}\; \forall k\in\Z}.
\]
Using again the continuation method, the solution $s^n$ can then be extended by the implicit function theorem to a complex neighborhood $W\subset \hat W$ of $\Hr$ which is independent of $n\geq 0$ so that for any $v\in V_{v_0}\cap W$ with $v_0\in\Hr$, the sequences $(\sigma_{j,m}^n(v))_{m\in\Z}$ given by
\[
\sigma_{j,m}^n (v) = m\pi + s_{j,m}^n(v) \quad \forall j=1,2, \; \forall m\in\Z
\]
satisfy
\[
\sigma_{1,j}^n (v)\in U_{1,m}, \qquad (-16\sigma_{2,m}^n(v))^{-1} \in U_{2,m}\, , \qquad \forall m\in\Z.
\]
\end{proof}

The following lemma was used in the proof of Lemma \ref{lem:kap8.2.21}. Recall that
\[
\Omega_{r,*} = \set{(v,s)\in \Omega}{\sigma_{1,k} \in G_{1,k}(v), \; (-16\sigma_{2,k})^{-1} \in G_{2,k}(v) \; \forall k\in\Z}
\]
\begin{lem}\label{lem:kap8.2.21bis}
For any $(v,s)\in\Omega_{r,*}$ $\norm{(Q^n(v,s))^{-1}}$ is uniformly bounded with respect to $n$, locally uniformly with respect to $(v,s)$, and for any given $v\in\Hr$, uniformly with respect to $s$ with $(v,s)\in \Omega_{r,*}$.
\end{lem}
\begin{proof}
We begin by investigating the asymptotics of $Q^n(v,s)$ as $|n|\to \infty$ for $(v,s)\in \Omega_{r,*}$. To this end we consider the infinite matrices $Q_{11}^n,$ $Q_{12}^n$, $Q_{21}^n$, and $Q_{22}^n$ in the matrix representation of $Q^n$
\[
Q^n = \begin{pmatrix} Q_{11}^n & Q_{12}^n\\ Q_{21}^n & Q_{22}^n\end{pmatrix}
\]
individually. First let us consider $Q_{11,mr}^n = \frac{\partial F_{1,m}^n}{\partial s_{1,r}}$ $(m,r\not=n)$. By \eqref{eq:kap8.2.301} one has for $r\not=m$
\[
Q_{11,mr}^n = (n-m) \int_{\Gm[1,m]} \frac{\sigma_{1,m}-\lm}{ w_{1,m}(\lm)} \frac{1}{\sigma_{1.r}-\lm} \zeta_{1,m}^n (\lm) \dlm
\]
which by \eqref{eq:kap8.2.29.3} - \eqref{eq:kap8.2.29.4} can be written as
\[
Q_{11,mr}^n = \frac{\ii }{\pi} \int_{\Gm[1,m]} \frac{\sigma_{1,m}-\lm}{\sigma_{1,r}-\lm} \frac{n\pi -m\pi}{\sigma_{1,n}-\lm} \frac{\zeta_{1,m}(\lm)}{w_{1,m}(\lm)}\dlm
\]
where
\begin{equation*}
\zeta_{1,m}(\lm)  = \prod_{k\not=m} \frac{\sigma_{1,k} - \lm}{w_{1,k}(\lm)}  \frac{f_{2}(\lm) / f_{2}(\infty)}{\sqrt[c]{\chi_2(\lm)} / \sqrt[c]{\chi_2(\infty)}}, \quad \sigma_{1,n} \in U_{1,n}.
\end{equation*}
Here we have dropped the subindex $n$ in $f_{n,2}(\lm)$ and simply write $f_2(\lm) = \prod_{k\in\Z} \frac{\sigma_{2,k}+ \frac{1}{16\lm}}{\pi_k}$ and $f_2(\infty) = \prod_{k\in\Z} \frac{\sigma_{2,k}}{\pi_k}$. Since $\sigma_{1,n}\in U_{1,n}$ we have
\[
\frac{n\pi - m\pi}{\sigma_{1,n}-\lm} = 1 + \frac{(n\pi - \sigma_{1,n}) + (\lm-m\pi)}{\sigma_{1,n} - \lm} = 1+O(\frac{1}{n-m}).
\]
It implies that $Q_{11,mr}^* : = \lim_{n\to\infty} Q_{11,mr}^n$ exists and
\[
Q^*_{11,mr}(v,s) = \frac{\ii}{\pi} \int_{\Gm[1,m]} \frac{\sigma_{1,m}-\lm}{\sigma_{1,r}-\lm} \frac{\zeta_{1,m}(\lm)}{w_{1,m}(\lm)} \dlm.
\]
Similarly, for $m=r$ one has
\[
Q_{11,mm}^n = \frac{\ii}{\pi} \int_{\Gm[1,m]} \frac{n\pi - m\pi }{\sigma_{1,n}-\lm} \frac{\zeta_{1,m}(\lm)}{w_{1,m}(\lm)} \dlm.
\]
Hence $Q_{11,mm}^* = \lim_{n\to\infty} Q_{11,mm}^n$ exists and
\[
Q_{11,mm}^* = \frac{\ii}{\pi} \int_{\Gm[1,m]} \frac{\zeta_{1,m}(\lm)}{w_{1,m}(\lm)} \dlm.
\]
By inspection of $\zeta_{1,m}(\lm)$ one concludes that $Q_{11,mm}^*\not=0$ for any $m\in\Z$.
Next let us consider $Q_{12,mr}^n$. By \eqref{eq:kap8.2.302} for any $m,r\not=n$,
\[
Q_{12,mr}^n = (n-m) \int_{\Gm[1,m]} \frac{\sigma_{1,m}-\lm}{w_{1,m}(\lm) } \big(\frac{1}{\sigma_{2,r}+ \frac{1}{16\lm}} - \frac{1}{\sigma_{2,r}}\big) \zeta_{1,m}^n(\lm) \dlm
\]
which by \eqref{eq:kap8.2.29.3} - \eqref{eq:kap8.2.29.4} can be written as
\[
Q_{12,mr}^n = \frac{\ii}{\pi} \int_{\Gm[1,m]} \frac{\sigma_{1,m}-\lm}{w_{1,m}(\lm)} \frac{n\pi - m\pi }{\sigma_{1,n}-\lm} \big(\frac{1}{\sigma_{2,r}+ \frac{1}{16\lm}} - \frac{1}{\sigma_{2,r}}\big) \zeta_{1,m}(\lm) \dlm.
\]
As above we conclude that for any $m,r\in\Z$, $Q^*_{12,mr} : = \lim_{n\to\infty} Q_{12,mr}^n$ exists and
\[
Q_{12,mr}^* = \frac{\ii}{\pi} \int_{\Gm[1,m]} \frac{\sigma_{1,m}-\lm}{w_{1,m}(\lm)}  \big(\frac{1}{\sigma_{2,r}+ \frac{1}{16\lm}} - \frac{1}{\sigma_{2,r}}\big) \zeta_{1,m}(\lm) \dlm.
\]
Next let us consider $Q_{22,mr}$. For $m=r$, one has by \eqref{eq:kap8.2.303}
\[
Q_{22,mm}^n = 16\pi_m^2 \int_{\Gm[2,m]} \big(1- \frac{\sigma_{2,m}+ \frac{1}{16\lm}}{\sigma_{2,m}}\big) \frac{\zeta_{2,m}^n(\lm)}{w_{2,m}(\lm}\dlm
 = 16\pi_m^2 \int_{\tilde{\Gamma}_{2,m}} \frac{\mu}{\sigma_{2,m}} \frac{\zeta_{2,m}^n(-\frac{1}{16\mu})}{w_{2,m}(-\frac{1}{16\mu})} \frac{\dmu}{16\mu^2}
\]
which by \eqref{eq:kap8.2.26bis} can be written as
\[
Q_{22,mm}^n = \ii \int_{\tilde{\Gamma}_{2,m}} \frac{\mu}{\sigma_{2,m}} \frac{\pi_m^2}{\mu^2} \frac{\pi_n}{\sigma_{1,n} - \frac{1}{16\mu}} \frac{\zeta_{2,m}(-\frac{1}{16\mu})}{w_{2,m}(-\frac{1}{16\mu})} \dmu
\]
where
\[
\zeta_{2,m}(\lm) = \big(\prod_{k\in\Z} \frac{\sigma_{1,k}-\lm}{\pi_k}\big) \frac{1}{\sqrt[c]{\chi_1(\lm)} / \sqrt[c]{\chi_1(0)}} \frac{1}{f_2(\infty)} \prod_{k\not=m} \frac{\sigma_{2,k} + \frac{1}{16\lm}}{w_{2,k}(\lm)}.
\]
Since $\frac{\pi_n}{\sigma_{1,n} - \frac{1}{16\mu}} = 1+ O(\frac{1}n)$ it follows that $Q_{22,mm}^* := \lim_{n\to\infty} Q_{22,mm}^n$ exists and
\[
Q_{22,mm}^* = \ii \int_{\tilde{\Gamma}_{2,m}} \frac{\mu}{\sigma_{2,m}} \frac{\pi_m^2}{\mu^2} \frac{\zeta_{2,m}(-\frac{1}{16\mu})}{w_{2,m}(-\frac{1}{16\mu})} \dmu \not=0.
\]
Arguing as in the proof of Lemma \ref{lem:kap8.2.16} one sees that
\begin{equation}\label{eq:kap8.2.50}
Q_{22,mm}^* = 2\pi \prod_{k\in\Z} \frac{\sigma_{1,k}}{\sigma_{2,k}} + \ell_m^2.
\end{equation}
Similarly, for $m\not=r$ one has
\begin{align*}
Q_{22,mr}^n =& 16\pi_m^2 \int_{\Gm[2,m]} \frac{\sigma_{2,m} + \frac{1}{16\lm}}{w_{2,m}(\lm)} \big(\frac{1}{\sigma_{2,r} + \frac{1}{16\lm}} - \frac{1}{\sigma_{2,r}}\big) \zeta_{2,m}^n(\lm)\dlm
\\=& \ii \int_{\tilde{\Gamma}_{2,m}} \frac{\sigma_{2,m} - \mu}{w_{2,m}(-\frac{1}{16\mu})} \big(\frac{1}{\sigma_{2,r}-\mu} - \frac{1}{\sigma_{2,r}}\big) \frac{\pi_n}{\sigma_{1,n}-\frac{1}{16\mu}} \zeta_{2,m}(-\frac{1}{16\mu}) \frac{\pi_m^2}{\mu^2}\dmu.
\end{align*}
It implies that $Q_{22,mr}^* = \lim_{n\to\infty} Q_{22,mr}^n$ exists and
\[
Q_{22,mr}^* = \ii \int_{\tilde{\Gamma}_{2,m}} \frac{\sigma_{2,m} - \mu}{w_{2,m}(-\frac{1}{16\mu})} \big(\frac{1}{\sigma_{2,r}-\mu} - \frac{1}{\sigma_{2,r}}\big) \zeta_{2,m}(-\frac{1}{16\mu}) \frac{\pi_m^2}{\mu^2}\dmu.
\]
Finally, for $m,r\in\Z$ one has
\begin{align*}
Q_{21,mr}^n =& 16\pi_m^2 \int_{\Gm[2,m]} \frac{\sigma_{2,m} + \frac{1}{16\lm}}{w_{2,m}(\lm)} \frac{1}{\sigma_{1,r}+\lm} \zeta_{2,m}^n(\lm)\dlm
 \\=& \ii \int_{\tilde{\Gamma}_{2,m}} \frac{\sigma_{2,m}^n - \mu}{w_{2,m}(-\frac{1}{16\mu})} \frac{1}{\sigma_{1,r}-\frac{1}{16\mu}} \frac{\pi_n}{\sigma_{1,n}- \frac{1}{16\mu}} \zeta_{2,m}(-\frac{1}{16\mu}) \frac{\pi_m^2}{\mu^2}\dmu.
\end{align*}
Again, the limit $Q_{21,mr}^* := \lim_{n\to\infty} Q_{21,mr}^n$ exists and
\[
Q_{21,mr}^*   = \ii \int_{\tilde{\Gamma}_{2,m}} \frac{\sigma_{2,m}^n - \mu}{w_{2,m}(-\frac{1}{16\mu})} \frac{1}{\sigma_{1,r}-\frac{1}{16\mu}} \zeta_{2,m}(-\frac{1}{16\mu}) \frac{\pi_m^2}{\mu^2}\dmu.
\]
By Lemma \ref{lem:kap8.2.16} and its proof one sees that the coefficients $Q_{jj',mr}^*$ ($1\leq j,j'\leq 2,$ $m,r\in\Z$) satisfy the same asymptotic estimates as the ones for $Q_{jj',mr}^n$ of Lemma \ref{lem:kap8.2.16} except the one for $Q_{22,mm}^*$ which is given by \eqref{eq:kap8.2.50}. Hence
\[
Q^* = \begin{pmatrix} Q_{11}^* & Q_{12}^* \\ Q_{21}^* & Q_{22}^*\end{pmatrix} : \ell^2\times \ell^2 \to \ell^2\times \ell^2
\]
defines a bounded linear operator on $\ell^2\times \ell^2$. By a slight abuse of notation we now view $Q^n$ as an operator on $\ell^2\times \ell^2$ by setting $Q_{11,nn}^n = 2$ and $Q_{11,nm}^n = 0$ , $Q_{11,mn}^n=0$ for any $m\in\Z\setminus\{n\}$ as well as $Q_{12,nm}^n =0$ for any $m\in\Z$. We claim that $\lim_{n \to \infty}Q^n = Q^*$
 in operator norm, locally uniformly on $\Omega_{r,*}$. To see this, split $Q^*$ into its diagonal part $D^*$ and its off-diagonal part $K^*$, $Q^*= D^* + K^*$. Arguing as in the proof of Lemma \ref{lem:kap8.2.16}, one sees that the following holds
\begin{align*}
0\not= D_{11,m}^* :=& Q_{11,mm}^* = 2+\ell_m^2, \quad \forall m\in\Z\, ,
\\ 0\not= D_{22,m}^*:=& Q_{22,mm}^*= 2\pi \prod_{k\in\Z} \frac{\sigma_{1,k}}{\sigma_{2,k}} + \ell_m^2, \quad \forall m\in\Z \, ,
\\ K_{11,mr}^* = & Q_{11,mr}^* = \frac{\ell_m^2}{|m-r|}\, , \quad \forall m,r\in\Z, \; m\not=r \, ,
\\ K_{22,mr}^* = & Q_{22,mr}^* = \frac{\ell_m^2}{|m-r|} +\frac{\ell_m^2}{\langle r\rangle } \, , \quad \forall m,r\in\Z, \; m\not=r \, ,
\\ K_{12,mr}^* = & Q_{12,mr}^* = \frac{\ell_m^2}{\langle r\rangle^2} \, , \quad \forall m,r\in\Z \, ,
\\ K_{21,mr}^* = & Q_{21,mr}^* = \frac{\ell_m^2}{\langle r\rangle} \, , \quad \forall m,r\in\Z \, .
\end{align*}
We now show that $D^n \to D^*$ and $K^n\to K^*$ in operator norm. For any $h_1\in\ell^2$, taking into account that $D_{11,nn}^n =2$, one gets
\[
\norm{(D_{11}^* - D_{11}^n) h_1}^2 = |(Q_{11,nn}^*-2) h_{1,n}|^2 + \sum_{m\not=n} \big| \frac{1}{\ii\pi} \int_{\Gm[1,m]} \Big( 1- \frac{n\pi-m\pi}{\sigma_{1,n} - \lm}\Big) \frac{\zeta_{1,m}(\lm)}{w_{1,m}(\lm)} \dlm \, h_{1,m}\big|^2.
\]
Note that $|(Q_{11,nn}^* -2)h_{1,n}|^2 = \norm{h_1}\cdot \ell_n^1.$ Moreover writing
\[
1- \frac{n\pi-m\pi}{\sigma_{1,n}-\lm} = \frac{\sigma_{1,n}-n\pi}{\sigma_{1,n}-\lm} + \frac{\lm-m\pi}{\sigma_{1,n}-\lm}
\]
and using Lemma \ref{lem:kap8.2.15}, one sees that
\[
\sum_{m\not=n} \big|  \int_{\Gm[1,m]} \frac{\sigma_{1,n}-n\pi}{\sigma_{1,n}-\lm}\frac{\zeta_{1,m}(\lm)}{w_{1,m}(\lm)} \dlm \, h_{1,m}\big|^2 = O(|\sigma_{1,n}-n\pi|^2 \norm{h_1}^2).
\]
Using in addition that for $\lm\in G_{1,m}$, $|\lm-m\pi| \leq |\lm-\tau_{1,m}| + |\tau_{1,m} - m\pi|$ one gets
\[
\sum_{m\not=n} \big|  \int_{\Gm[1,m]} \frac{\lm-m\pi}{\sigma_{1,n}-\lm}\frac{\zeta_{1,m}(\lm)}{w_{1,m}(\lm)} \dlm \, h_{1,m}\big|^2 = O(\norm{h_1}^2 \big(\sup_{m\not=n} \frac{|\gm[1,m]| +|\tau_{1,m} - m\pi|^2}{|n-m|}\big)^2).
\]
which can be bounded by
\[
C\norm{h_1}^2 \Big( \sum_{|m-n|\leq \frac{|n|}{2}} (|\gm[1,m]|^2 + |\tau_{1,m}-m\pi|^2 ) + \frac{1}{n^2}\Big).
\]
Since $\sigma_{1,n} -n\pi = \ell_n^2$, $\gm[1,m] = \ell_m^2$, and $\tau_{1,m}-m\pi =\ell_m^2$ we conclude that in operator norm $\lim_{n\to\infty}\norm{D_{11}^* - D_{11}^n}=0.$ Similarly, for $h_2\in\ell^2$, one gets
\[
\norm{(D_{22}^*- D_{22}^n)h_2}^2 = \sum_{m\in\Z} \big| \int_{\tilde{\Gamma}_{2,m}} \frac{\mu}{\sigma_{2,m}} \frac{\pi_m^2}{\mu^2} \big(1-\frac{\pi_n}{\sigma_{1,n}-\frac{1}{16\mu}}\big) \frac{\zeta_{2,m}(-\frac{1}{16\mu})}{w_{2,m}(-\frac{1}{16\mu})} \dmu h_{2,m} \big|^2.
\]
Using again Lemma \ref{lem:kap8.2.15} and $\lim_{n\to\infty} (1-\frac{\pi_n}{\sigma_{1,n}}) = 0$ one sees that $\lim_{n\to\infty}\norm{D_{22}^* - D_{22}^n} = 0.$ Next we turn to $K_{11}^*$. For $h_1\in \ell^2,$
\[
\norm{(K_{11}^*- K_{11}^n)h_1}^2 = \sum_{m\not=n} \big| \sum_{r\not=m} (Q_{11,mr}^* - Q_{11,mr}^n) h_{1,r}\big|^2 + \big| \sum_{r\not=n}(Q_{11,nr}^* - Q_{11,nr}^n)h_{1,r}\big|^2.
\]
Since by definition $Q_{11,nr}^n=0$ for any $r\not=n$ the Cauchy-Schwarz inequality yields
\[
\big| \sum_{r\not=n}(Q_{11,nr}^* - Q_{11,nr}^n)h_{1,r}\big|^2 \leq \sum_{r\not=n} |Q_{11,nr}^*|^2 \norm{h_1}^2.
\]
Since for $n\not=r$
\[
Q_{11,nr}^* = \frac{\ii}{\pi} \int_{\Gm[1,n]} \frac{\sigma_{1,n}-\lm}{\sigma_{1,r} - \lm} \frac{\zeta_{1,n}(\lm)}{w_{1,n}(\lm)} \dlm
\]
one gets again by Lemma \ref{lem:kap8.2.15} that
\[
\sum_{r\not=n} |Q_{11,nr}^*|^2 = \big(\sum_{r\not=n}\frac{1}{|n-r|^2} \big) \cdot \ell_n^1.
\]
Similarly, since by definition $Q_{11,mn}^n = 0$ for $m\not=n$ one has
\[
\sum_{m\not=n} | \sum_{r\not=m} (Q_{11,mr}^* - Q_{11,mr}^n)h_{1,r}|^2 \leq I + II
\]
where
\[
I : =\sum_{m\not=n } |\sum_{r\not=m,n} \frac{1}{\pi} \int_{\Gm[1,m]} \frac{\sigma_{1,m}-\lm}{\sigma_{1,r}-\lm} \big(1- \frac{n\pi - m\pi}{\sigma_{1,n}-\lm}\big) \frac{\zeta_{1,m}(\lm)}{w_{1,m}(\lm)} \dlm \, h_{1,r}|^2 \, ,
\]
\[
II : = \sum_{m\not=n} | Q_{11,mn}^* h_{1,n}|^2 = \sum_{m\not=n} \big| \frac{1}{\pi} \int_{\Gm[1,m]} \frac{\sigma_{1,m}-\lm}{\sigma_{1,n} -\lm} \frac{\zeta_{1,m}(\lm)}{w_{1,m}(\lm)} \dlm \,h_{1,n}|^2.
\]
Let us begin by estimating the latter sum:
\[
II \leq C \sum_{m\not=n} \frac{|\gm[1,m]|^2}{|n-m|^2} \norm{h_1}^2 \leq C\sum_{|m|\geq |n|/2} |\gm[1,m]|^2 \norm{h_1}^2 + C\frac{1}{n^2} \norm{\gm[1]}^2 \norm{h_1}^2.
\]
The sum $I$ is estimated similarly:
\begin{align*}
I\leq &\sum_{m\not=n} \sum_{r\not=m,n} \big(\sup_{\lm\in G_{1,m}} |\frac{\sigma_{1,m}-\lm}{\sigma_{1,r}-\lm}| \frac{|n\pi - \sigma_{1,n}| + |\lm-m\pi|}{|\sigma_{1,n}-\lm|} |\zeta_{1,m}(\lm)|\big)^2 \norm{h_1}^2
\\ \leq & C\norm{h_1}^2 \sum_{m\not=n} |\gm[1,m]|^2 \frac{|n\pi -\sigma_{1,n}|^2 + |\tau_{1,m} -m\pi|^2 + |\gm[1,m]|^2}{|n-m|^2} \sum_{r\not=m,n} \frac{1}{|r-m|^2}
\end{align*}
leading to an estimate of the same kind as for the sum $II$. As a result we conclude that in the $\ell^2$-operator norm
\[
\lim_{n\to\infty} \norm{K_{11}^*- K_{11}^n} = 0.
\]
In the same way one shows that
\[
\lim_{n\to\infty} \norm{K_{22}^*- K_{22}^n} = 0, \qquad \lim_{n\to\infty} \norm{K_{12}^*- K_{12}^n} = 0,
\qquad \lim_{n\to\infty} \norm{K_{21}^*- K_{21}^n} = 0.
\]
Hence we have established that
\[
\lim_{n\to\infty} \norm{D^* - D^n} = 0, \qquad \lim_{n\to\infty}\norm{K^*- K^n} =0.
\]
Going trough the arguments of the proof one verifies that the convergence is locally uniform in $\Omega_{r,*}$. For any $n\geq 0$, $Q^n$ is a continuous map 
$$
Q^n: \Omega_{r,*} \to \mathcal{L}(\ell^2\times \ell^2).
$$ 
By the locally uniform convergence $Q^n \to Q^*$ for $n\to\infty$, it then follows that $Q^*:\Omega_{r,*} \to \mathcal{L}(\ell^2\times \ell^2)$ is continuous as well. 
Arguing as in the proof of Lemma \ref{lem:kap8.2.18}, $Q^*$ is boundedly invertible at every point of $\Omega_{r,*}$. Since or any given $v\in\Hr$, the set
\[
\Pi (v) := \set{s=(s_1,s_2) \in\ell^2_\R \times \ell_\R^2}{\sigma_{1,k}\in G_{1,k}(v),\; (-16\sigma_{2,k})^{-1} \in G_{2,k}(v)\;\forall k\in\Z}
\]
is compact in $\ell_\R^2\times \ell_\R^2$, the operator $Q^*(v,s)$ is indeed uniformly boundedly invertible for any $s\in \Pi(v)$. 
By continuity, $Q^n(v,s)$ is also uniformly boundedly invertible for all large $n$ and $s\in\Pi(v)$, 
and hence by Corollary \ref{cor:kap8.2.19}, for all $n\geq 0$. This finishes the proof of Lemma \ref{lem:kap8.2.21bis}.
\end{proof}
Summarizing our results so far, the functions $\psi_n(\lm) =-\frac{1}{\pi_n} \frac{1}{\psi_{n,2}(\infty)} \psi_{n,1}(\lm) \psi_{n,2}(\lm)$ with
\[
\psi_{n,1}(\lm) = \prod_{k\not=n} \frac{\sigma_{1,k}^n - \lm}{\pi_k}, \qquad 
\psi_{n,2}(\lm) = \prod_{k\in\Z} \frac{\sigma_{2,k}^n + \frac{1}{16\lm}}{\pi_k}
\]
satisfy
\begin{equation} \label{eq:kap8.2.70}
\int_{\Gm[1,m]} \frac{\psi_{n}(\lm)}{\sqrt[c]{\chi_p(\lm)}} \dlm =0, \quad m\not=n, \qquad 
\int_{\Gm[2,m]} \frac{\psi_n(\lm)}{\sqrt[c]{\chi_p(\lm)}} \dlm =0, \quad \forall m\in\Z.
\end{equation}
We now check that they also satisfy the normalization condition
\[
\frac{1}{2\pi} \int_{\Gm[1,n]} \frac{\psi_n(\lm)}{\sqrt[c]{\chi_p(\lm)}} \dlm = 1.
\]
\begin{lem}
On the complex neighborhood $W\subset \hat W$ of $\Hr$ of Lemma \ref{lem:kap8.2.21} one has for any $n\geq0$
\[
\int_{\Gm[1,n]} \frac{\psi_n(\lm)}{\sqrt[c]{\chi_p(\lm)}} \dlm = 2\pi.
\]
\end{lem}
\begin{proof}
Let $v\in W$ and $n\geq 0$ be arbitrary. By \eqref{eq:kap8.2.70} it follows by Cauchy's theorem that for $N\geq1$ sufficiently large so that $U_{j,m}\cap \partial B_N =\emptyset,$ $U_{j,m} \cap \partial B_{-N} =\emptyset$ $\forall m\in\Z$, $k=1,2$, one has
\begin{equation}\label{eq:kap8.2.71}
\int_{\Gm[1,n]} \frac{\psi_n(\lm)}{\sqrt[c]{\chi_p(\lm)}} \dlm = I_N - II_N
\end{equation}
where
\[
I_N : = \int_{\partial B_N} \frac{\psi_n(\lm)}{\sqrt[c]{\chi_p(\lm)}} \dlm,
\qquad  II_N : = \int_{\partial B_{-N}} \frac{\psi_n(\lm)}{\sqrt[c]{\chi_p(\lm)}} \dlm.
\]
Let us first compute $I_N$. Using that $\sqrt[c]{\chi_1(0)} = \sqrt[c]{\chi_2(\infty)}$, one has $\sqrt[c]{\chi_p(\lm)} = \ii \sqrt[c]{\chi_1(\lm)} \sqrt[c]{\chi_2(\lm)}/ \sqrt[c]{\chi_2(\infty)}$. Letting $\sigma_{1,n}^n : = \tau_{1,n}$, the contour integral $I_N$ can be written as
\[
I_N = \frac{1}{\ii} \int_{\partial B_N} \frac{1}{\lm - \sigma_{1,n}^n} \big(\prod_{k\in\Z} \frac{\sigma^n_{1,k}-\lm}{ w_{1,k}(\lm)}\big) \frac{\psi_{n,2}(\lm) / \psi_{n,2}(\infty)}{\sqrt[c]{\chi_2(\lm)}/\sqrt[c]{\chi_2(\infty)} } \dlm.
\]
Note that as $N\to\infty$
\[
\sup_{\lm\in\partial B_N} |\psi_{n,2}(\lm) / \psi_{n,2}(\infty) - 1 | = O(\frac{1}{N}),
\qquad \sup_{\lm\in\partial B_N} |\sqrt[c]{\chi_2(\lm)}/ \sqrt[c]{\chi_2(\infty)} - 1 | = O(\frac{1}{N}).
\]
Furthermore, by \cite[Lemma C.5]{nlsbook},
\[
\sup_{\lm\in\partial B_N} |\prod_{k\in\Z} \frac{\sigma_{1,k}^n- \lm}{w_{1,k}(\lm)} - 1 | = o(1) \quad \text{as} \; N\to\infty.
\]
Hence we get by Cauchy's theorem
\begin{equation}\label{eq:kap8.2.72}
\lim_{N\to\infty} I_N = \lim_{N\to\infty} \frac{1}{\ii} \int_{\partial B_N} \frac{1}{\lm-\sigma_{1,n}^n} (1+o(1)) \dlm = 2\pi.
\end{equation}
Now let us turn towards $II_N$. By the change of variable of integration $\lm=-\frac{1}{16\mu}$,
\[
II_N = \int_{\partial B_{-N}}  \frac{\psi_n(\lm)}{\sqrt[c]{\chi_p(\lm)} }\dlm
= \int_{\partial B_{N}} \frac{\ii }{\pi_n} \frac{\psi_{n,1} (-\frac{1}{16\mu}) / \psi_{n,2}(\infty)}{\sqrt[c]{\chi_1(-\frac{1}{16\mu})} / \sqrt[c]{\chi_1(0)}} \prod_{k\in\Z} \frac{\sigma_{2,k}^n - \mu}{ w_{2,k}(-\frac{1}{16\mu})} \frac{\dmu}{16\mu^2}.
\]
By the definition of $w_{2,k}$, one has $w_{2,k}(-\frac{1}{16\mu}) = \sqrt[s]{(\lm_{2,k}^+-\mu)(\lm_{2,k}^--\mu)}$. Hence again 
by \cite[Lemma C.5]{nlsbook} one has
\[
\sup_{\mu\in\partial B_N} \big| \prod_{k\in\Z} \frac{\sigma_{2,k}^n - \mu}{w_{2,k}(-\frac{1}{16\mu})} -1\big| = o(1) \quad \text{as}\; N\to\infty.
\]
Furthermore, as $N\to\infty$,
\[
\sup_{\mu\in\partial B_N} |\psi_{n,1}(-\frac{1}{16\mu})- \psi_{n,1}(0)| = O(\frac{1}{N}),
\qquad \sup_{\mu\in\partial B_N} | \frac{\sqrt[c]{\chi_1(-\frac{1}{16\mu})}}{\sqrt[c]{\chi_1(0)}} -1 | = O(\frac{1}{N}).
\]
Altogether, one then concludes that
\begin{equation}\label{eq:kap8.2.73}
\lim_{N\to\infty} II_N = \lim_{N\to\infty} \int_{\partial B_N} \frac{\ii }{\pi_n} \frac{\psi_{n,1}(0)}{ \psi_{n,2}(\infty)} (1+o(1)) \frac{\dmu}{16\mu^2} =0.
\end{equation}
Combining \eqref{eq:kap8.2.72} - \eqref{eq:kap8.2.73} with \eqref{eq:kap8.2.71} yields the claimed identity.
\end{proof}
To finish the proof of Theorem \ref{thm:kap8.2.12} it remains to establish the claimed estimates of $\sigma_{1,m}^n,$ $\sigma_{2,m}^n$. Recall that for convenience we set $\sigma_{1,n}^n=\tau_{1,n}$.
\begin{lem}\label{lem:kap8.2.23}
On the complex neighborhood $W\subset \hat W$ of $\Hr$ of Lemma \ref{lem:kap8.2.21} one has for $j=1,2$
\begin{equation}\label{eq:kap8.2.74}
\sigma_{j,m}^n = \tau_{j,m} + \gm[j,m]^2 \ell_m^2
\end{equation}
uniformly in $n\geq 0$ and locally uniformly on $W$. It means that $|\sigma_{j,m}^n - \tau_{j,m} | \leq |\gm[j,m]|^2 a_{j,m}^n$ where $a_{j,m}^n\geq 0$ and $\sum_m (|a_{1,m}^n|^2 + |a_{2,m}^n|^2)\leq C$. The constant $C$ can be chosen uniformly in $n$ and locally uniformly on $W$.
\end{lem}
\begin{proof}
For $v\in W$ and  let $(U_{j,m})_{m\in \Z, j=1,2}$ be a set of isolating neighborhoods which work for a complex neighborhood of $v$ in $\hat W$ (cf. \eqref{U_1m} - \eqref{U_2m}).
Let $n\geq0$ be arbitrary. First we treat the case $j=1$. Multiplying the identity, $F_{1,m}^n(v,s^n(v)) =0$ by $\pi$ it can be written as
\begin{equation}\label{eq:kap8.2.75}
  0 = \int_{\Gm[1,m]} \frac{\sigma_{1,m}^n - \lm}{w_{1,m}(\lm)} \chi_{1,m}^n(\lm) \dlm
\end{equation}
where $\chi_{1,m}^n(\lm) = \frac{\pi(n-m)}{\sigma_{1,n}^n -\lm} \zeta_{1,m}^n(\lm)$ with $\sigma_{1,n}^n = \tau_{1,n}$ and (cf \eqref{eq:kap8.2.29.3} - \eqref{eq:kap8.2.29.3bis})
\[
 \zeta_{1,m}^n : = \ii \Big( \prod_{k\not=m} \frac{\sigma_{1,k}^n - \lm}{w_{1,k}(\lm)}\Big) \frac{\psi_{n,2}(\lm)/\psi_{n,2}(\infty)}{\sqrt[c]{\chi_2(\lm)}/\sqrt[c]{\chi_2(\infty)}}.
\]
Expanding $\chi_{1,m}^n$ at $\lm= \tau_{1,m}$ up to first order, one has
\[
 \chi_{1,m}^n(\lm) = \chi_{1,m}^n(\tau_{1,m})   + (\lm - \tau_{1,m})b_{1,m}^n(\lm).
\]
Since $\int_{\Gm[1,m]} \frac{\tau_{1,m}-\lm}{w_{1,m}(\lm)} \dlm = 0$ and $\int_{\Gm[1,m]}\frac{1}{w_{1,m}(\lm)}\dlm = -2\pi\ii$
one has
\[
\frac{1}{2\pi \ii} \int_{\Gm[1,m]} \frac{\sigma_{1,m}^n - \lm}{w_{1,m}(\lm)} \chi_{1,m}^n(\tau_{1,m}) \dlm = -(\sigma_{1,m}^n - \tau_{1,m}) \chi_{1,m}^n(\tau_{1,m})
\]
and hence \eqref{eq:kap8.2.75} becomes
\begin{equation}\label{eq:kap8.2.76}
\chi_{1,m}^n(\tau_{1,m}) (\sigma_{1,m}^n -\tau_{1,m}) = \frac{1}{2\pi\ii} \int_{\Gm[1,m]}\frac{\sigma_{1,m}^n-\lm}{w_{1,m}(\lm)} (\lm - \tau_{1,m}) b_{1,m}^n(\lm) \dlm.
\end{equation}
Note that by \eqref{eq:kap8.2.29.4}, $\sup_{\lm\in U_{1,m}} |\zeta_{1,m}(\lm) - 1| = \ell_m^2$ and 
by the definition of $U_{1,n}$, $\sigma_{1,n}^n = \tau_{1,n} \in U_{1,n}$, implying that for any $m \ne n$, $\sup_{\lm\in U_{1,m}} \big|\frac{\pi(n-m)}{\sigma_{1,n}^n-\lm} -1\big| = \ell_m^2$. 
Hence by the definition of $\chi_{1,m}^n(\lm)$
\[
\sup_{\lm\in U_{1,m}} \big| \chi_{1,m}^n(\lm)-\ii \big| = \ell_m^2, 
\qquad \inf_{m\not=n} \inf_{\lm \in U_{1,m}}|\chi_{1,m}^n(\lm)|>0 .
\]
By choosing $\Gm[1,m]$ so that $\inf_m \mathrm{dist}(\Gm[1,m],\partial U_{1,m})>0$ one then gets by Cauchy's estimate that
\[
\sum_{m\not=n} \sup_{\lm\in \Gm[1,m]} |b_{1,m}^n(\lm)|^2 \leq C
\]
for some $C>0$. It then follows from Lemma \ref{lem:kap8.2.15} applied to \eqref{eq:kap8.2.76} that for $m\not=n$
\[
|\sigma_{1,m}^n - \tau_{1,m}| = \max_{\lm\in G_{1,m}} |\sigma_{1,m}^n-\lm| |\gm[1,m]|\ell_m^2.
\]
 Since for any $\lm\in G_{1,m}$, $|\sigma_{1,m}^n-\lm| \leq |\sigma_{1,m}^n - \tau_{1,m}| + |\gm[1,m]/2|$ one gets
\begin{equation}\label{eq:kap8.2.77}
  |\sigma_{1,m}^n -\tau_{1,m}| = \big(|\sigma_{1,m}^n -\tau_{1,m}| + |\gm[1,m]/2| \big)|\gm[1,m]|\ell_m^2
\end{equation}
implying that $|\sigma_{1,m}^n - \tau_{1,m}| = |\gm[1,m]|\ell_m^2.$ 
Substituting the latter estimate into the right hand side of \eqref{eq:kap8.2.77} one obtains $|\sigma_{1,m}^n - \tau_{1,m}| = |\gm[1,m]|^2 a_{1,m}^n$ 
for some $a^n_{1,m}\geq0$ (and $a_{1,n}^n=0$) with $\sum_m |a_{1,m}^n|^2 \leq C.$ Going trough the arguments of the proof one verifies that $C>0$ can be chosen uniformly in $n\geq0$ and locally uniformly on $V_v\subset W$. The case $j=2$ is treated in a similar fashion.
\end{proof}

It remains to prove Corollary \ref{cor:kap8.2.12}.
\begin{proof}[Proof of Corollary \ref{cor:kap8.2.12}]
Let $v = (q,p)$ be in $W$. We restrict ourselves to
  consider the case $j=2$, $m\in\Z$,  since the case $j=1$, $m\in\Z$, is dealt with in a similar way. 
  By the change of variable $\mu = \frac{1}{16\lm}$ one has for any $n \ge 1,$
  \[
   \int_{\Gm[2,m](q,p)} \frac{\psi_{-n}(\lm,v)}{\sqrt[c]{\chi_p(\lm,v)}}\dlm
   = \int_{\Gm[2,m](q,p)} \frac{\psi_n(\frac{1}{16\lm},-q,p)}{\sqrt[c]{\chi_p(\lm,q,p)}}\frac{\dlm}{16\lm^2}
   = \int_{\Gm[1,-m](-q,p)} \frac{\psi_n(\mu,-q,p)}{\sqrt[c]{\chi_p(\frac{1}{16\mu},q,p)}}(-1)\dmu.
  \]
  Since $\sqrt[c]{\chi_p(\frac{1}{16\mu},q,p)} = -\sqrt[c]{\chi_p(\mu,-q,p)}$ (cf. \eqref{eq:kap6.44}) it then follows that
   \[
   \int_{\Gm[2,-m](q,p)} \frac{\psi_{-n}(\lm,v)}{\sqrt[c]{\chi_p(\lm,v)}} \dlm
   = \int_{\Gm[1,-m](-q,p)} \frac{\psi_n(\mu,-q,p)}{\sqrt[c]{\chi_p(\mu,-q,p)}}\dmu = 2\pi \dl_{-n,m}.
  \]
  All other claims follow from Theorem \ref{thm:kap8.2.12}.
  \end{proof}


\appendix
\section{Appendix A: Liouville's theorem on $\C^*$}\label{Liouville}
\setcounter{thm}{0}
\setcounter{equation}{0}
For the convenience of the reader, we state and prove the standard version of Liouville's theorem
on the punctured complex plane $\C^*$, used in several proofs of this paper.

Recall that $B_n$, $n\in\Z$, denote the discs given by 
$B_0 = \set{\lm}{|\lm| < \pi/2}$ and (cf. \eqref{definition Bn})
\[
 B_n = \set{\lm}{|\lm| < n\pi+\pi/2},\quad B_{-n} = \set{\lm}{|\lm| \leq \frac{1}{16(n\pi +\pi/2)}} \, , \qquad \forall n \geq 1\, .
\]

\begin{lem}\label{ap:lem:liouvilleCn}
Let $f$ be an analytic function on $\C^*$. If $f$ is uniformly bounded on the circles $\partial B_n$ and $\partial B_{-n}$ for any $n \in \N$, then $f$ is constant.
\end{lem}
\begin{proof}
First note that if $f$ is bounded by $M>0$ on all these circles then so is $g$, defined by $g(\lm) = f\pfrac{1}{16\,\lm}$. Furthermore if the Laurent series of $f$ around $0$ is given by
\[
 f(\lm) = \sum_{k\in\Z} a_k \lm^k
\]
then the Laurent series of $g$ is given by
\[
 g(\lm) = \sum_{k\in\Z} a_{-k}(16  \lm)^k.
\]
Hence by Cauchy's inequality $|a_k| \leq r^{-k} \sup_{|z|=r} |f(z)|$ for any $r>0$. Choosing $r=n\pi+\pi/2$ we can bound $|f(z)|$ by $M$ and hence $a_k=0$ for any $k\geq 1$. Using $g$, the same argument shows that $a_{-k}=0$ for any $k\geq 1$. Hence $f \equiv a_0$ is constant.
\end{proof}

\section{Appendix B: Interpolation}\label{interpolation}
\setcounter{thm}{0}
\setcounter{equation}{0}
In this appendix we present an interpolation lemma which is used for the construction of the  functions $\psi_n$ in Theorem \ref{thm:kap8.2.12}. 

Recall that for any $N\geq 1$, we denote by $A_N =B_N\setminus B_{-N}$ the annulus, centered at $0$, 
with boundary $\partial A_N = \partial B_N \cup \partial B_{-N}$ where $B_N = \set{\lm}{|\lm| < n\pi +\pi/2}$  
and $B_{-N} = \set{\lm}{|\lm|\leq (16(N\pi+\pi/2))^{-1}}$ (cf. \eqref{definition Bn}). 
Assume that $(\sigma_{1,n})_{n\in\Z}$, $(\sigma_{2,n})_{n\in\Z}$ 
are sequences in $\C^*$ which together with $\kappa_{2,n}:= (-16\sigma_{2,n})^{-1}$
have the following properties,
\begin{equation}\label{eq:D1}
\sigma_{1,n} \not=\sigma_{1,m} \;(n\not=m), \quad \kappa_{2,n} \not=\kappa_{2,m}\;(m\not=n), \quad \sigma_{1,n}\not=\kappa_{2,m} \, , \qquad \forall n,m \in \Z \, ,
\end{equation}
and
\begin{equation}\label{eq:D2}
\sigma_{1,n},\sigma_{2,n} = n\pi + \ell_n^2 \quad\text{as}\; |n| \to \infty.
\end{equation}
Let
\begin{equation}\label{eq:D3}
f(\lm) := f_1(\lm) f_2(\lm)
\end{equation}
where
\begin{equation}\label{eq:D4}
f_1(\lm) : = \prod_{n\in\Z} \frac{\sigma_{1,n}-\lm}{\pi_n} ,\qquad f_2(\lm) : = \prod_{n\in\Z} \frac{\sigma_{2,n} + \frac{1}{16\lm}}{\pi_n}.
\end{equation}
According to \cite[Lemma C.1]{nlsbook}, the infinite products $f_1$ and $f_2$ define analytic functions on $\C^*$ with roots $\sigma_{1,n}$, $n\in\Z$, and respectively, $\kappa_{2,n}$, $n\in\Z$. 
Note that $f_1(\lm)$ is analytic at $0$ with $ f_1(0)\not=0$ and $f_2(\lm)$ is analytic at $\infty$ with $f_2(\infty)\not=0$,
\[
f_1(0) = \prod_{n\in\Z} \frac{\sigma_{1,n}}{\pi_n},\qquad f_2(\infty) = \prod_{n\in\Z} \frac{\sigma_{2,n}}{\pi_n}.
\]
Furthermore, by \cite[Lemma C.4]{nlsbook},
\begin{equation} \label{eq:D5}
\sup_{\lm\in\partial B_N} \big| \frac{f_1(\lm)}{\sin(\lm)} + 1\big| = o(1) \quad \text{as} \; N\to\infty
\end{equation}
and
\begin{equation} \label{eq:D6}
\sup_{\lm\in\partial B_{-N}} \big| \frac{f_2(\lm)}{\sin(-\frac{1}{16\lm})} +1\big| =
\sup_{\mu\in\partial B_{N}} \big| \frac{f_2(-\frac{1}{16\mu})}{\sin(\mu)} + 1\big| =
o(1) \quad \text{as} \; N\to\infty \, .
\end{equation}
\begin{lem}\label{ap:lem:interpolation}
Assume that $\ph:\C^*\to \C$ is analytic and that $(\sigma_{1,n})_{n\in\Z}$, $(\sigma_{2,n})_{n\in\Z}$ are sequences in $\C^*$ so that  \eqref{eq:D1}-\eqref{eq:D2} are satisfied. 
If as $N\to\infty$
\begin{equation}\label{eq:D8}
 \sup_{\lm\in\partial B_N} |\frac{\ph(\lm)}{\sin(\lm)}| = o(1), \qquad 
 \sup_{\lm\in\partial B_{-N}} |\frac{\ph(\lm) }{\sin(-\frac{1}{16\lm})} |\;\; \Big(\,=\sup_{\mu\in\partial B_{N}} |\frac{\ph(-\frac{1}{16\mu})}{\sin(\mu)} | \Big)  = o(N) \, , 
\end{equation}
then for any $z\in\C^*$
\[
\ph(z) = \sum_{n\in\Z} \big( \frac{\ph(\sigma_{1,n})}{\dot f(\sigma_{1,n})} \frac{f(z)}{z-\sigma_{1,n}} + \frac{\ph(\kappa_{2,n})}{\dot f(\kappa_{2,n})} \frac{f(z)}{z - \kappa_{2,n}}\big)
\]
where $f$ is the function defined in \eqref{eq:D3}, $\dot f(\lm) =\partial_\lm f(\lm)$,  and $\kappa_{2,n} = -(16\sigma_{2,n})^{-1}$ for any $n\in\Z$. In particular, the latter sum converges.
\end{lem}
\begin{proof}
Consider for any $z\in\C^*\setminus\set{\sigma_{1,n},\kappa_{2,n}}{n\in\Z}$ the function
\[
g(\lm) : = \frac{\ph(\lm)}{(\lm-z)f(\lm)}
\]
where $f(\lm)$ is given by \eqref{eq:D3}. Then $g(\lm)$ is a meromorphic function on $\C^*$ with poles in $\sigma_{1,n}$ and $\kappa_{2,n}$,  $n\in\Z$, as well as $z$. Chose $N\geq 1$ so large that,
\[
z\in A_N,\quad
\sigma_{1,n}, \kappa_{2,n}\in A_N \;\forall |n|\leq N,\quad \sigma_{1,n},\kappa_{2,n} \in \C^*\setminus A_N \; \forall |n|\geq N+1.
\]
By the residue theorem it then follows that
\begin{align*}
\frac{1}{2\pi\ii } \int_{\partial A_N} g(\lm)\dlm =& \Res_z g + \sum_{|n|\leq N} \Res_{\sigma_{1,n}} g + \sum_{|n|\leq N} \Res_{\kappa_{2,n}} g
\\=& \frac{\ph(z)}{f(z)} + \sum_{|n|\leq N} \frac{\ph(\sigma_{1,n})}{\dot f(\sigma_{1,n})} \frac{1}{\sigma_{1,n}-z}  + \sum_{|n|\leq N} \frac{\ph(\kappa_{2,n})}{\dot f (\kappa_{2,n})} \frac{1}{\kappa_{2,n}-z} \, .
\end{align*}
On the other hand, write $\frac{1}{2\pi \ii } \int_{\partial A_N} g(\lm)\dlm = \frac{1}{2\pi\ii} I_N- \frac{1}{2\pi\ii} II_N$ where
\[
 I_N = \int_{\partial B_N} \frac{\ph(\lm)}{f(\lm)} \frac{1}{\lm-z} \dlm, \qquad II_N = \int_{\partial B_{-N}} \frac{\ph(\lm)}{f(\lm)} \frac{1}{\lm-z} \dlm.
\]
It follows from \eqref{eq:D5} and \eqref{eq:D8} that
\[
\sup_{\lm\in \partial B_N} \Big| \frac{\ph(\lm)}{f(\lm)}\Big| = o(1) \quad \text{as} \; N\to\infty \, ,
\]
implying that $\lim_{N\to\infty} I_N = 0.$ 
Similarly, it follows from \eqref{eq:D6}--\eqref{eq:D8} that
\[
\sup_{\lm\in\partial B_{-N}} \Big|\frac{\ph(\lm)}{f(\lm)} \Big| = \sup_{\mu\in\partial B_{N}} \Big| \frac{\ph(-\frac{1}{16\mu})}{f(-\frac{1}{16\mu})} \Big| =o(N) \quad \text{as} \; N\to\infty
\]
yielding that
\[
\int_{\partial B_{-N}} \frac{\ph(\lm)}{f(\lm)} \frac{1}{\lm-z} \dlm = \int_{\partial B_N} \frac{\ph(-\frac{1}{16\mu})}{f(-\frac{1}{16\mu})} \frac{1}{-\frac{1}{16\mu}-z} \frac{d\mu}{16\mu^2} \to 0 
\qquad \text{as}\; N\to \infty.
\]
Combining the results obtained one gets
\[
0 = \frac{\ph(z)}{f(z)} + \sum_{n\in\Z} \big( \frac{\ph(\sigma_{1,n})}{\dot f(\sigma_{1,n})} \frac{1}{\sigma_{1,n}-z} + \frac{\ph(\kappa_{2,n})}{\dot f(\kappa_{2,n})} \frac{1}{\kappa_{2,n}-z}\big)
\]
and hence
\begin{equation}\label{eq:kapD.8}
\ph(z) = \sum_{n\in\Z} \big( \frac{\ph(\sigma_{1,n})}{\dot f(\sigma_{1,n})} \frac{f(z)}{z-\sigma_{1,n}} + \frac{\ph(\kappa_{2,n})}{\dot f(\kappa_{2,n})} \frac{f(z)}{z - \kappa_{2,n}}\big) .
\end{equation}
Since $\frac{f(z)}{z-\sigma_{1,n}}$ has a removable singularity at $z=\sigma_{1,n}$ and $\frac{f(z)}{z-\kappa_{2,n}}$ has one at $z=\kappa_{2,n}$, one infers 
that the identity \eqref{eq:kapD.8} actually holds for any $z\in\C^*$.
\end{proof}


  \bibliographystyle{acm}

 \end{document}